\documentclass[12pt]{amsart}
\usepackage{amsmath,amssymb,latexsym,color,epsfig,enumerate,a4,dsfont,setspace,geometry,mathtools,amstext,amsthm,babel,graphicx,caption,mwe,amssymb,physics,enumitem,nicefrac}

\usepackage[all]{xy}

\usepackage{enumitem}
\usepackage{float}
\usepackage{comment}
\usepackage{tikz,tikz-cd,pgf}

\usepackage{fullpage}

\usepackage{geometry}

\usepackage[bookmarksopen,bookmarksdepth=3]{hyperref}

\newtheorem {Theorem}                    {Theorem}
\newtheorem{mainTheorem}{Theorem}

\newtheorem {Proposition}[Theorem]       {Proposition}
\newtheorem {Lemma}      [Theorem]       {Lemma}

\theoremstyle{definition}
\newtheorem  {Definition} [Theorem]{Definition}

\theoremstyle{remark}
\newtheorem{Remark}[Theorem]	{Remark}
\newtheorem*{Remark-nn}		{Remark}    

\newtheorem{Notation}[Theorem]	{Notation}

\numberwithin{equation}{section}

\newcommand{\RR}{\mathbb{R}}
\newcommand{\CC}{\mathbb{C}}

\newcommand{\ZZ}{\mathbb{Z}}

\newcommand{\GW}{\text{GW}}
\newcommand{\OGW}{\text{OGW}}
\newcommand{\OGWbar}{\text{OGW}}

\newcommand{\cupdot}{\mathbin{\mathaccent\cdot\cup}}
\newcommand{\invecZN}[1]{\vec{#1}\in\left(\ZZ_{\geq 0}\right)^{N+1}}
\newcommand{\invecZK}[1]{\vec{#1}\in\left(\ZZ_{\geq 0}\right)^{K+1}}
\newcommand{\sumsupto}[2]{#1\, \vdash\, #2}

\newcommand{\parentheses}[1]{(#1)}

\allowdisplaybreaks[1]

\begin{document}

	\title{WDVV-based recursion for open Gromov-Witten invariants}
	
	\author{Roi Blumberg}
	\address{School of Mathematical Sciences\\ Tel Aviv University\\Tel Aviv, 6997801, Israel}	
	\email{roiblumberg@mail.tau.ac.il}
	
	\author[S. Tukachinsky]{Sara B. Tukachinsky}
	\address{School of Mathematical Sciences\\ Tel Aviv University\\Tel Aviv, 6997801, Israel}\email{sarabt1@gmail.com}
	
	\date{December 2025}
	
	\begin{abstract}
		We give a computability result for open Gromov-Witten invariants based on open WDVV equations. This is analogous to the result of Kontsevich-Manin for closed Gromov-Witten invariants.
		For greater generality, we base the argument on a formal object, the Frobenius superpotential, that generalizes several different definitions of open Gromov-Witten invariants.
		As an application, we prove vanishing properties for open Gromov-Witten invariants on products of projective spaces.
	\end{abstract}
	
	\maketitle
	%
	%
	
	%
	%
	\setcounter{tocdepth}{1}
	\tableofcontents
	
	\section{Introduction}
	
	\subsection{Closed Gromov-Witten theory}
	Let $(X,\omega)$ be a closed symplectic manifold with $\dim_\RR X = 2n$ for some $n\geq 1$, and let $L\subset X$ be a closed, oriented, and relatively spin Lagrangian submanifold. Let $J$ be an $\omega$-tame almost complex structure on $X$. Let $H^*(X;\RR)$ and $H^*(X,L;\RR)$ be the absolute and relative de Rham cohomology rings of $X$ and $(X,L)$ with coefficients in $\RR$, respectively (see Section \ref{sec: Topological Preliminaries} for more details on the version of relative cohomology used here). Let $H_2(X;\ZZ)$ and $H_2(X,L;\ZZ)$ be the absolute and relative second singular homology groups of $X$ and $(X,L)$, respectively.
	
	Closed genus 0 Gromov-Witten theory seeks to provide suitable machinery to enumerate $J$-holomorphic curves of genus $0$ of a given degree $\beta\in H_2(X;\ZZ)$ that intersect the Poincar\'{e} duals to some given cohomology classes $\gamma_1,\dots,\gamma_\ell\in H^*(X;\RR)$.
	These invariants can be thought of as linear homomorphisms $\GW_{\beta}:H^*(X;\RR)^{\otimes\ell}\rightarrow\RR$, for $\beta\in H_2(X;\ZZ)$ and $\ell\in\ZZ_{\geq 0}$. For a given invariant $\GW_{\beta}\big(\gamma_1\otimes\dots\otimes\gamma_\ell\big)$, we call the class $\beta$ the \emph{degree} of the invariant, and we call the classes $\gamma_1,\dots,\gamma_\ell$ the \emph{constraints} of the invariant. These homomorphisms can be realized as coefficients of a power series in a Novikov ring, called the closed Gromov-Witten potential, and they satisfy basic properties, called the Gromov-Witten axioms (see Section \ref{Closed GW Invariants} for further details).
	
	\subsection{Open Gromov-Witten theory} In a similar fashion to the closed theory, genus 0 open Gromov-Witten theory seeks to provide suitable machinery to enumerate pseudoholomorphic disks with boundary in $L$, of a relative degree $\beta\in H_2(X,L;\ZZ)$, that intersect $L$ at $k\in\ZZ_{\geq 0}$ points and intersect the Poincar\'{e} duals to given relative cohomology classes $\gamma_1,\dots,\gamma_\ell\in H^*(X,L;\RR)$. Genus 0 open Gromov-Witten invariants were defined in several different works, such as \cite{Solomon2006Intersection}, \cite{Fukaya2011CountingPseudoHolomorphic}, \cite{Georgieva2016OGWDiskInvariants}, \cite{SolomonTukachinsky2019relative}, and \cite{SolomonTukachinsky2021PointLikeBoundingChain} . Throughout, we will restrict our discussion to genus $0$ open and closed Gromov-Witten invariants.
	
	Under any definition, the open Gromov-Witten invariants can be thought of as a family of linear homomorphisms $\OGWbar_{\beta,k}^{\parentheses{\ell}}:H^*(X,L;\RR)^{\otimes\ell}\rightarrow\RR$, for $\beta\in H_2(X,L;\ZZ)$ and $k,\ell\in\ZZ_{\geq 0}$. For a given invariant $\OGWbar_{\beta,k}^{\parentheses{\ell}}\big(\gamma_1\otimes\dots\otimes\gamma_\ell\big)$, we call the class $\beta$  the \emph{degree} of the invariant,  $k$ the \emph{number of boundary constraints} of the invariant, and $\gamma_1,\dots,\gamma_\ell$ the \emph{interior constraints} of the invariant. In a fashion similar to the closed case, these homomorphisms can be realized as coefficients of some power series in a Novikov ring, called the open Gromov-Witten superpotential, and they satisfy basic properties, analogous to the Gromov-Witten axioms.
	
	\subsection{Basic notation and definitions}
	Denote by $\mu:H_2(X,L;\ZZ)\rightarrow\ZZ$ the Maslov index and let $c_1:H_2(X;\ZZ)\rightarrow\ZZ$ be the first Chern class.
	Write
	\begin{equation*}
		\Pi = H_2(X,L;\ZZ)/S_L,
	\end{equation*}
	with $S_L$ some subgroup of the group $\ker(\omega\oplus\mu)$. Let $\beta_0$ be the zero element of $\Pi$. Set
	\begin{equation*}
		\Pi_{\geq 0} = \left\{\beta\in\Pi\;\middle|\;\omega(\beta)\geq 0\right\}
	\end{equation*}
	and
	\begin{equation*}
		\Pi_{> 0} = \left\{\beta\in\Pi\;\middle|\;\omega(\beta)> 0\right\}.
	\end{equation*}
	
	Let $\pi:H_2(X;\ZZ)\rightarrow H_2(X,L;\ZZ)$ be the map coming from the long exact sequence for relative homology, and let $\varpi:H_2(X;\ZZ)\rightarrow \Pi$ be the composition of $\pi$ with the projection map from $H_2(X,L;\ZZ)$ to $\Pi$.
	
	The following statement holds, as shown, for example, in \cite[Proposition~4.1.4]{Mcduff2004JHolomorphicCA}:
	\begin{Proposition}\label{hbar Proposition}
		There exists a constant $\hbar>0$, such that:
		\begin{enumerate}
			\item For any $\beta\in\Pi$, if $\omega(\beta)>0$ and there exists a $J$-holomorphic disk $u:(D,\partial D)\rightarrow(X,L)$ with $u_*[D] = \beta$, then $\omega(\beta)\geq\hbar$.
			\item For any $\beta\in H_2(X;\ZZ)$, if $\omega(\beta)>0$ and there exists a $J$-holomorphic curve $u:S^2\rightarrow X$ such that $u_*[X] = \beta$, then $\omega(\beta)\geq\hbar$.
		\end{enumerate}
	\end{Proposition}
	
	 Define the Novikov coefficient rings by
	\begin{equation*}
		\Lambda = \left\{\sum\limits_{i=0}^\infty a_iT^{\beta_i}\;\middle|\;a_i\in\RR,\beta_i\in H_2(X,L;\RR),\omega\left(\beta_i\right)\geq 0,\lim\limits_{i\rightarrow\infty}\omega\left(\beta_i\right)=\infty\right\},
	\end{equation*}
	\begin{equation*}
		\Lambda_c = \left\{\sum\limits_{i=0}^\infty a_iT^{\varpi(\beta_i)}\;\middle|\;a_i\in\RR,\beta_i\in H_2(X;\RR),\omega\left(\beta_i\right)\geq 0,\lim\limits_{i\rightarrow\infty}\omega\left(\beta_i\right)=\infty\right\}.
	\end{equation*}
	Let $s,t_0,\dots,t_N$ be formal variables. Define the following rings of formal power series:
	\begin{equation*}
		R = \Lambda[[s,t_0,\dots,t_N]],
	\end{equation*}
	\begin{equation*}
		Q = \Lambda_c[[t_0,\dots,t_N]].
	\end{equation*}	
	
	\subsection{Computability for Gromov-Witten invariants} In \cite{KontsevitchManin1994}, Kontsevich and Manin proved the following theorem:
	\begin{Theorem}\label{Konsevitch-Manin}
		Let $(X,\omega)$ be a symplectic manifold, and assume that the cohomology ring $H^*(X;\RR)$ is generated by $H^2(X;\RR)$. Then all closed Gromov-Witten invariants of $X$ can be computed from the following set of initial values:
		\begin{align*}
			\{\GW_{\beta}\left(\gamma_1\otimes\gamma_2\otimes\gamma_3\right)\, |\, &\beta\in H_2(X;\ZZ),\, \gamma_1,\gamma_2,\gamma_3\in H^*(X;\RR),\\
			&|\gamma_3|=2, 2n+2c_1(\beta)=|\gamma_1|+|\gamma_2|+|\gamma_3|\}.
		\end{align*}
	\end{Theorem}
	The theorem shows that closed Gromov-Witten invariants can be recursively computed from an initial set of known values in some cases. Under the further assumption that $X$ is monotone, this initial set of values is also finite.

	Open Gromov-Witten invariants have been shown to be computable from a finite initial set of values in several special cases. These cases include blowups of $\CC P^2$ at real points and pairs of conjugate points~\cite{HorevSolomon, ChenZinger2021WDVVRelations}, $\CC P^n$ with odd $n$ \cite{SolomonTukachinsky2019relative, GeorgievaZingerEnumeration},
	and quadric hypersurfaces~\cite{Examples_of_relative_quatum_cohomology,ChenZinger2021WDVVRelations}.
	In all these special cases, the Lagrangian submanifold is taken to be the fixed locus of an anti-symplectic involution.
	Another example is the Chiang Lagrangian \cite{ChiangLagrangianRQH}. Furthermore, in some cases for $2$ or $3$ dimensions, the open Gromov-Witten invariants coincide, up to a known factor, with the Welschinger invariants~\cite[Theorem 5]{SolomonTukachinsky2021PointLikeBoundingChain}, for which various computational results have been published.
	
	It is natural to ask whether an analogous statement to Theorem \ref{Konsevitch-Manin} can be made for open Gromov-Witten invariants. This is the main objective of the current work.

	\subsection{Frobenius superpotentials} In this work, to make the results general and apply them to different definitions of the open Gromov-Witten invariants, we define an object called a \emph{Frobenius superpotential}. A Frobenius superpotential is an element of the ring $R$ that satisfies Gromov-Witten type axioms, as detailed in Section \ref{section: Frobenius Superpotential}. The Frobenius superpotential serves as an algebraic generalization of the open Gromov-Witten superpotential.

	\subsection{The main theorems}\label{main theorems}
	\begin{Notation}\label{Notation Convention for Frobenius}
		For any list of $N+1$ non-negative integers $\vec{r} = \left(r_0,\dots,r_N\right)\in\left(\ZZ_{\geq 0}\right)^{N+1}$ and for any non-negative integer $\ell$, we write $\ell\vdash\vec{r}$ if $\sum\limits_{i=0}^Nr_i=\ell$.
	\end{Notation}
	Let $K+1$ be the dimension of $H^*(X;\RR)$, and $N+1$ the dimension of $H^*(X,L;\RR)$. In Section \ref{Choice of Bases for Relative and Absolute Cohomologies} we choose a basis $\left\{\Gamma_0,\dots,\Gamma_N\right\}$ for $H^*(X,L;\RR)$ compatible with a splitting of the exact sequence~\eqref{Relative long exact sequence} relating $H^*(X;\RR)$ and $H^*(X,L;\RR)$.
	Let
	\begin{equation}\label{set of OGWs}
		A = \left\{\OGWbar_{\beta,k}^{(\ell)}:H^*(X,L;\RR)^{\otimes\ell}\rightarrow\RR \middle|\;\beta\in \Pi_{\geq 0},\,k,\ell\in\ZZ_{\geq 0}\right\}
	\end{equation}
	be a family of homomorphisms, and denote by $\Omega$ the following element in $R$:
	\begin{equation}\label{Open Generating Function}
		\Omega = \sum\limits_{\substack{\beta\in \Pi_{\geq 0}\\ k,\ell\in\ZZ_{\geq 0}}}\sum\limits_{\substack{\invecZN{r}\\ \sumsupto{\ell}{\vec{r}}}}T^\beta\frac{s^k}{k!}\prod\limits_{i=0}^{N}\frac{t_{N-i}^{r_{N-i}}}{(r_{N-i})!}\OGWbar_{\beta,k}^{\parentheses{\ell}}\big(\bigotimes\limits_{i=0}^{N}\Gamma_{N-i}^{\otimes r_{N-i}}\big).
	\end{equation}
	Assume that the following conditions are satisfied:
	\begin{enumerate}
		\item \label{cond_1_main_therorem}$\Omega$ is a Frobenius superpotential (see Section \ref{section: Frobenius Superpotential} for details).
		\item \label{cond_2_main_therorem}The absolute cohomology ring $H^*(X;\RR)$ is generated by the second cohomology, $H^2(X;\RR)$.
	\end{enumerate}
	
	Then we prove two main theorems:
	\begin{mainTheorem}\label{All OGW are computable 1 Fellowship of the OGW}
		For any $\beta\in H_2(X,L;\ZZ)$, any $k,\ell\in\ZZ_{\geq 0}$, and any $\gamma_1,\dots,\gamma_\ell\in H^*(X,L;\RR)$, the value $\OGWbar_{\beta,k}^{\parentheses{\ell}}\left(\gamma_1\otimes\dots\otimes\gamma_\ell\right)$ can be computed from the following sets of initial values:
		\begin{align*}
			\big\{&\OGWbar_{\beta,k}^{\parentheses{0}}\, \big|\, \beta\in\Pi_{\geq 0},\,  n-3+\mu(\beta)+k=kn\big\},\\
			\big\{&\OGWbar_{\beta,0}^{\parentheses{1}}\left(\Gamma_i\right)\, \big|\, \beta\in\Pi_{\geq 0},\, 0\leq i\leq K,\, n-3+\mu(\beta)+2=|\Gamma_i|\big\},\\
			\big\{&\GW_{\beta}\left(\gamma_1\otimes\gamma_2\otimes\gamma_3\right)\, \big|\, \beta\in H_2(X;\ZZ),\, \gamma_1,\gamma_2,\gamma_3\in H^*(X;\RR),\\
			&\hspace{9em}|\gamma_3|=2, 2n+2c_1(\beta)=|\gamma_1|+|\gamma_2|+|\gamma_3|\big\}.
		\end{align*}
	\end{mainTheorem}

	\begin{mainTheorem}\label{All OGW are computable 2 two OGW two computable}
		Assume $S_L = \ker(\omega\oplus\mu)$. Assume that there exists some $\beta'\in\Pi_{>0}$ with $\omega(\beta') = \hbar$, some $\ell'\in\ZZ_{\geq 1}$, and some $0\leq i_1,...,i_{\ell'}\leq K$, for which
		\[ \OGWbar_{\beta',0}^{\parentheses{\ell'}}\left(\Gamma_{i_1}\otimes\dots\otimes\Gamma_{i_{\ell'}}\right)\neq 0. \]
		Assume further that the pair $(\beta',\ell')$ is minimal in the sense that for any $\beta''\in \Pi_{>0}$ with $\omega(\beta'') = \hbar$, any $\ell''\in\ZZ_{\geq 0}$ that satisfies $\ell''<\ell'$, and any $1\leq j_1,\dots,j_{\ell''}\leq K$, the term $\OGWbar_{\beta',0}^{(\ell'')}\left(\Gamma_{j_1}\otimes\dots\otimes\Gamma_{j_{\ell''}}\right)$ vanishes.
		
		If either $\ell'> 1$ or $\left|\Gamma_{i_1}\right| = 2n$, then the sets of initial values from Theorem~\ref{All OGW are computable 1 Fellowship of the OGW} can be restricted to
		\begin{align*}
			\big\{&\OGWbar_{\beta',0}^{\parentheses{\ell'}}\left(\Gamma_{i_1}\otimes\dots\otimes\Gamma_{i_{\ell'}}\right)\big\}\\
			\big\{&\OGWbar_{\beta,k}^{\parentheses{0}}\, \big|\, \beta\in\Pi_{\geq 0},\, k\in\left\{0,1\right\},\, n-3+\mu(\beta)+k=kn\big\},\\
			\big\{&\OGWbar_{\beta,0}^{\parentheses{1}}\left(\Gamma_i\right)\, \big|\, \beta\in\Pi_{\geq 0},\, 0\leq i\leq K,\, n-3+\mu(\beta)+2=|\Gamma_i|\big\},\\
			\big\{&\GW_{\beta}\left(\gamma_1\otimes\gamma_2\otimes\gamma_3\right)\, \big|\, \beta\in H_2(X;\ZZ),\, \gamma_1,\gamma_2,\gamma_3\in H^*(X;\RR),\\
			&\hspace{9em}|\gamma_3|=2, 2n+2c_1(\beta)=|\gamma_1|+|\gamma_2|+|\gamma_3|\big\}.
		\end{align*}
	\end{mainTheorem}
	
	\begin{Remark}\label{ell'=0 remark}
		Theorem \ref{All OGW are computable 2 two OGW two computable} does not give us information in the case when we only know that there exists some $\beta'\in\Pi_{>0}$ with $\omega(\beta')=\hbar$ such that
		\[ \OGWbar_{\beta',0}^{\parentheses{0}} \neq 0. \]
		However, it follows from the degree axiom (see Section \ref{section: Frobenius Superpotential} for more details) that this case is only possible if $\mu(\beta')=3-n$, which does not hold for a wide variety of examples (for instance, if $L$ is monotone and $n\geq 3$, or if $L$ is a special Lagrangian in a Calabi-Yau $X$ with $n\neq 3$).
	\end{Remark}

	\begin{Remark}
		Theorem~\ref{All OGW are computable 2 two OGW two computable} is a generalization of the computation for projective spaces given in~\cite{SolomonTukachinsky2019relative}. It does not generalize the computation offered for the Chiang Lagrangian~\cite{ChiangLagrangianRQH} nor for quadric hypersurfaces~\cite{Examples_of_relative_quatum_cohomology}. Each of those other computations uses properties which are rather specific to the manifold in question.
		There does seem to be a principle of the form ``There exists an invariant with interior constraints the non-vanishing of which guarantees computability of invariants with boundary constraints only''. At this time we are not sure whether a recipe exists for identifying such an initial invariant in general.
	\end{Remark}

	The proofs of Theorems \ref{All OGW are computable 1 Fellowship of the OGW} and \ref{All OGW are computable 2 two OGW two computable}, and the proof of Remark \ref{ell'=0 remark}, are given in Section \ref{Results}.
	
	As stated before and as expanded upon in Section \ref{subsection:Geometric Realization}, in some situations, a Frobenius superpotential $\Omega$ can be realized as a generating function for open Gromov-Witten invariants. A corollary of Theorems \ref{All OGW are computable 1 Fellowship of the OGW} and \ref{All OGW are computable 2 two OGW two computable} is that in such cases, and under the specified topological assumptions, if the sets of initial values are known, then all open Gromov-Witten invariants can be computed.
	
	\subsection{Method of proof} The proof of Theorems \ref{All OGW are computable 1 Fellowship of the OGW} and \ref{All OGW are computable 2 two OGW two computable} is to be carried out in several steps. First, we prove a lemma that shows that certain closed Gromov-Witten invariants, with some specific degrees, vanish. We then define a partial order on the set $A$ defined in \eqref{set of OGWs}, and prove four reduction lemmas. Each lemma states that elements of a certain subset of $A$ may be computed via elements of strictly lower order. A main tool for the proof of the reduction lemmas is the open WDVV (Witten–Dijkgraaf–Verlinde–Verlinde) relations, which is a set of partial differential equations satisfied by the Frobenius superpotential $\Omega$. Theorems \ref{All OGW are computable 1 Fellowship of the OGW} and \ref{All OGW are computable 2 two OGW two computable} follow from the reduction lemmas.
	
	It is worthwhile to note that, while Theorems \ref{All OGW are computable 1 Fellowship of the OGW} and \ref{All OGW are computable 2 two OGW two computable} make an analogous statement to Theorem \ref{Konsevitch-Manin}, the proofs of Theorems \ref{All OGW are computable 1 Fellowship of the OGW} and \ref{All OGW are computable 2 two OGW two computable} are more involved. The idea of the proof of Theorem \ref{Konsevitch-Manin} was the following -- take some closed Gromov-Witten invariant $\GW_{\beta}\big(\gamma_1\otimes\dots\otimes\gamma_{{\ell}}\big)$, for some degree $\beta\in H_2(X;\ZZ)$ and constraints $\gamma_1,\dots,\gamma_\ell\in H^*(X;\RR)$. It is assumed that the second cohomology $H^2(X;\RR)$ generates the absolute cohomology ring $H^*(X;\RR)$. Therefore, we may generally decompose one of the interior constraints into a product of a divisor and a cohomology class of lower degree. Using the Gromov-Witten axioms, notably among them the WDVV relations (see Section \ref{Closed GW Invariants} for more details), one can then show a relation between $\GW_{\beta}\big(\gamma_1\otimes\dots\otimes\gamma_{{\ell}}\big)$ and Gromov-Witten invariants with constraints of lower degree.
	
	In the open theory, one can deal with the interior constraints of the open Gromov-Witten invariants in a similar manner, and indeed the method of proof for Theorem \ref{All OGW are computable 1 Fellowship of the OGW} is similar to the method of proof for Theorem \ref{Konsevitch-Manin}. However, the boundary constraints of the open Gromov-Witten invariants do not admit such a decomposition. This sets an obstruction on Konsevitch and Manin's original method of proof. The additional assumption in the statement of Theorem \ref{All OGW are computable 2 two OGW two computable} allows us to circumvent this obstruction.
	
	\subsection{Application}

	We apply Theorem \ref{All OGW are computable 1 Fellowship of the OGW} to the special case of
	\[
	(X,L) = (\CC P^{n_1}\times\CC P^{n_2},\RR P^{n_1}\times\RR P^{n_2})
	\]
	with $n_1,n_2>1$ and $n_1,n_2 \equiv 1\pmod 2$, with respect to the product symplectic and complex structures.
	In Section~\ref{products of projective spaces} we
	go through the computational steps from Theorem \ref{All OGW are computable 1 Fellowship of the OGW} to verify that a Frobenius superpotential on $(X,L)$ will be trivial. The precise statement is given in Propositions~\ref{OGWs Vanish for n1 n2 are 5mod6} and~\ref{OGWs Vanish for n1 n2 are 1mod2 and pointlike}.

	\begin{Remark}
		In the closed case, there is a product formula that allows to compute the closed Gromov-Witten invariants of a product of two manifolds. This formula is stated and proved in \cite{behrend1997product}. In the open case, the open Gromov-Witten invariants of $(\CC P^{n_i}, \RR P^{n_i})$ were computed and shown to not vanish in \cite{SolomonTukachinsky2019relative}. The result above hints that perhaps there is no straightforward generalization of said product formula for open Gromov-Witten invariants.
		
	\end{Remark}
	
	\subsection{Acknowledgements} We would like to thank Peleg Bar-Lev and Daniel Tsodikovich for helpful conversations.
	We are thankful to Jake Solomon for pointing out a mistake in our original argument in Section~\ref{products of projective spaces}.
	We were partially supported by the ISF grant No. 2793/21. The second author was partially supported by the Colton Foundation.
	
	\section{Topological preliminaries}\label{sec: Topological Preliminaries}
	
	\subsection{De Rham cohomology and relative cohomology}
	We write $A^*(X;\RR)$ for the de-Rham complex of differential forms on $X$. Define the subcomplex $\hat{A}^*(X,L;\RR)\subseteq A^*(X;\RR)$ by
	\[ \hat{A}^*(X,L;\RR):=\left\{\eta\in A^*(X;\RR)\;\middle|\; \int_L\eta = 0\right\}. \]
	The exterior differential $d$ makes $\hat{A}^*(X,L;\RR)$ a cochain complex. Let $H^*(X,L;\RR)$ be the cohomology of the complex $\hat{A}^*(X,L;\RR)$. Denote by $\mathfrak{i}_\RR:H^*(X;\RR)\rightarrow\RR$ the map given by integration over $L$. Then we have a short exact sequence:
	\[ 0\longrightarrow \hat{A}^*(X,L;\RR) \longrightarrow A^*(X;\RR) \overset{\mathfrak{i}_\RR}{\longrightarrow} \RR \longrightarrow 0. \]
	By the Snake Lemma, this induces the long exact sequence:
	\begin{equation}\label{Relative long exact sequence}
		\begin{tikzcd}
			H^*(X,L;\RR) \arrow[rr, "\rho"]&			 & H^*(X;\RR)\arrow[dl,"\mathfrak{i}_\RR"] \\
			&\RR\arrow[ul, "y"] &
		\end{tikzcd},
	\end{equation}
	where $\rho:H^*(X,L;\RR)\rightarrow H^*(X,\RR)$ is the map induced by the inclusion
	\[ \hat{A}^*(X,L;\RR)\hookrightarrow A^*(X;\RR), \]
	and $y:\RR\rightarrow H^*(X,L;\RR)$ is the connecting homomorphism for the long exact sequence.
	
	\subsection{A choice of bases for relative and absolute cohomologies}\label{Choice of Bases for Relative and Absolute Cohomologies}
	Let $\left\{\Delta_0,\dots,\Delta_K\right\}$ be an ordered homogeneous basis for the absolute cohomology ring $H^*(X;\RR)$. Write $|\Delta_i|=j_i$ where $\Delta_i\in H^{j_i}(X;\RR)$.
	
	We have two distinct cases to consider, depending on the homology class of $L$ in $H_*(X;\RR)$. If $[L]=0$ in $H_n(X;\RR)$, then $\mathfrak{i}_\RR = 0$. Therefore $\ker(\mathfrak{i}_\RR)=H^*(X;\RR)$ and the long exact sequence \eqref{Relative long exact sequence} splits into
	\begin{equation}\label{L=0}
		H^*(X,L;\RR) \cong \text{Im}\rho\oplus\ker\rho = \ker(\mathfrak{i}_\RR)\oplus\text{Im}(y) = H^*(X;\RR)\oplus Span(y(1)).
	\end{equation}
	On the other hand, if $[L]\neq 0$ in $H_n(X;\RR)$, then $\text{Im}(\mathfrak{i}_\RR)=\RR$. Therefore, $\ker(y) = \text{Im}(\mathfrak{i}_\RR) =\RR$ and so $\text{Im}(y)=0$. In that case,
	\begin{equation*}
		H^*(X,L;\RR) \cong \text{Im}(\rho)\oplus\ker(\rho)=\text{Im}(\rho)\oplus\text{Im}(y) = \text{Im}(\rho),
	\end{equation*}
	and the long exact sequence \eqref{Relative long exact sequence} splits into
	\begin{equation}\label{L neq 0}
		H^*(X;\RR)\cong \ker(\mathfrak{i}_\RR)\oplus\text{Im}(\mathfrak{i}_\RR) = \text{Im}\left(\rho\right)\oplus\text{Im}(\mathfrak{i}_\RR) = H^*(X,L;\RR)\oplus\RR,
	\end{equation}
	where $\text{Im}(\mathfrak{i}_\RR)$ is generated by integration on $\text{PD}([L])$, the Poincar\'{e} dual to $[L]$.\\
	By \eqref{L=0} and \eqref{L neq 0}, we may choose a basis $\left\{\Gamma_0,\dots,\Gamma_N\right\}$ for $H^*(X,L;\RR)$ in the following way:
	\begin{enumerate}
		\item If $[L]=0$ in $H^*(X;\RR)$, then $N=K+1$ and we may identify $\Gamma_i=\Delta_i$ for all $0\leq i\leq K$ and take $\Gamma_\diamond=\Gamma_{K+1}=y(1)$. Thus, $\left\{\Gamma_0,\dots,\Gamma_K,\Gamma_\diamond\right\}$ is a basis for $H^*(X,L;\RR)$.
		\item If $[L]\neq 0$ in $H^*(X;\RR)$, then $N=K-1$ and $H^*(X,L;\RR)$ may be viewed as a subspace of $H^*(X;\RR)$. In this case, without loss of generality, $H^*(X,L;\RR)$ is spanned by $\left\{\Delta_0,\dots,\Delta_{K-1}\right\}$ and $\Delta_K = \text{PD}([L])$. Thus, if we denote $\Gamma_i=\Delta_i$ for all $0\leq i\leq K-1$, then $\left\{\Gamma_0,\dots,\Gamma_{K-1}\right\}$ is a basis for $H^*(X,L;\RR)$.
	\end{enumerate}
	In any case, we write $\kappa = \min\{K,N\}$.
	
	Let
	\[
	P_\RR:H^*(X,L;\RR)\rightarrow\text{Coker}(\mathfrak{i}_\RR)
	\]
	be a homomorphism that is left inverse to
	\[
	\bar{y}:\text{Coker}(\mathfrak{i}_\RR)\rightarrow H^*(X,L;\RR),
	\]
	induced by $y$ from \eqref{Relative long exact sequence}. Note that if $[L]\neq 0$ in $H^*(X;\RR)$, then $\text{Im}(\mathfrak{i}_\RR)=\RR$ and $\text{Coker}(\mathfrak{i}_\RR)=0$, so $P_\RR=0$. If $[L]=0$ in $H^*(X;\RR)$, there may be several different choices for $P_\RR$. We take $P_\RR$ that is compatible with the splitting \eqref{L=0}, namely, $P_\RR(\Gamma_\diamond)=1$ and $P_\RR(\Gamma_j)=0$ for all $0\leq j\leq K$. Note that in either case we get that our choice of $P_\RR$ vanishes on $H^*(X;\RR)$.
	
	\subsection{Relative spin structure} Let $\mathfrak{s}_L$ be a relative spin structure on $L$. Let $w_2(TL)$ be the second Steifel-Whitney class of the tangent bundle $TL$. Then $\mathfrak{s}_L$ determines a class $w_s\in H^2(X;\ZZ/2\ZZ)$ that satisfies $w_2(TL)=w_s|_L$. For more details on relative spin structures, see \cite[Definition 8.1.2]{FOOOLagrangianIntersection} and \cite[Definition 3.1.2(c)]{WehrheimWoodwardpseudoholomorphicquilts}.

	\section{Closed Gromov-Witten invariants and the Gromov-Witten axioms}\label{Closed GW Invariants}
	\emph{Genus zero Gromov-Witten invariants of degree $\beta$} are particular homomorphisms
	\[
	\GW_{\beta}:H^*(X;\RR)^{\otimes \ell}\rightarrow\RR,
	\quad \beta\in H_2(X;\ZZ), \ell\in\ZZ_{\geq 0}.
	\]
	While we do not need the particulars of the definition, it is important that $\GW_\beta$ satisfy a set of basic properties known as the Gromov-Witten axioms detailed below (see \cite{KontsevitchManin1994}).
	\begin{Proposition}[The Gromov-Witten Axioms for Closed Gromov-Witten Invariants]
		
		Let $\beta\in H_2(X;\ZZ)$ and $\gamma_1,\dots,\gamma_\ell\in H^*(X;\RR)$. Then the following statements hold:
		\begin{enumerate}[label = (\alph*)]
			\item (\emph{Degree}) $\GW_{\beta}(\gamma_1\otimes\dots\otimes\gamma_\ell) = 0$ unless
			\begin{equation} \label{Closed Degree}
				2n-6+2c_1(\beta)+2\ell = \sum\limits_{i=1}^\ell|\gamma_i|.
			\end{equation}
			
			\item (\emph{Fundamental Class})
			\begin{equation}\label{Closed Fundamental Class}
				\GW_{\beta}\left(1\otimes\gamma_1\otimes\dots\otimes\gamma_{\ell-1}\right) = \begin{cases}
					\int_X\gamma_1\wedge\gamma_2, & (\ell,\beta) = (3,0),\\
					0, &\text{otherwise}.
				\end{cases}
			\end{equation}
			
			\item (\emph{Zero})
			\begin{equation}\label{Closed Zero}
				\GW_{0}\left(\gamma_1,\dots,\gamma_\ell\right) = \begin{cases}
					\int_X\gamma_1\wedge\gamma_2\wedge\gamma_3, & \ell = 3,\\
					0, &\text{otherwise}.
				\end{cases}
			\end{equation}
			
			\item (\emph{Divisor}) If $|\gamma_\ell| = 2$ then
			\begin{equation} \label{Closed Divisor}
				\GW_{\beta}\left(\gamma_1,\dots,\gamma_\ell\right) = \int_\beta\gamma_\ell\cdot\GW_{\beta}\left(\gamma_1\dots,\gamma_{\ell-1}\right).
			\end{equation}
		\end{enumerate}
	\end{Proposition}
	
	Furthermore, the closed Gromov-Witten invariants satisfy another property, called the \textit{WDVV equation}, named after Witten–Dijkgraaf–Verlinde–Verlinde. It can be phrased concisely as a family of PDEs using the closed Gromov-Witten potential and the intersection matrix $g$, as explained in the following.
	
	\begin{Notation}\label{gij matrix}
		Let $\left\{\Delta_0,\dots,\Delta_K\right\}$ be the basis for $H^*(X;\RR)$ chosen in Section \ref{sec: Topological Preliminaries}. We denote by $\left(g_{ij}\right)_{0\leq i,j\leq K+1}$ the $(K+1)\times(K+1)$ matrix satisfying
		\begin{equation*}
			g_{ij} = \int_X\Delta_i\wedge\Delta_j
		\end{equation*}
		for all $0\leq i,j\leq K$. We denote by $\left(g^{ij}\right)_{0\leq i,j\leq K}$ the inverse matrix to  $\left(g_{ij}\right)_{0\leq i,j\leq K}$.
	\end{Notation}
	
	\begin{Definition}
		The \textit{closed potential} $\Phi$ is a power series in $Q$, defined as
		\begin{equation}\label{Closed Generating Function}
			\Phi = \sum\limits_{\substack{\beta\in H_2(X;\ZZ)\\ r_i\in\ZZ_{\geq 0}}}(-1)^{w_s(\beta)}T^{\varpi(\beta)}\prod\limits_{i=0}^{K}\frac{t_{K-i}^{r_{K-i}}}{r_{K-i}!}\GW_\beta\big(\bigotimes\limits_{i=0}^{K}\Delta_{K-i}^{\otimes r_{K-i}}\big).
		\end{equation}
	\end{Definition}
	
	\begin{Remark}
		The term $(-1)^{w_s(\beta)}$ does not typically appear in the definition of the closed Gromov-Witten potential. Nevertheless, it is needed when the potential is used in open Gromov-Witten theory.
	\end{Remark}
	
	As stated, the closed potential $\Phi$ satisfies a family of PDEs, called the WDVV equations.
	\begin{Proposition}[Closed WDVV Equations]
		For all $0\leq a,b,c,d\leq K$, the following relation holds:
		\begin{equation}\label{Closed WDVV}
			\sum\limits_{0\leq\mu,\nu\leq K}\partial_a\partial_b\partial_\mu\Phi g^{\mu\nu}\partial_c\partial_d\partial_\nu\Phi = \sum\limits_{0\leq\mu,\nu\leq K}\partial_a\partial_d\partial_\mu\Phi g^{\mu\nu}\partial_b\partial_c\partial_\nu\Phi,
		\end{equation}
		where we use the shorthand notation $\partial_j\Phi = \partial_{t_j}\Phi$.
	\end{Proposition}

	\section{Frobenius superpotentials}\label{section: Frobenius Superpotential}
	
	\begin{Definition}\label{def: Frobenius superpotential}
		Let
		\[
		\left\{\OGWbar^{(\ell)}_{\beta,k}\;\middle|\; \beta\in \Pi_{\geq 0}, k,\ell\in\ZZ_{\geq 0}\right\}
		\]
		be a family of homomorphisms $\OGWbar^{(\ell)}_{\beta,k}:H^*(X,L;\RR)^{\otimes\ell}\to\RR$. Let
		\begin{equation*}
			\Omega = \sum\limits_{\substack{\beta\in\Pi_{\geq 0}\\ k,\ell\in\ZZ_{\geq 0}}}\sum\limits_{\substack{\invecZN{r}\\ \sumsupto{\ell}{\vec{r}}}}
			T^\beta\frac{s^k}{k!}\prod\limits_{i=0}^{N}\frac{t_{N-i}^{r_{N-i}}}{(r_{N-i})!}
			\OGWbar^{(\ell)}_{\beta,k}\big(\bigotimes\limits_{i=0}^{N}\Gamma_{N-i}^{\otimes r_{N-i}}\big).
		\end{equation*}
		In particular, $\Omega\in R$. Then $\Omega$ is called a \emph{Frobenius superpotential} if the following conditions are satisfied:
		\begin{enumerate}[label=(\roman*)]
			\item (\emph{Open WDVV equations})\label{Open WDVV Relations} For all $a,b,c\in \{0,\dots,N\}\cup\{s\}$, the following PDEs, called the open WDVV equations, hold:
			\begin{align}
				\label{Open WDVV Inner Constraints}
				\sum\limits_{0\leq \mu,\nu\leq \kappa}\partial_a\partial_\mu\Omega g^{\mu\nu}\partial_b\partial_c\partial_\nu\Phi - \partial_a\partial_s\Omega\partial_b\partial_c\Omega
				=& \sum\limits_{0\leq \mu,\nu\leq \kappa}\partial_a\partial_b\partial_\mu\Phi g^{\mu\nu}\partial_c\partial_\nu\Omega - \partial_a\partial_b\Omega\partial_s\partial_c\Omega,\\
				\label{Open WDVV Boundary Constraint}
				-\partial_a\partial_s\Omega\partial_b\partial_s\Omega=&\sum\limits_{0\leq \mu,\nu\leq \kappa}\partial_a\partial_b\partial_\mu\Phi g^{\mu\nu}\partial_s\partial_\nu\Omega-\partial_a\partial_b\Omega\partial_s^2\Omega,
			\end{align}
			where we use the shorthand notation $\partial_j\Omega = \partial_{t_j}\Omega$ for $0\leq j\leq N$.
		\end{enumerate}
		
		For the remaining conditions, let $\beta\in \Pi_{\geq 0}$ and $k,\ell\in\ZZ_{\geq 0}$, and $\gamma_1,\dots,\gamma_{{\ell}}\in \left\{\Gamma_0,\dots,\Gamma_N\right\}$.
		
		\begin{enumerate}[resume, label=(\roman*)]
			\item  (\emph{Energy Gap}) If $\omega(\beta)<\hbar$ and $\beta\neq \beta_0$, then
			\begin{equation}\label{Energy Gap}
				\OGWbar^{(\ell)}_{\beta,k}(\gamma_1\otimes\dots\otimes\gamma_\ell)=0.
			\end{equation}
			
			\item (\emph{Symmetry}) Let $\sigma$ be a permutation in $S_\ell$. Then
			\begin{equation}\label{Open Symmetry}
				\OGWbar^{(\ell)}_{\beta,k}\left(\gamma_1\otimes\dots\otimes\gamma_\ell\right) = (-1)^{s_\sigma(\gamma)}\OGWbar^{(\ell)}_{\beta,k}\left(\gamma_{\sigma(1)}\otimes\dots\otimes\gamma_{\sigma(\ell)}\right),
			\end{equation}
			where $s_\sigma(\gamma) := \sum\limits_{i>j,\sigma(i)<\sigma(j)}\left|\gamma_{\sigma(i)}\right|\cdot\left|\gamma_{\sigma(j)}\right|$.
			
			\item (\emph{Degree}) $\OGWbar^{(\ell)}_{\beta,k}(\gamma_1\otimes\dots\otimes\gamma_\ell) = 0$ unless
			\begin{equation} \label{Open Degree}
				n-3+\mu(\beta)+2\ell+k =  kn + \sum\limits_{i=1}^\ell|\gamma_i|.
			\end{equation}
			
			\item (\emph{Fundamental Class}) Assume $\ell\geq 1$. Then
			\begin{equation}\label{Open Fundamental Class}
				\OGWbar^{(\ell)}_{\beta,k}\left(1\otimes\gamma_1\otimes\dots\otimes\gamma_{\ell-1}\right) = \begin{cases}
					-1, & (\ell,\beta,k) = (1,\beta_0,1),\\
					P_\RR(\gamma_1), &(\ell,\beta,k) = (2,\beta_0, 0),\\
					0, &\text{otherwise}.
				\end{cases}
			\end{equation}
			
			\item (\emph{Zero})
			\begin{equation}\label{Open Zero}
				\OGWbar^{(\ell)}_{\beta_0,k}\left(\gamma_1\otimes\dots\otimes\gamma_\ell\right) = \begin{cases}
					-1, & (\ell,k) = (1,1)\text{ and } \gamma_1 = 1,\\
					P_\RR(\gamma_1\wedge\gamma_2), & (\ell,k) = (2,0),\\
					0, &\text{otherwise}.
				\end{cases}
			\end{equation}
			
			\item (\emph{Divisor}) If $\ell\geq 1$ and $|\gamma_\ell| = 2$, then
			\begin{equation} \label{Open Divisor}
				\OGWbar^{(\ell)}_{\beta,k}\left(\gamma_1\otimes\dots\otimes\gamma_\ell\right) = \int_\beta\gamma_\ell\cdot \OGWbar^{(\ell-1)}_{\beta,k}\left(\gamma_1\otimes\dots\otimes\gamma_{\ell-1}\right).
			\end{equation}
			
			\item (\emph{Wall-Crossing}) If $L$ is homologically trivial in $X$, then the following wall-crossing formula holds:
			\begin{equation}\label{Wall Crossing}
				\OGWbar^{(\ell+1)}_{\beta,k}(\Gamma_\diamond\otimes\gamma_1\otimes\dots\otimes\gamma_\ell) = -\OGWbar^{(\ell)}_{\beta,k+1}\left(\gamma_1\otimes\dots\otimes\gamma_\ell\right).
			\end{equation}
			
		\end{enumerate}
	\end{Definition}

	\begin{Remark}\label{Remark on Omega'}
		It is worthwhile to note that one can define the auxiliary superpotential $\Omega' = \Omega|_{s=0}$. Explicitly,
    \begin{equation*}
			\Omega' = \sum\limits_{\substack{\beta\in\Pi_{\geq 0}\\ \ell\in\ZZ_{\geq 0}}}\sum\limits_{\substack{\invecZN{r}\\ \sumsupto{\ell}{\vec{r}}}}
			T^\beta\prod\limits_{i=0}^{N}\frac{t_{N-i}^{r_{N-i}}}{(r_{N-i})!}
			\OGWbar^{(\ell)}_{\beta,0}\big(\bigotimes\limits_{i=0}^{N}\Gamma_{N-i}^{\otimes r_{N-i}}\big).
		\end{equation*}
Equation~\eqref{Open WDVV Inner Constraints} implies that $\Omega'$ satisfies the following PDE for all $0\leq a,b,c\leq N$:
		\begin{equation} \label{Open WDVV Inner Constraints Omega'}
			\sum\limits_{0\leq\mu,\nu\leq \kappa}\partial_a\partial_\mu\Omega'g^{\mu\nu}\partial_b\partial_c\partial_\nu\Phi = \sum\limits_{0\leq \mu,\nu\leq \kappa}\partial_a\partial_b\partial_\mu\Phi g^{\mu\nu}\partial_c\partial_\nu\Omega'.
		\end{equation}
		Thus, analogous axioms hold for $\Omega'$, with the open WDVV equations~ \eqref{Open WDVV Inner Constraints} and~\eqref{Open WDVV Boundary Constraint} replaced with~\eqref{Open WDVV Inner Constraints Omega'}, and the zero~\eqref{Open Zero} and fundamental class~\eqref{Open Fundamental Class} axioms replaced with:
\begin{enumerate}[label=(\roman*')]\setcounter{enumi}{4}
  \item (\emph{Fundamental Class}) Assume $\ell\geq 1$. Then
			\begin{equation*}
				\OGWbar^{(\ell)}_{\beta,0}\left(1\otimes\gamma_1\otimes\dots\otimes\gamma_{\ell-1}\right) = \begin{cases}
					P_\RR(\gamma_1), &(\ell,\beta) = (2,\beta_0),\\
					0, &\text{otherwise}.
				\end{cases}
			\end{equation*}
			
			\item (\emph{Zero})
			\begin{equation*}
				\OGWbar^{(\ell)}_{\beta_0,0}\left(\gamma_1\otimes\dots\otimes\gamma_\ell\right) = \begin{cases}
					P_\RR(\gamma_1\wedge\gamma_2), & \ell=2,\\
					0, &\text{otherwise}.
				\end{cases}
			\end{equation*}
\end{enumerate}
We refer to $\Omega'$ with its modified axioms as a \textit{specialized Frobenius superpotential}. We remark that the proof of Theorem~\ref{All OGW are computable 1 Fellowship of the OGW} holds for the specialized case, while Theorem~\ref{All OGW are computable 2 two OGW two computable} becomes redundant.

	\end{Remark}
	
	\begin{Definition}
		Let $\Omega$ be as in Definition \ref{def: Frobenius superpotential}. Let $\OGWbar^{(\ell)}_{\beta,k}\big(\gamma_1\otimes\dots\otimes\gamma_\ell\big)$ be a coefficient of $\Omega$, for some $\beta\in\Pi_{\geq 0}$, some $k,\ell\in\ZZ_{\geq 0}$, and some $\gamma_1,\dots,\gamma_\ell\in H^*(X,L;\RR)$. We say:
		\begin{itemize}
			\item $\beta$ is the \emph{degree} of the coefficient.
			\item $k$ is the \emph{number of boundary constraints} of the coefficient.
			\item $\ell$ is the \emph{number of inner constraints} of the coefficient.
			\item $\gamma_1,\dots,\gamma_\ell$ are the \emph{inner constraints} of the coefficient.
		\end{itemize}
	\end{Definition}
	
	\subsection{Geometric realization}\label{subsection:Geometric Realization}
Various methods are used in the literature to define open Gromov-Witten invariants in various setups.
Notably,
\cite{Solomon2006Intersection} and~\cite{Georgieva2016OGWDiskInvariants} define invariants directly via integration on moduli spaces;
\cite{Fukaya2011CountingPseudoHolomorphic} defines invariants via a superpotential using bounding cochains that are 1-forms;
\cite{SolomonTukachinsky2021PointLikeBoundingChain} and~\cite{SolomonTukachinsky2019relative}
define invariants via a superpotential using point-like bounding pairs.
Whenever open Gromov-Witten invariants are defined, it is reasonable to expect that their generating potential will be either a Frobenius superpotential or a specialized Frobenius superpotential (see Remark~\ref{Remark on Omega'}). This expectation is discussed, for example, in \cite{HahnWalkerAModelImplications}, and it holds for all the definitions that exist in current literature.
	
	\begin{Remark}
		Since $\Omega$ is meant to be a symplectic invariant, it is natural to add the following axiom.
		\begin{enumerate}[label=(\roman*)]\setcounter{enumi}{8}
			\item\label{Deformation invariance}
			(Deformation invariance) The invariants $\OGW$ remain constant under deformations of the symplectic form $\omega$ for which $L$ remains Lagrangian and the fixed subgroup $S_L \le H_2(X, L; \ZZ)$ remains in the kernel of $\omega$.
		\end{enumerate}
		We do not use this property in proving the main theorems, but it comes up in applications in Proposition~\ref{OGWs Vanish for n1 n2 are 1mod2 and pointlike}.
	\end{Remark}

	\section{Results}\label{Results}
	
	\subsection{Topological observations}
	It is easy to verify that $\omega(\varpi(\beta))=\omega(\beta)$. Thus, Proposition~\ref{hbar Proposition} implies the following observation.
	\begin{Lemma}\label{If varpi(beta)=0 then beta=0}
		Let $0\neq\beta\in H_2(X;\ZZ)$, and assume $\omega(\varpi(\beta))<\hbar$. Then for all $\gamma_1,\dots,\gamma_\ell\in H^*(X;\RR)$, the closed Gromov-Witten invariant  $\GW_\beta\left(\gamma_1\otimes\dots\otimes\gamma_\ell\right)$ vanishes.
	\end{Lemma}
	
	Now, recall the definition of the intersection matrix $(g_{ij})$ from Notation~\ref{gij matrix} and recall from Section \ref{Choice of Bases for Relative and Absolute Cohomologies} that $\kappa = \min\{K,N\} \le K$.
	
	\begin{Lemma}\label{g is block matrix when needed}
		Let $\gamma\in H^*(X,L;\RR)\cap H^*(X;\RR)$. Then
		\[ \sum\limits_{0 \leq\mu,\nu\leq \kappa}\left(\int_X \gamma\wedge\Delta_\mu\right)g^{\mu\nu}\Delta_\nu = \gamma. \]
	\end{Lemma}
	\begin{proof}
		By linearity, it is enough to show that for any $0\leq a\leq \kappa$, we have
		\[ \sum\limits_{0 \leq\mu,\nu\leq \kappa}\left(\int_X \Delta_a\wedge\Delta_\mu\right)g^{\mu\nu}\Delta_\nu = \Delta_a. \]
		There are two possible cases:
		\begin{enumerate}
			\item {$[L] = 0$ in $X$:} Then $\kappa = K$. In this case:
			\begin{align*}
				\sum\limits_{0 \leq\mu,\nu\leq \kappa}\left(\int_X \Delta_a\wedge\Delta_\mu\right)g^{\mu\nu}\Delta_\nu &= \sum\limits_{0 \leq\mu,\nu\leq K}\left(\int_X \Delta_a\wedge\Delta_\mu\right)g^{\mu\nu}\Delta_\nu\\
				&= \sum\limits_{0 \leq\mu,\nu\leq K}g_{a\mu}g^{\mu\nu}\Delta_\nu\\
				&= \sum\limits_{0 \leq \nu\leq K}\delta_{a\nu}\Delta_\nu\\
				&=\Delta_a,
			\end{align*}
			where $\delta_{a\nu}$ is the Kronecker delta.
			\item  {$[L] \neq 0$ in $X$:} Then $\kappa = K-1$ and $\Delta_K = \text{PD}([L])$. By Poincar\'e duality, for any $0\leq\mu\leq K-1$, we have
			\[ \int_X\Delta_\mu\wedge\Delta_K = \int_L\Delta_\mu = 0, \]
			since we identified $\Delta_\mu$ with $\Gamma_\mu\in H^*(X,L;\RR)$. Therefore,
			\begin{align*}
				\sum\limits_{0 \leq\mu,\nu\leq \kappa}\left(\int_X \Delta_a\wedge\Delta_\mu\right)g^{\mu\nu}\Delta_\nu &= \sum\limits_{0 \leq\mu,\nu\leq K-1}\left(\int_X \Delta_a\wedge\Delta_\mu\right)g^{\mu\nu}\Delta_\nu\\  &=\sum\limits_{0 \leq\mu,\nu\leq K}\left(\int_X \Delta_a\wedge\Delta_\mu\right)g^{\mu\nu}\Delta_\nu\\
				&= \sum\limits_{0 \leq\mu,\nu\leq K}g_{a\mu}g^{\mu\nu}\Delta_\nu\\
				&= \sum\limits_{0 \leq\nu\leq K}\delta_{a\nu}\Delta_\nu\\
				&=\Delta_a,
			\end{align*}
			where $\delta_{a\nu}$ is the Kronecker delta.
		\end{enumerate}
		In either case, the statement holds.
	\end{proof}
	
	\subsection{Power series expansion of the open WDVV relations}
	We will now expand the open WDVV relations (see Definition \ref{def: Frobenius superpotential}) into two relations between elements in $R$. These relations will allow us later to extract useful information on the coefficients of the Frobenius superpotential.
	
	Let $\Omega$ be a Frobenius superpotential, and write $\Omega$ as in equation \eqref{Open Generating Function}. For the first relation, let $a,b,c\in \{0,\dots,N\}$. Substitute the terms in \eqref{Closed Generating Function} and \eqref{Open Generating Function} into the first open WDVV equation, \eqref{Open WDVV Inner Constraints}, and get the following relation (using Notation~\ref{Notation Convention for Frobenius}).
	
	\begin{align}\label{Open WDVV Inner Constraints Polynomial Before Expansion}
		\sum\limits_{0\leq \mu,\nu\leq \kappa}\partial_a\partial_\mu\big(\sum\limits_{\substack{\beta\in \Pi_{\geq 0}\\ k,\ell\in\ZZ_{\geq 0}}}\sum\limits_{\substack{\invecZN{r}\\ \sumsupto{\ell}{\vec{r}}}}&T^\beta\frac{s^k}{k!}\prod\limits_{i=0}^{N}\frac{t_{N-i}^{r_{N-i}}}{r_{N-i}!}\OGWbar_{\beta,k}^{\parentheses{\ell}}\big(\bigotimes\limits_{i=0}^{N}\Gamma_{N-i}^{\otimes r_{N-i}}\big)\big)g^{\mu\nu}\\
		\notag\cdot\partial_b\partial_c\partial_\nu\big(\sum\limits_{\substack{\beta\in H_2(X;\ZZ)\\ \invecZK{r}}}(-1)^{w_s(\beta)}&T^{\varpi(\beta)}\prod\limits_{i=0}^K\frac{t_{K-i}^{r_{K-i}}}{r_{K-i}!}\GW_\beta\big(\bigotimes\limits_{i=0}^K\Delta_{K-i}^{\otimes r_{K-i}}\big)\big)\\
		\notag-\partial_a\partial_s\big(\sum\limits_{\substack{\beta\in\Pi_{\geq 0}\\ k,\ell\in\ZZ_{\geq 0}}}\sum\limits_{\substack{\invecZN{r}\\ \sumsupto{\ell}{\vec{r}}}}&T^\beta\frac{s^k}{k!}\prod\limits_{i=0}^N\frac{t_{N-i}^{r_{N-i}}}{r_{N-i}!}\OGWbar_{\beta,k}^{\parentheses{\ell}}\big(\bigotimes\limits_{i=0}^N\Gamma_{N-i}^{\otimes r_{N-i}}\big)\big)\\
		\notag\cdot\partial_b\partial_c\big(\sum\limits_{\substack{\beta\in\Pi_{\geq 0}\\ k,\ell\in\ZZ_{\geq 0}}}\sum\limits_{\substack{\invecZN{r}\\ \sumsupto{\ell}{\vec{r}}}}&T^\beta\frac{s^k}{k!}\prod\limits_{i=0}^N\frac{t_{N-i}^{r_{N-i}}}{r_{N-i}!}\OGWbar_{\beta,k}^{\parentheses{\ell}}\big(\bigotimes\limits_{i=0}^N\Gamma_{N-i}^{\otimes r_{N-i}}\big)\big)\\
		\notag=\sum\limits_{0\leq \mu,\nu\leq \kappa}\partial_a\partial_b\partial_\mu\big(\sum\limits_{\substack{\beta\in H_2(X;\ZZ)\\ \invecZK{r}}}(-1)^{w_s(\beta)}&T^{\varpi(\beta)}\prod\limits_{i=0}^K\frac{t_{K-i}^{r_{K-i}}}{r_{K-i}!}\GW_\beta\big(\bigotimes\limits_{i=0}^K\Delta_{K-i}^{\otimes r_{K-i}}\big)\big)g^{\mu\nu}\\
		\notag\cdot\partial_c\partial_\nu\big(\sum\limits_{\substack{\beta\in\Pi_{\geq 0}\\ k,\ell\in\ZZ_{\geq 0}}}\sum\limits_{\substack{\invecZN{r}\\ \sumsupto{\ell}{\vec{r}}}}&T^\beta\frac{s^k}{k!}\prod\limits_{i=0}^N\frac{t_{N-i}^{r_{N-i}}}{r_{N-i}!}\OGWbar_{\beta,k}^{\parentheses{\ell}}\big(\bigotimes\limits_{i=0}^N\Gamma_{N-i}^{\otimes r_{N-i}}\big)\big)\\
		\notag-\partial_a\partial_b\big(\sum\limits_{\substack{\beta\in\Pi_{\geq 0}\\ k,\ell\in\ZZ_{\geq 0}}}\sum\limits_{\substack{\invecZN{r}\\ \sumsupto{\ell}{\vec{r}}}}&T^\beta\frac{s^k}{k!}\prod\limits_{i=0}^N\frac{t_{N-i}^{r_{N-i}}}{r_{N-i}!}\OGWbar_{\beta,k}^{\parentheses{\ell}}\big(\bigotimes\limits_{i=0}^N\Gamma_{N-i}^{\otimes r_{N-i}}\big)\big)\\
		\notag\cdot\partial_s\partial_c\big(\sum\limits_{\substack{\beta\in\Pi_{\geq 0}\\ k,\ell\in\ZZ_{\geq 0}}}\sum\limits_{\substack{\invecZN{r}\\ \sumsupto{\ell}{\vec{r}}}}&T^\beta\frac{s^k}{k!}\prod\limits_{i=0}^N\frac{t_{N-i}^{r_{N-i}}}{r_{N-i}!}\OGWbar_{\beta,k}^{\parentheses{\ell}}\big(\bigotimes\limits_{i=0}^N\Gamma_{N-i}^{\otimes r_{N-i}}\big)\big).
	\end{align}
	Expanding \eqref{Open WDVV Inner Constraints Polynomial Before Expansion}, we get
	\begin{align}\label{Open WDVV Inner Constraints Polynomial After Expansion}
		\sum\limits_{0\leq\mu,\nu\leq \kappa}\sum\limits_{\substack{\beta_1\in \Pi_{\geq 0}\\ \beta_2\in H_2(X;\ZZ)\\ k,\ell\in\ZZ_{\geq 0}}}\sum\limits_{\substack{\invecZN{r}\\ \invecZK{p}\\ \sumsupto{\ell}{\vec{r}}}}&(-1)^{w_s(\beta_2)}T^{\beta_1+\varpi(\beta_2)}\frac{s^k}{k!}\prod\limits_{\substack{0\leq i\leq N\\ 0\leq j\leq K}}\frac{t_{N-i}^{r_{N-i}}t_{N-j}^{p_{N-j}}}{r_{N-i}!p_{N-j}!}\\
		\notag&\cdot\OGWbar^{\parentheses{\ell}}_{\beta_1,k}\big(\Gamma_\mu\otimes\Gamma_a\otimes\bigotimes\limits_{i=0}^N\Gamma_{N-i}^{\otimes r_{N-i}}\big)\\
		\notag&\cdot g^{\mu\nu}\GW_{\beta_2}\big(\Delta_\nu\otimes\Delta_c\otimes\Delta_b\otimes\bigotimes\limits_{j=0}^K\Delta_{K-j}^{\otimes p_{K-j}}\big)\\
		\notag-\sum\limits_{\substack{\beta_1,\beta_2\in \Pi_{\geq 0},\\ k_1,k_2,\ell_1,\ell_2\in\ZZ_{\geq 0}}}\sum\limits_{\substack{\vec{r},\vec{p}\in\left(\ZZ_{\geq 0}\right)^{N+1}\\ \sumsupto{\ell_1}{\vec{r}},\sumsupto{\ell_2}{\vec{p}}}}&T^{\beta_1+\beta_2}\frac{s^{k_1+k_2}}{k_1!k_2!}\prod\limits_{0\leq i,j\leq N}\frac{t_{N-i}^{r_{N-i}}t_{N-j}^{p_{N-j}}}{r_{N-i}!p_{N-j}!}\\
		\notag&\cdot\OGWbar_{\beta_1,k_1+1}^{\parentheses{\ell_1}}\big(\Gamma_a\otimes\bigotimes\limits_{i=0}^N\Gamma_{N-i}^{\otimes r_{N-i}}\big)\\
		\notag&\cdot\OGWbar_{\beta_2,k_2}^{\parentheses{\ell_2}}\big(\Gamma_c\otimes\Gamma_b\otimes\bigotimes\limits_{j=0}^N\Gamma_{N-j}^{\otimes p_{N-j}}\big)\\
		\notag= \sum\limits_{0\leq \mu,\nu\leq \kappa}\sum\limits_{\substack{\beta_1\in H_2(X;\ZZ)\\ \beta_2\in \Pi_{\geq 0}\\ k,\ell\in\ZZ_{\geq 0}}}\sum\limits_{\substack{\invecZK{r}\\ \invecZN{p}\\ \sumsupto{\ell}{\vec{p}}}}&(-1)^{w_s(\beta_1)}T^{\varpi(\beta_1)+\beta_2}\frac{s^k}{k!}\prod\limits_{\substack{0\leq i\leq K\\ 0\leq j\leq N}}\frac{t_{K-i}^{r_{K-i}}t_{N-j}^{p_{N-j}}}{r_{N-i}!p_{N-j}!}\\
		\notag&\cdot\GW_{\beta_1}\big(\Delta_\mu\otimes\Delta_b\otimes\Delta_a\otimes\bigotimes\limits_{i=0}^K\Delta_{K-i}^{\otimes r_{K-i}}\big)\\
		\notag&\cdot g^{\mu\nu}\OGWbar_{\beta_2,k}^{\parentheses{\ell}}\big(\Gamma_\nu\otimes\Gamma_c\otimes\bigotimes\limits_{j=0}^N\Gamma_{N-j}^{\otimes p_{N-j}}\big)\\
		\notag-\sum\limits_{\substack{\beta_1,\beta_2\in \Pi_{\geq 0},\\ k_1,k_2,\ell_1,\ell_2\in\ZZ_{\geq 0}}}\sum\limits_{\substack{\vec{r},\vec{p}\in\left(\ZZ_{\geq 0}\right)^{N+1}\\ \sumsupto{\ell_1}{\vec{r}},\sumsupto{\ell_2}{\vec{p}}}}&T^{\beta_1+\beta_2}\frac{s^{k_1+k_2}}{k_1!k_2!}\prod\limits_{0\leq i,j\leq N}\frac{t_{N-i}^{r_{N-i}}t_{N-j}^{p_{N-j}}}{r_{N-i}!p_{N-j}!}\\
		\notag&\cdot\OGWbar_{\beta_1,k_1}^{\parentheses{\ell_1}}\big(\Gamma_b\otimes \Gamma_a\otimes\bigotimes\limits_{i=0}^N\Gamma_{N-i}^{\otimes r_{N-i}}\big)\\
		\notag&\cdot\OGWbar_{\beta_2,k_2+1}^{\parentheses{\ell_2}}\big(\Gamma_c\otimes\bigotimes\limits_{j=0}^N\Gamma_{N-j}^{\otimes p_{N-j}}\big).
	\end{align}
	Similarly, for the second WDVV relation, let $a,b\in \{0,\dots,N\}$. Substitute the terms in \eqref{Closed Generating Function} and \eqref{Open Generating Function} into the first open WDVV relation, \eqref{Open WDVV Boundary Constraint}, and get
	\begin{align}\label{Open WDVV Boundary Constraints Polynomial Before Expansion}
		-&\partial_a\partial_s\big(\sum\limits_{\substack{\beta\in\Pi_{\geq 0}\\ k,\ell\in\ZZ_{\geq 0}}}\sum\limits_{\substack{\invecZN{r}\\ \sumsupto{\ell}{\vec{r}}}}T^\beta\frac{s^k}{k!}\prod\limits_{i=0}^N\frac{t_{N-i}^{r_{N-i}}}{r_{N-i}!}\OGWbar_{\beta,k}^{\parentheses{\ell}}\big(\bigotimes\limits_{i=0}^N\Gamma_{N-i}^{\otimes r_{N-i}}\big)\big)\\
		\notag&\qquad\cdot\partial_b\partial_s\big(\sum\limits_{\substack{\beta\in\Pi_{\geq 0}\\ k,\ell\in\ZZ_{\geq 0}}}\sum\limits_{\substack{\invecZN{r}\\ \sumsupto{\ell}{\vec{r}}}}T^\beta\frac{s^k}{k!}\prod\limits_{i=0}^N\frac{t_{N-i}^{r_{N-i}}}{r_{N-i}!}\OGWbar_{\beta,k}^{\parentheses{\ell}}\big(\bigotimes\limits_{i=0}^N\Gamma_{N-i}^{\otimes r_{N-i}}\big)\big)\\
		\notag=\sum\limits_{0\leq \mu,\nu\leq \kappa}&\partial_a\partial_b\partial_\mu\big(\sum\limits_{\substack{\beta\in H_2(X;\ZZ)\\ \invecZK{r}}}(-1)^{w_s(\beta)}T^{\varpi(\beta)}\prod\limits_{i=0}^K\frac{t_{K-i}^{r_{K-i}}}{r_{K-i}!}\GW_\beta\big(\bigotimes\limits_{i=0}^K\Delta_{K-i}^{\otimes r_{K-i}}\big)\big)g^{\mu\nu}\\
		\notag&\qquad\cdot\partial_s\partial_\nu\big(\sum\limits_{\substack{\beta\in\Pi_{\geq 0}\\ k,\ell\in\ZZ_{\geq 0}}}\sum\limits_{\substack{\invecZN{r}\\ \sumsupto{\ell}{\vec{r}}}}T^\beta\frac{s^k}{k!}\prod\limits_{i=0}^N\frac{t_{N-i}^{r_{N-i}}}{r_{N-i}!}\OGWbar_{\beta,k}^{\parentheses{\ell}}\big(\bigotimes\limits_{i=0}^N\Gamma_{N-i}^{\otimes r_{N-i}}\big)\big)\\
		\notag-&\partial_a\partial_b\big(\sum\limits_{\substack{\beta\in\Pi_{\geq 0}\\ k,\ell\in\ZZ_{\geq 0}}}\sum\limits_{\substack{\invecZN{r}\\ \sumsupto{\ell}{\vec{r}}}}T^\beta\frac{s^k}{k!}\prod\limits_{i=0}^N\frac{t_{N-i}^{r_{N-i}}}{r_{N-i}!}\OGWbar_{\beta,k}^{\parentheses{\ell}}\big(\bigotimes\limits_{i=0}^N\Gamma_{N-i}^{\otimes r_{N-i}}\big)\big)\\
		\notag&\qquad\cdot\partial_s^2\big(\sum\limits_{\substack{\beta\in\Pi_{\geq 0}\\ k,\ell\in\ZZ_{\geq 0}}}\sum\limits_{\substack{\invecZN{r}\\ \sumsupto{\ell}{\vec{r}}}}T^\beta\frac{s^k}{k!}\prod\limits_{i=0}^N\frac{t_{N-i}^{r_{N-i}}}{r_{N-i}!}\OGWbar_{\beta,k}^{\parentheses{\ell}}\big(\bigotimes\limits_{i=0}^N\Gamma_{N-i}^{\otimes r_{N-i}}\big)\big).
	\end{align}
	
	Expanding \eqref{Open WDVV Boundary Constraints Polynomial Before Expansion}, we get the following.
	\begin{align}\label{Open WDVV Boundary Constraints Polynomial After Expansion}
		-\sum\limits_{\substack{\beta_1,\beta_2\in \Pi_{\geq 0},\\ k_1,k_2,\ell_1,\ell_2\in\ZZ_{\geq 0}}}\sum\limits_{\substack{\vec{r},\vec{p}\in\left(\ZZ_{\geq 0}\right)^{N+1}\\ \sumsupto{\ell_1}{\vec{r}},\sumsupto{\ell_2}{\vec{p}}}}&T^{\beta_1+\beta_2}\frac{s^{k_1+k_2}}{k_1!k_2!}\prod\limits_{0\leq i,j\leq N}\frac{t_{N-i}^{r_{N-i}}t_{N-j}^{p_{N-j}}}{r_{N-i}!p_{N-j}!}\\
		\notag&\cdot\OGWbar_{\beta_1,k_1+1}^{\parentheses{\ell_1}}\big(\Gamma_a\otimes\bigotimes\limits_{i=0}^N\Gamma_{N-i}^{\otimes r_{N-i}}\big)\\
		\notag&\cdot\OGWbar_{\beta_2,k_2+1}^{\parentheses{\ell_2}}\big(\Gamma_b\otimes\bigotimes\limits_{j=0}^N\Gamma_{N-j}^{\otimes p_{N-j}}\big)\\
		= \sum\limits_{0\leq \mu,\nu\leq \kappa}\sum\limits_{\substack{\beta_1\in H_2(X;\ZZ)\\ \beta_2\in \Pi_{\geq 0}\\ k,\ell\in\ZZ_{\geq 0}}}\sum\limits_{\substack{\invecZK{r}\\ \invecZN{p}\\ \notag\sumsupto{\ell}{\vec{p}}}}&(-1)^{w_s(\beta_1)}T^{\varpi(\beta_1)+\beta_2}\frac{s^k}{k!}\prod\limits_{\substack{0\leq i\leq K\\ 0\leq j\leq N}}\frac{t_{K-i}^{r_{K-i}}t_{N-j}^{p_{N-j}}}{r_{N-i}!p_{N-j}!}\\
		\notag&\cdot\GW_{\beta_1}\big(\Delta_\mu\otimes\Delta_b\otimes\Delta_a\otimes\bigotimes\limits_{i=0}^K\Delta_{K-i}^{\otimes r_{K-i}}\big)\\
		\notag&\cdot g^{\mu\nu}\OGWbar_{\beta_2,k+1}^{\parentheses{\ell}}\big(\Gamma_\nu\otimes\bigotimes\limits_{j=0}^N\Gamma_{N-j}^{\otimes p_{N-j}}\big)\\
		-\sum\limits_{\substack{\beta_1,\beta_2\in \Pi_{\geq 0},\\ k_1,k_2,\ell_1,\ell_2\in\ZZ_{\geq 0}}}\sum\limits_{\substack{\vec{r},\vec{p}\in\left(\ZZ_{\geq 0}\right)^{N+1}\\ \notag\sumsupto{\ell_1}{\vec{r}},\sumsupto{\ell_2}{\vec{p}}}}&T^{\beta_1+\beta_2}\frac{s^{k_1+k_2}}{k_1!k_2!}\prod\limits_{0\leq i,j\leq N}\frac{t_{N-i}^{r_{N-i}}t_{N-j}^{p_{N-j}}}{r_{N-i}!p_{N-j}!}\\
		\notag&\cdot\OGWbar_{\beta_1,k_1}^{\parentheses{\ell_1}}\big(\Gamma_b\otimes \Gamma_a\otimes\bigotimes\limits_{i=0}^N\Gamma_{N-i}^{\otimes r_{N-i}}\big)\\
		\notag&\cdot\OGWbar_{\beta_2,k_2+2}^{\parentheses{\ell_2}}\big(\bigotimes\limits_{j=0}^N\Gamma_{N-j}^{\otimes p_{N-j}}\big).
	\end{align}
	
	Relations \eqref{Open WDVV Inner Constraints Polynomial After Expansion} and \eqref{Open WDVV Boundary Constraints Polynomial After Expansion} will be very useful to us throughout the rest of our proof, as they enable us to relate coefficients of $\Omega$ in various degrees to one another.
	
	\subsection{Proof of Theorems \ref{All OGW are computable 1 Fellowship of the OGW} and \ref{All OGW are computable 2 two OGW two computable}}
	We now turn to the main part of our proof. Let $\Omega$ be as in equation \eqref{Open Generating Function}, and assume that $\Omega$ is a Frobenius superpotential. Assume from now on that the second cohomology $H^2(X;\RR)$ generates the absolute cohomology ring $H^*(X;\RR)$.

	We define a partial order $<$ on the set: \[ B = \left\{(e,k,\ell,a_1,\dots, a_\ell)\in\RR_{\geq 0}\times\ZZ_{\geq 0}\times\left(\bigcup\limits_{\ell=0}^\infty \{\ell\}\times\left(\ZZ_{\geq 0}\right)^{\ell}\right)\;\middle|\;a_1\geq\dots\geq a_\ell\right\},\]
	where we denote $\ZZ^{0} = \{0\}$.
	\begin{Definition}
		Let $(e,k,\ell,a_1,\dots, a_\ell),(e',k',\ell',a'_1,\dots,a'_\ell)\in B$. We say that
		\[
		(e,k,\ell,a_1,\dots,a_\ell)<(e',k',\ell',a'_1,\dots,a'_{\ell'})
		\]
		if any of the following conditions is satisfied:
		\begin{itemize}
			\item $e<e'$.
			\item $e=e'$ and $k<k'$.
			\item $e = e'$, $k=k'$, and $\ell<\ell'$.
			\item $e = e'$, $k=k'$, $\ell=\ell'$, and there exists some $1\leq i\leq \ell$ such that  $a_i<a'_i$ and $a_j=a'_j$ for all $i+1\leq j\leq \ell$.
		\end{itemize}
		Note that the order $<$ is linear.
	\end{Definition}
	
	For convenience, for any element $\alpha = (e,k,\ell,a_1,\dots,a_\ell)\in B$, we say that
	\begin{itemize}
		\item $e$ is the \emph{energy} of $\alpha$.
		\item $k$ is the \emph{number of boundary constraints} of $\alpha$.
		\item $\ell$ is the \emph{number of interior  constraints} of $\alpha$.
		\item $a_i$ is the \emph{degree of the $i$'th interior constraint} of $\alpha$  for any $1\leq i\leq \ell$.
	\end{itemize}

	\begin{Definition}
		We say that a symbol $\OGWbar_{\beta,k}^{\parentheses{\ell}}\big(\Gamma_{i_1}\otimes\dots\otimes\Gamma_{i_\ell}\big)$, where $0\leq i_1,\dots,i_\ell\leq K$ and $\left|\Gamma_{i_1}\right|\geq\dots \ge \left|\Gamma_{i_\ell}\right|$, \textit{has an order} of $(\omega(\beta),k,\ell,\left|\Gamma_{i_1}\right|,\dots,\left|\Gamma_{i_\ell}\right|)$ (note that this assignment is on the symbol, and not the numerical value of $\OGWbar_{\beta,k}^{\parentheses{\ell}}\big(\Gamma_{i_1}\otimes\dots\otimes\Gamma_{i_\ell}\big)$).
	\end{Definition}

	The proof of Theorems \ref{All OGW are computable 1 Fellowship of the OGW} and \ref{All OGW are computable 2 two OGW two computable} is performed via four reduction steps. The reduction steps are as follows:
	\begin{enumerate}
		\item First, we show in Lemma \ref{All OGW are Computable lemma pt. 1} that all coefficients of $\Omega$ may be computed from the coefficients of the form  $\OGWbar_{\beta,k}^{\parentheses{\ell}}\big(\Gamma_{i_1}^{\otimes r_{i_1}}\otimes\dots\otimes\Gamma_{i_\ell}^{\otimes r_{i_\ell}}\big)$, where $\beta\in\Pi_{\geq 0}$, $k$ and $\ell$ are non-negative integers, $i_1,\dots,i_\ell\in\left\{0,\dots,K\right\}$ and $r_{i_1},\dots,r_{i_\ell}$ are non-negative integers.
		\item Next, we show in Lemma \ref{All OGW are Computable lemma pt. 2} that for any coefficient of $\Omega$ of the form $\OGWbar_{\beta,k}^{\parentheses{\ell}}\big(\Gamma_{i_1}\otimes\dots\otimes\Gamma_{i_\ell}\big)$ with $0\leq i_j\leq K$ for any $1\leq j\leq \ell$, if $\ell\geq 2$, then this coefficient  is computable via coefficients of order lower than $(\omega(\beta),k,\ell,\left|\Gamma_{i_1}\right|,\dots\left|\Gamma_{i_\ell}\right|)$ and closed Gromov-Witten invariants.
		\item After this, we show in Lemma \ref{All OGW are Computable lemma pt. 2, 2} that any coefficient of $\Omega$ of the form $\OGWbar_{\beta,k}^{\parentheses{\ell}}(\Gamma_{i_1}\otimes\dots\otimes\Gamma_{i_\ell})$, with $0\leq i_j\leq K$ for any $1\leq j\leq \ell$, if $\ell\geq 1$ and $k\geq 1$, then it is computable via coefficients of order lower than $(\omega(\beta),k,\ell,\left|\Gamma_{i_1}\right|,\dots,\left|\Gamma_{i_\ell}\right|)$ and closed Gromov-Witten invariants.
		\item Finally, we show in Lemma \ref{All OGW are Computable lemma pt. 3} that under the assumptions of Theorem \ref{All OGW are computable 2 two OGW two computable}, any coefficient of $\Omega$ of the form $\OGWbar_{\beta,k}^{\parentheses{0}}$ with $k\geq 2$ is computable via coefficients of order lower than $(\omega(\beta),k,0,0)$ and closed Gromov-Witten invariants.
	\end{enumerate}
	We use the four reduction steps, together with Theorem \ref{Konsevitch-Manin} that shows that the closed Gromov-Witten invariants are computable from an initial set as well, to deduce Theorems \ref{All OGW are computable 1 Fellowship of the OGW} and \ref{All OGW are computable 2 two OGW two computable}.
	
	The first reduction step in the proof states that we may assume without loss of generality that the inner constraints are basis elements of $H^*(X;\RR)$:
	\begin{Lemma}
		\label{All OGW are Computable lemma pt. 1} All coefficients of $\Omega$ are computable via coefficients of the form
		\[
		\OGWbar_{\beta,k}^{\parentheses{\ell}}(\Gamma_{i_1}\otimes\dots\otimes\Gamma_{i_\ell}),
		\]
		with $0\leq i_j\leq K$ for all $1\leq j\leq \ell$ and $\left|\Gamma_{i_1}\right|\geq\dots\geq\left|\Gamma_{i_\ell}\right|$.
	\end{Lemma}
	
	\begin{proof}
		Consider the coefficient $\OGWbar_{\beta,k}^{\parentheses{\ell}}\big(\Gamma_{i_1}\otimes\dots\otimes\Gamma_{i_\ell}\big)$
		with $0\leq i_1,\dots,i_\ell\leq N$.
		By \eqref{Choice of Bases for Relative and Absolute Cohomologies}, we have two cases to consider:
		\begin{itemize}
			\item If $L$ is homologically nontrivial in $X$. Then $N=K-1$, so $0\leq i_j\leq K$ for all $1\leq j\leq \ell$.
			\item Otherwise, $[L] = 0$ in $H^*(X;\RR)$. In this case, $\left\{\Gamma_1,\dots,\Gamma_N\right\} = \left\{\Delta_1,\dots,\Delta_K\right\}\cup\left\{\Gamma_\diamond\right\}$. We may assume by the symmetry property \eqref{Open Symmetry} and without loss of generality, that $\Gamma_{i_1}=\dots=\Gamma_{i_p} = \Gamma_\diamond$ for some $0\leq p\leq\ell$. In this case note that $0\leq i_{p+1},\dots,i_{\ell}\leq K$. Then use the wall-crossing formula \eqref{Wall Crossing} repeatedly to get
			\begin{equation*}\label{Wall Crossing Step}
				\OGWbar_{\beta,k}^{\parentheses{\ell}}\big(\Gamma_{i_1}\otimes\dots\otimes\Gamma_{i_\ell}\big) = (-1)^p\OGWbar_{\beta,k+p}^{\parentheses{\ell-p}}\left(\Gamma_{i_{p+1}}\otimes\dots\otimes\Gamma_{i_\ell}\right).
			\end{equation*}
		\end{itemize}
		By the symmetry axiom \eqref{Open Symmetry}, we may order $\Gamma_{i_1},\dots,\Gamma_{i_\ell}$ in such a way that $\left|\Gamma_{i_1}\right|\geq\dots\geq\left|\Gamma_{i_\ell}\right|$.
		Lemma \ref{All OGW are Computable lemma pt. 1} follows.
	\end{proof}
	
	\begin{Remark}
		Since we assume that $H^2(X;\RR)$ generates $H^*(X;\RR)$, the degrees of all basis elements $\Delta_0,\dots,\Delta_N$ are even. Note that as a corollary of this fact and the symmetry axiom \eqref{Open Symmetry}, the value of any coefficient $\OGWbar_{\beta,k}^{\parentheses{\ell}}\big(\Gamma_{i_1}\otimes\dots\otimes\Gamma_{i_\ell}\big)$ for $\beta\in \Pi_{\geq 0}$, $k,\ell\in\ZZ_{\geq 0}$, and $0\leq i_1,\dots,i_\ell\leq K$ does not change under any permutation of its inner constraints.
	\end{Remark}
	
	We go on to the next reduction step in our proof, which shows that all coefficients of $\Omega$ with $\ell\geq 2$ are computable via terms of lower order:
	\begin{Lemma}
		\label{All OGW are Computable lemma pt. 2} Let $\beta\in\Pi_{\geq 0}$,  $k,\ell\in\ZZ_{\geq 0}$, and $0\leq i_1,\dots,i_\ell\leq K$, and assume that $\ell\geq 2$ and $\left|\Gamma_{i_1}\right|\geq\dots\geq\left|\Gamma_{i_\ell}\right|$. Then the coefficient $\OGWbar_{\beta,k}^{\parentheses{\ell}}(\Gamma_{i_1}\otimes\dots\otimes\Gamma_{i_\ell})$ is computable via coefficients of $\Omega$ of order lower than $(\omega(\beta),k,\ell,\left|\Gamma_{i_1}\right|,\dots,\left|\Gamma_{i_\ell}\right|)$, with energy $\omega(\beta)$ or energy less than or equal to $\omega(\beta)-\hbar$, and closed Gromov-Witten invariants.
	\end{Lemma}
	
	\begin{proof}
		In this part of the proof, we will closely follow the ideas of Konsevitch and Manin in the proof of Theorem \ref{Konsevitch-Manin} (see \cite[Theorem 3.1]{KontsevitchManin1994} for more details). By our assumption, $H^*(X;\RR)$ is generated by $H^2(X;\RR)$. Then there are several possible cases:
		\begin{itemize}
			\item  {$|\Gamma_{i_\ell}| = 0$:} In this case, $\Gamma_{i_\ell}=1$. In this case,  $\OGWbar_{\beta,k}^{\parentheses{\ell}}(\Gamma_{i_1}\otimes\dots\otimes\Gamma_{i_\ell})$ is directly computable via the fundamental class axiom \eqref{Open Fundamental Class}.
			\item  {$|\Gamma_{i_\ell}| = 2$}: In this case, by the divisor axiom \eqref{Open Divisor}, we have that:
			\begin{equation*}
				\OGWbar_{\beta,k}^{\parentheses{\ell}}(\Gamma_{i_1}\otimes\dots\otimes\Gamma_{i_\ell}) = \int_\beta\Gamma_{i_\ell}\cdot\OGWbar_{\beta,k}^{\parentheses{\ell-1}}(\Gamma_{i_1}\otimes\dots\otimes\Gamma_{i_{\ell-1}})
			\end{equation*}
			And the coefficient is computable via another coefficient of a lower order.
			\item  {$|\Gamma_{i_\ell}|>2$}: In this case, since $H^*(X;\RR)$ is generated by $H^2(X;\RR)$, we may write $\Gamma_{i_\ell} = \sum\limits_j\delta'_j\wedge\delta_j$, with $\delta_j\in H^2(X;\RR)$ and $|\delta'_j|\geq 2$ for all $j$. By linearity of the homomorphism $\OGWbar_{\beta,k}^{\parentheses{\ell}}$, and the wedge product, we may assume without loss of generality that $\Gamma_{i_\ell} = \delta'\wedge\delta$, where $\delta\in H^2(X;\RR)$, $|\delta'|\geq 2$, and both $\delta,\delta'$ are basis elements in $H^*(X;\RR)$.
			
			Therefore, $\delta'=\Gamma_{b}$ and $\delta=\Gamma_c$ for some $0\leq b,c\leq K$. We take $a=i_1$ and apply the first open WDVV relation -- let us equate the coefficients of $T^\beta s^k \prod\limits_{j=2}^\ell t_{i_j}$ in \eqref{Open WDVV Inner Constraints Polynomial After Expansion}. This gives
			
			\begin{align}
				\label{Open WDVV for Gamma_a=Gamma_i_1, Gamma_b=delta', Gamma_c=delta 1}\sum\limits_{\substack{S_1\cupdot S_2 =\\ \left\{2,\dots,\ell-1\right\}}} \sum\limits_{0\leq\mu,\nu\leq \kappa}\sum\limits_{\substack{\beta_1\in\Pi_{\geq 0}\\ \beta_2\in H_2(X;\ZZ)\\ \beta_1+\varpi(\beta_2)=\beta}}\frac{(-1)^{w_s(\beta_2)}}{k!}&\OGWbar_{\beta_1,k}^{\parentheses{|S_1|+2}}\big(\Gamma_\mu\otimes\Gamma_a\otimes\bigotimes\limits_{p\in S_1}\Gamma_{i_p}\big)\\
				\notag g^{\mu\nu}&\GW_{\beta_2}\big(\Delta_b\otimes\Delta_c\otimes\Delta_\nu\otimes\bigotimes\limits_{q\in S_2}\Delta_{i_q}\big)\\
				-\label{Open WDVV for Gamma_a=Gamma_i_1, Gamma_b=delta', Gamma_c=delta 2}\sum\limits_{\substack{S_1\cupdot S_2 =\\ \left\{2,\dots,\ell-1\right\}}}\sum\limits_{\substack{k_1,k_2\in\ZZ_{\geq 0}\\ k_1+k_2=k}}\sum\limits_{\substack{\beta_1,\beta_2\in\Pi_{\geq 0}\\ \beta_1+\beta_2=\beta}}\frac{1}{k_1!k_2!}&\OGWbar_{\beta_1,k_1+1}^{\parentheses{|S_1|+1}}\big(\Gamma_a\otimes\bigotimes\limits_{p\in S_1}\Gamma_{i_p}\big)\\
				\notag&\OGWbar_{\beta_2,k_2}^{\parentheses{|S_2|+2}}\big(\Gamma_b\otimes\Gamma_c\otimes\bigotimes\limits_{q\in S_2}\Gamma_{i_q}\big) \\
				\label{Open WDVV for Gamma_a=Gamma_i_1, Gamma_b=delta', Gamma_c=delta 3}=\sum\limits_{\substack{S_1\cupdot S_2 =\\ \left\{2,\dots,\ell-1\right\}}}\sum\limits_{0\leq\mu,\nu\leq \kappa}\sum\limits_{\substack{\beta_1\in H_2(X;\ZZ)\\ \beta_2\in\Pi_{\geq 0}\\ \varpi(\beta_1)+\beta_2=\beta}}\frac{(-1)^{w_s(\beta_1)}}{k!}&\GW_{\beta_1}\big(\Delta_a\otimes\Delta_b\otimes\Delta_\mu\otimes\bigotimes_{p\in S_1}\Delta_{i_p}\big)\\
				\notag g^{\mu\nu}&\OGWbar_{\beta_2,k}^{\parentheses{|S_2|+2}}\big(\Gamma_c\otimes\Gamma_\nu\otimes\bigotimes\limits_{q\in S_2}\Gamma_{i_q}\big)\\
				-\label{Open WDVV for Gamma_a=Gamma_i_1, Gamma_b=delta', Gamma_c=delta 4}\sum\limits_{\substack{S_1\cupdot S_2 =\\ \left\{2,\dots,\ell-1\right\}}}\sum\limits_{\substack{k_1,k_2\in\ZZ_{\geq 0}\\ k_1+k_2=k}}\sum\limits_{\substack{\beta_1,\beta_2\in\Pi_{\geq 0}\\ \beta_1+\beta_2=\beta}}\frac{1}{k_1!k_2!}&\OGWbar_{\beta_1,k_1}^{\parentheses{|S_1|+2}}\big(\Gamma_a\otimes\Gamma_b\otimes\bigotimes\limits_{p\in S_1}\Gamma_{i_p}\big)\\
				\notag&\OGWbar_{\beta_2,k_2+1}^{\parentheses{|S_1|+1}}\big(\Gamma_c\otimes\bigotimes\limits_{q\in S_2}\Gamma_{i_q}\big).
			\end{align}
			
			We look for nonzero terms of order $(\omega(\beta), k,\ell, \left|\Gamma_{i_1}\right|,\dots,\left|\Gamma_{i_{\ell}}\right|)$ or higher in the above equation, by considering each summand separately. In \eqref{Open WDVV for Gamma_a=Gamma_i_1, Gamma_b=delta', Gamma_c=delta 1}, the higher order terms are of the form:
			\begin{equation}\label{Open WDVV for Gamma_a=Gamma_i_1, Gamma_b=delta', Gamma_c=delta 1,1}
				\sum\limits_{0\leq\mu,\nu\leq \kappa}\OGWbar_{\beta_1,k}^{\parentheses{|S_1|+2}}\big(\Gamma_a\otimes\Gamma_\mu\otimes\bigotimes\limits_{p\in S_1}\Gamma_{i_p}\big)g^{\mu\nu}\GW_{\tilde{\beta}}\big(\Delta_b\otimes\Delta_c\otimes\Delta_\nu\otimes\bigotimes\limits_{q\in S_2}\Delta_{i_q}\big),
			\end{equation}
			for some disjoint sets $S_1,S_2$ that satisfy $S_1\cupdot S_2 = \left\{2,\dots,\ell-1\right\}$ and some $\beta_1\in\Pi_{\geq 0},\tilde{\beta}\in H_2(X;\ZZ)$ such that $\beta_1+\varpi(\tilde{\beta})=\beta$. For the open coefficient $\OGWbar$ in \eqref{Open WDVV for Gamma_a=Gamma_i_1, Gamma_b=delta', Gamma_c=delta 1,1} to have order $(\omega(\beta),k,\ell,\left|\Gamma_{i_1}\right|,\dots,\left|\Gamma_{i_\ell}\right|)$ or higher, we must have $\omega(\beta_1) = \omega(\beta)$, which implies that:
			\[ \omega(\tilde{\beta}) = \omega(\varpi(\tilde{\beta})) = \omega(\beta) - \omega(\beta_1) = 0<\hbar. \]
			By Lemma \ref{If varpi(beta)=0 then beta=0}, if this term does not vanish, we deduce that $\tilde{\beta} = 0$ and therefore $\beta_1 = \beta$.

			Since $\tilde{\beta}=0$, then by the closed zero axiom \eqref{Closed Zero}, the value $\GW_{\tilde{\beta}}\big(\Delta_b\otimes\Delta_c\otimes\Delta_\nu\otimes\bigotimes\limits_{q\in S_2}\Delta_{i_q}\big)$ is nonzero only if $|S_2|=0$, implying $S_1=\{2,\dots,\ell-1\}$. In this case \eqref{Open WDVV for Gamma_a=Gamma_i_1, Gamma_b=delta', Gamma_c=delta 1,1} reduces by Lemma \ref{g is block matrix when needed} to
			\begin{align*}
				\sum\limits_{0\leq\mu,\nu\leq \kappa} \OGWbar_{\beta,k}^{\parentheses{\ell}}&\big(\Gamma_a\otimes\Gamma_\mu\otimes\Gamma_{i_2}\otimes\dots\otimes\Gamma_{i_{\ell-1}}\big)g^{\mu\nu}\GW_{0}\big(\Delta_b\otimes\Delta_c\otimes\Delta_\nu\big) \\
				=\sum\limits_{0\leq \mu,\nu\leq \kappa} &\OGWbar_{\beta,k}^{\parentheses{\ell}}\big(\Gamma_a\otimes\Gamma_\mu\otimes\Gamma_{i_2}\otimes\dots\otimes\Gamma_{i_{\ell-1}}\big)g^{\mu\nu}\int_X\Delta_b\wedge\Delta_c\wedge\Delta_\nu \\
				=&\OGWbar_{\beta,k}^{\parentheses{\ell}}\big(\Gamma_a\otimes\big(\sum\limits_{0\leq \mu,\nu\leq \kappa}\int_X\Delta_b\wedge\Delta_c\wedge\Delta_\nu\big)g^{\mu\nu}\Gamma_\mu\otimes\Gamma_{i_2}\otimes\dots\otimes\Gamma_{i_\ell}\big)\\
				=&\OGWbar_{\beta,k}^{\parentheses{\ell}}\big(\Gamma_a\otimes\Delta_b\wedge\Delta_c\otimes\Gamma_{i_2}\otimes\dots\otimes\Gamma_{i_{\ell-1}}\big)\\
				=&\OGWbar_{\beta,k}^{\parentheses{\ell}}\big(\Gamma_{i_1}\otimes\Gamma_{i_2}\otimes\dots\otimes\Gamma_{i_{\ell-1}}\otimes\delta'\wedge\delta\big)\\
				=&\OGWbar_{\beta,k}^{\parentheses{\ell}}\big(\Gamma_{i_1}\otimes\Gamma_{i_2}\otimes\dots\otimes\Gamma_{i_{\ell-1}}\otimes\Gamma_{i_\ell}\big).
			\end{align*}
			The last expression is the coefficient we are looking to compute.
			Thus, the only non-vanishing higher order open term in \eqref{Open WDVV for Gamma_a=Gamma_i_1, Gamma_b=delta', Gamma_c=delta 1} is $\OGWbar_{\beta,k}^{\parentheses{\ell}}\big(\Gamma_{i_1}\otimes\Gamma_{i_2}\otimes\dots\otimes\Gamma_{i_{\ell-1}}\otimes\Gamma_{i_\ell}\big)$.
			
			In \eqref{Open WDVV for Gamma_a=Gamma_i_1, Gamma_b=delta', Gamma_c=delta 2}, the higher order terms are of the form
			\begin{align}
				\label{Open WDVV for Gamma_a=Gamma_i_1, Gamma_b=delta', Gamma_c=delta 2,0}&\OGWbar_{\beta_1,k_1+1}^{\parentheses{|S_1|+1}}\big(\Gamma_a\otimes\bigotimes\limits_{p\in S_1}\Gamma_{i_p}\big)\OGWbar_{\beta_2,k_2}^{\parentheses{|S_2|+2}}\big(\Gamma_b\otimes\Gamma_c\otimes\bigotimes\limits_{q\in S_2}\Gamma_{i_q}\big),
			\end{align}
			for some $\beta_1,\beta_2\in\Pi_{\geq 0}$ such that $\beta_1+\beta_2=\beta$, some disjoint sets $S_1,S_2$ that satisfy $S_1\cupdot S_2 = \left\{2,\dots,\ell-1\right\}$ and some $k_1,k_2\in\ZZ_{\geq 0}$ that satisfy $k_1+k_2=k$. If one of the factors in \eqref{Open WDVV for Gamma_a=Gamma_i_1, Gamma_b=delta', Gamma_c=delta 2,0} has order higher or equal to $(\omega(\beta),k,\ell,\left|\Gamma_{i_1}\right|,\dots,\left|\Gamma_{i_\ell}\right|)$, then $\omega(\beta_1)=\omega(\beta)$ or $\omega(\beta_2)=\omega(\beta)$. Therefore, $\omega(\beta_2)=0$ or $\omega(\beta_1)=0$. By the energy gap axiom \eqref{Energy Gap}, this term does not vanish only if $\beta_1=\beta_0$ or $\beta_2=\beta_0$. Therefore higher-order contributions to term \eqref{Open WDVV for Gamma_a=Gamma_i_1, Gamma_b=delta', Gamma_c=delta 2,0} reduce to
			\begin{align}
				\label{Open WDVV for Gamma_a=Gamma_i_1, Gamma_b=delta', Gamma_c=delta 2,1}&\OGWbar_{\beta,k_1+1}^{\parentheses{|S_1|+1}}\big(\Gamma_a\otimes\bigotimes\limits_{p\in S_1}\Gamma_{i_p}\big)\OGWbar_{\beta_0,k_2}^{\parentheses{|S_2|+2}}\big(\Gamma_b\otimes\Gamma_c\otimes\bigotimes\limits_{q\in S_2}\Gamma_{i_q}\big),
			\end{align}
			or
			\begin{align}
				\label{Open WDVV for Gamma_a=Gamma_i_1, Gamma_b=delta', Gamma_c=delta 2,2}&\OGWbar_{\beta_0,k_1+1}^{\parentheses{|S_1|+1}}\big(\Gamma_a\otimes\bigotimes\limits_{p\in S_1}\Gamma_{i_p}\big)\OGWbar_{\beta,k_2}^{\parentheses{|S_2|+2}}\big(\Gamma_b\otimes\Gamma_c\otimes\bigotimes\limits_{q\in S_2}\Gamma_{i_q}\big).
			\end{align}
			
			For \eqref{Open WDVV for Gamma_a=Gamma_i_1, Gamma_b=delta', Gamma_c=delta 2,1}, by the zero axiom \eqref{Open Zero}, a nonzero term must have $k_2=0$ and $|S_2|=0$. In this case, $S_1=\left\{2,\dots,\ell-1\right\}$ and $k_1=k$. Therefore \eqref{Open WDVV for Gamma_a=Gamma_i_1, Gamma_b=delta', Gamma_c=delta 2,1} reduces to
			\begin{equation*}
				\begin{aligned}
					\OGWbar_{\beta,k+1}^{\parentheses{\ell-1}}&\big(\Gamma_a\otimes\Gamma_{i_2}\otimes\dots\otimes\Gamma_{i_{\ell-1}}\big)\OGWbar_{\beta_0,0}^{\parentheses{2}}\big(\Gamma_b\otimes\Gamma_c\big) \\
					=&\OGWbar_{\beta,k+1}^{\parentheses{\ell-1}}\big(\Gamma_a\otimes\Gamma_{i_2}\otimes\dots\otimes\Gamma_{i_{\ell-1}}\big)P_\RR\left(\Gamma_b\wedge\Gamma_c\right) \\
					=&0,
				\end{aligned}
			\end{equation*}
			where the last equality holds because $P_\RR$ vanishes on $H^*(X;\RR)$. For \eqref{Open WDVV for Gamma_a=Gamma_i_1, Gamma_b=delta', Gamma_c=delta 2,2}, note that $|\Gamma_a|\geq 2>0$, so $\Gamma_a\neq 1$, and also that $k_1+1\geq 1$. By the zero axiom \eqref{Open Zero}, then, the value $\OGWbar_{\beta_0,k_1+1}^{\parentheses{|S_1|+1}}\big(\Gamma_a\otimes\bigotimes\limits_{p\in S_1}\Gamma_{i_p}\big)$ vanishes and thus \eqref{Open WDVV for Gamma_a=Gamma_i_1, Gamma_b=delta', Gamma_c=delta 2,2} reduces to zero. We thus conclude that there are no higher order terms in \eqref{Open WDVV for Gamma_a=Gamma_i_1, Gamma_b=delta', Gamma_c=delta 2}.
			
			In \eqref{Open WDVV for Gamma_a=Gamma_i_1, Gamma_b=delta', Gamma_c=delta 3}, the higher order terms are
			\begin{align}
				\label{Open WDVV for Gamma_a=Gamma_i_1, Gamma_b=delta', Gamma_c=delta 3,1}\sum\limits_{0 \leq\mu,\nu\leq K} &\GW_{\tilde{\beta}}\big(\Delta_a\otimes\Delta_b\otimes\Delta_\mu\otimes\bigotimes_{p\in S_1}\Delta_{i_p}\big)g^{\mu\nu}\OGWbar_{\beta_2,k}^{\parentheses{|S_2|+2}}\big(\Gamma_c\otimes\Gamma_\nu\otimes\bigotimes\limits_{q\in S_2}\Gamma_{i_q}\big),
			\end{align}
			for some disjoint sets $S_1,S_2$ that satisfy $S_1\cupdot S_2 = \left\{2,\dots,\ell-1\right\}$ and some $\tilde{\beta}\in H_2(X;\ZZ),\beta_2\in\Pi_{\geq 0}$ that satisfy $\varpi(\tilde{\beta})+\beta_2=\beta$. For the open factor $\OGWbar$ in \eqref{Open WDVV for Gamma_a=Gamma_i_1, Gamma_b=delta', Gamma_c=delta 3,1} to have order $(\omega(\beta),k,\ell,\left|\Gamma_{i_1}\right|,\dots,\left|\Gamma_{i_\ell}\right|)$ or higher, we must have $\omega(\beta_2) = \omega(\beta)$, and therefore
			\[
			\omega(\tilde{\beta}) = 0<\hbar.
			\]
			By Lemma \ref{If varpi(beta)=0 then beta=0}, we deduce that the  term in \eqref{Open WDVV for Gamma_a=Gamma_i_1, Gamma_b=delta', Gamma_c=delta 3,1} has higher order and the closed factor does not vanish only if $\tilde{\beta} = 0$.
			
			Since $\tilde{\beta}=0$, by the closed zero axiom \eqref{Closed Zero}, the value $\GW_{\tilde{\beta}}\big(\Delta_a\otimes\Delta_b\otimes\Delta_\mu\otimes\bigotimes\limits_{p\in S_1}\Delta_{i_p}\big)$ is nonzero only if $|S_1|=0$, implying $S_2 = \left\{2,\dots,\ell-1\right\}$. In this case \eqref{Open WDVV for Gamma_a=Gamma_i_1, Gamma_b=delta', Gamma_c=delta 3,1} reduces, by Lemma \ref{g is block matrix when needed}, to
			\begin{align*}
				\sum\limits_{0\leq\mu,\nu\leq K} \GW_{0}&\big(\Delta_a\otimes\Delta_b\otimes\Delta_\mu\big)g^{\mu\nu}\OGWbar_{\beta,k}^{\parentheses{\ell}}\big(\Gamma_c\otimes\Gamma_\nu\otimes\Gamma_{i_2}\otimes\dots\otimes\Gamma_{i_{\ell-1}}\big)\\
				=&\sum\limits_{0\leq \mu,\nu\leq K} \big(\int_X\Delta_a\wedge\Delta_b\wedge\Delta_\mu\big) g^{\mu\nu}\OGWbar_{\beta,k}^{\parentheses{\ell}}\big(\Gamma_c\otimes\Gamma_\nu\otimes\Gamma_{i_2}\otimes\dots\otimes\Gamma_{i_{\ell-1}}\big)\\
				=&\OGWbar_{\beta,k}^{\parentheses{\ell}}\big(\Gamma_c\otimes\big(\sum\limits_{\mu,\nu}\int_X\Delta_a\wedge\Delta_b\wedge\Delta_\mu\big) g^{\mu\nu}\Gamma_\nu\otimes\Gamma_{i_2}\otimes\dots\otimes\Gamma_{i_{\ell-1}}\big)\\
				=&\OGWbar_{\beta,k}^{\parentheses{\ell}}\big(\Gamma_c\otimes\Delta_a\wedge\Delta_b\otimes\Gamma_{i_2}\otimes\dots\otimes\Gamma_{i_{\ell-1}}\big)\\
				=&\OGWbar_{\beta,k}^{\parentheses{\ell}}\big(\delta\otimes\Gamma_{i_1}\wedge\delta'\otimes\Gamma_{i_2}\otimes\dots\otimes\Gamma_{i_{\ell-1}}\big)\\
				=&\int_\beta\delta\cdot\OGWbar_{\beta,k}^{\parentheses{\ell-1}}\big(\Gamma_{i_1}\wedge\delta'\otimes\Gamma_{i_2}\otimes\dots\otimes\Gamma_{i_{\ell-1}}\big),
			\end{align*}
			where the last equality comes from the divisor axiom \eqref{Open Divisor}. Note that
			\[ \OGWbar_{\beta,k}^{\parentheses{\ell-1}}\big(\Gamma_{i_1}\wedge\delta'\otimes\Gamma_{i_2}\otimes\dots\otimes\Gamma_{i_{\ell-1}}\big) \]
			has an order that is strictly lower than the order of $\OGWbar_{\beta,k}^{\parentheses{\ell}}\big(\Gamma_{i_1}\otimes\dots\otimes\Gamma_{i_\ell}\big)$.\\
			We conclude that all open factors of higher order in \eqref{Open WDVV for Gamma_a=Gamma_i_1, Gamma_b=delta', Gamma_c=delta 3} are computable via lower order terms.
			
			Finally, for \eqref{Open WDVV for Gamma_a=Gamma_i_1, Gamma_b=delta', Gamma_c=delta 4}, the higher order terms are of the form
			\begin{align}
				\label{Open WDVV for Gamma_a=Gamma_i_1, Gamma_b=delta', Gamma_c=delta 4,0}&\OGWbar_{\beta_1,k_1}^{\parentheses{|S_1|+2}}\big(\Gamma_a\otimes\Gamma_b\otimes\bigotimes\limits_{p\in S_1}\Gamma_{i_p}\big)\OGWbar_{\beta_2,k_2+1}^{\parentheses{|S_2|+1}}\big(\Gamma_c\otimes\bigotimes\limits_{q\in S_2}\Gamma_{i_q}\big),
			\end{align}
			for some $\beta_1,\beta_2\in\Pi_{\geq 0}$ such that $\beta_1+\beta_2=\beta$, some disjoint sets $S_1,S_2$ that satisfy $S_1\cupdot S_2 = \left\{2,\dots,\ell-1\right\}$, and some $k_1,k_2\in\ZZ_{\geq 0}$ that satisfy $k_1+k_2=k$. If either factor in \eqref{Open WDVV for Gamma_a=Gamma_i_1, Gamma_b=delta', Gamma_c=delta 4,0} has order higher or equal to $(\omega(\beta),k,\ell,\left|\Gamma_{i_1}\right|,\dots,\left|\Gamma_{i_\ell}\right|)$, then $\omega(\beta_1)=\omega(\beta)$ or $\omega(\beta_2) = \omega(\beta)$, and therefore $\omega(\beta_2)=0<\hbar$ or $\omega(\beta_1)=0<\hbar$. By the energy gap axiom \eqref{Energy Gap}, we deduce that the product does not vanish only if $\beta_2=\beta_0$ or $\beta_1 = \beta_0$. Therefore term \eqref{Open WDVV for Gamma_a=Gamma_i_1, Gamma_b=delta', Gamma_c=delta 4,0} reduces to
			\begin{align}
				\label{Open WDVV for Gamma_a=Gamma_i_1, Gamma_b=delta', Gamma_c=delta 4,1}&\OGWbar_{\beta,k_1}^{\parentheses{|S_1|+2}}\big(\Gamma_a\otimes\Gamma_b\otimes\bigotimes\limits_{p\in S_1}\Gamma_{i_p}\big)\OGWbar_{\beta_0,k_2+1}^{\parentheses{|S_2|+1}}\big(\Gamma_c\otimes\bigotimes\limits_{q\in S_2}\Gamma_{i_q}\big),
			\end{align}
			or
			\begin{align}
				\label{Open WDVV for Gamma_a=Gamma_i_1, Gamma_b=delta', Gamma_c=delta 4,2}&\OGWbar_{\beta_0,k_1}^{\parentheses{|S_1|+2}}\big(\Gamma_a\otimes\Gamma_b\otimes\bigotimes\limits_{p\in S_1}\Gamma_{i_p}\big)\OGWbar_{\beta,k_2+1}^{\parentheses{|S_2|+1}}\big(\Gamma_c\otimes\bigotimes\limits_{q\in S_2}\Gamma_{i_q}\big).
			\end{align}
			
			In \eqref{Open WDVV for Gamma_a=Gamma_i_1, Gamma_b=delta', Gamma_c=delta 4,1}, note that $|\Gamma_c|=2>0$, so $\Gamma_c\neq 1$. Therefore, by the zero axiom \eqref{Open Zero}, the value $\OGWbar_{\beta_0,k_2}^{\parentheses{|S_2|+1}}\big(\Gamma_c\otimes\bigotimes\limits_{q\in S_2}\Gamma_{i_q}\big)$ vanishes. We thus conclude that \eqref{Open WDVV for Gamma_a=Gamma_i_1, Gamma_b=delta', Gamma_c=delta 4,1} reduces to zero.
			
			In \eqref{Open WDVV for Gamma_a=Gamma_i_1, Gamma_b=delta', Gamma_c=delta 4,2}, if $\OGWbar_{\beta_0,k_1}^{\parentheses{|S_1|+2}}\big(\Gamma_a\otimes\Gamma_b\otimes\bigotimes\limits_{p\in S_1}\Gamma_{i_p}\big)$ does not vanish, then by the zero axiom \eqref{Open Zero} we must have $k_1=0$ and $|S_1|=0$, and therefore $S_2 = \left\{2,\dots,\ell-1\right\}$ and $k_2=k$. Thus, \eqref{Open WDVV for Gamma_a=Gamma_i_1, Gamma_b=delta', Gamma_c=delta 4,2} reduces to
			\begin{align*}
				&\OGWbar_{\beta_0,0}^{\parentheses{2}}\big(\Gamma_a\otimes\Gamma_b\big)\OGWbar_{\beta,k+1}^{\parentheses{\ell-1}}\big(\Gamma_c\otimes\Gamma_{i_2}\otimes\dots\otimes\Gamma_{i_{\ell-1}}\big) \\
				&=P_\RR\big(\Gamma_a\wedge\Gamma_b\big)\OGWbar_{\beta,k+1}^{\parentheses{\ell-1}}\big(\Gamma_c\otimes\Gamma_{i_2}\otimes\dots\otimes\Gamma_{i_{\ell-1}}\big) \\
				&=0,
			\end{align*}
			where the last equality comes from the fact that $P_\RR$ vanishes on $H^*(X;\RR)$. We conclude that there are no higher order factors in \eqref{Open WDVV for Gamma_a=Gamma_i_1, Gamma_b=delta', Gamma_c=delta 4}.
			
			We make a useful observation, though it is not needed for the purposes of our lemma, that in \eqref{Open WDVV for Gamma_a=Gamma_i_1, Gamma_b=delta', Gamma_c=delta 1} and \eqref{Open WDVV for Gamma_a=Gamma_i_1, Gamma_b=delta', Gamma_c=delta 3}, the are no degree-zero contributions, as detailed below.
			
			In \eqref{Open WDVV for Gamma_a=Gamma_i_1, Gamma_b=delta', Gamma_c=delta 1}, the only term with degree $\beta_0$ is of the form:
			\begin{equation}\label{Open WDVV for Gamma_a=Gamma_i_1, Gamma_b=delta', Gamma_c=delta 1,1 OMG Vanishing}
				\sum\limits_{0\leq\mu,\nu\leq \kappa}\OGWbar_{\beta_0,k}^{\parentheses{|S_1|+2}}\big(\Gamma_a\otimes\Gamma_\mu\otimes\bigotimes\limits_{p\in S_1}\Gamma_{i_p}\big)g^{\mu\nu}\GW_{\tilde{\beta}}\big(\Delta_b\otimes\Delta_c\otimes\Delta_\nu\otimes\bigotimes\limits_{q\in S_2}\Delta_{i_q}\big),
			\end{equation}
			for some disjoint sets $S_1,S_2$ that satisfy $S_1\cupdot S_2 = \left\{2,\dots,\ell-1\right\}$ and some $\tilde{\beta}\in H_2(X;\ZZ)$ such that $\varpi(\tilde{\beta})=\beta$. Since $|S_1|+2\geq 2$, we deduce by the zero axiom \eqref{Open Zero} that this term vanishes.
			
			In \eqref{Open WDVV for Gamma_a=Gamma_i_1, Gamma_b=delta', Gamma_c=delta 3}, the only term with degree $\beta_0$ is of the form:
			\begin{align}
				\label{Open WDVV for Gamma_a=Gamma_i_1, Gamma_b=delta', Gamma_c=delta 3,1 OMG Vanishing}\sum\limits_{0 \leq\mu,\nu\leq K} &\GW_{\tilde{\beta}}\big(\Delta_a\otimes\Delta_b\otimes\Delta_\mu\otimes\bigotimes_{p\in S_1}\Delta_{i_p}\big)g^{\mu\nu}\OGWbar_{\beta_0,k}^{\parentheses{|S_2|+2}}\big(\Gamma_c\otimes\Gamma_\nu\otimes\bigotimes\limits_{q\in S_2}\Gamma_{i_q}\big),
			\end{align}
			for some disjoint sets $S_1,S_2$ that satisfy $S_1\cupdot S_2 = \left\{2,\dots,\ell-1\right\}$ and some $\tilde{\beta}\in H_2(X;\ZZ)$ that satisfies $\varpi(\tilde{\beta})=\beta$. Since $|S_2|+2\geq 2$, we deduce by the zero axiom \eqref{Open Zero} that this term vanishes.
			
			In summary, we get the relation
			\begin{align}\label{summary of all OGW are computable lemma pt. 2} \OGWbar_{\beta,k}^{\parentheses{\ell}}(\Gamma_{i_1}\otimes\dots\otimes\Gamma_{i_\ell})\hspace{7em}& \\
				\notag= -\sum\limits_{\substack{S_1\cupdot S_2 =\\ \left\{2,\dots,\ell-1\right\}}} \sum\limits_{0\leq\mu,\nu\leq \kappa}\sum\limits_{\substack{\beta_1\in\Pi_{> 0}\\ \beta_2\in H_2(X;\ZZ)\\ \omega(\beta_2)>0\\ \beta_1+\varpi(\beta_2)=\beta}}&(-1)^{w_s(\beta_2)}\OGWbar_{\beta_1,k}^{\parentheses{|S_1|+2}}\big(\Gamma_\mu\otimes\Gamma_a\otimes\bigotimes\limits_{p\in S_1}\Gamma_{i_p}\big)\\
				\notag g^{\mu\nu}&\GW_{\beta_2}\big(\Delta_b\otimes\Delta_c\otimes\Delta_\nu\otimes\bigotimes\limits_{q\in S_2}\Delta_{i_q}\big)\\
				\notag-\sum\limits_{\substack{S_1\cupdot S_2 =\\ \left\{2,\dots,\ell-1\right\}}}\sum\limits_{\substack{k_1,k_2\in\ZZ_{\geq 0}\\ k_1+k_2=k}}\sum\limits_{\substack{\beta_1,\beta_2\in\Pi_{> 0}\\ \beta_1+\beta_2=\beta}}\binom{k}{k_1}&\OGWbar_{\beta_1,k_1+1}^{\parentheses{|S_1|+1}}\big(\Gamma_a\otimes\bigotimes\limits_{p\in S_1}\Gamma_{i_p}\big)\\
				\notag&\OGWbar_{\beta_2,k_2}^{\parentheses{|S_2|+2}}\big(\Gamma_b\otimes\Gamma_c\otimes\bigotimes\limits_{q\in S_2}\Gamma_{i_q}\big) \\
				\notag+\sum\limits_{\substack{S_1\cupdot S_2 =\\ \left\{2,\dots,\ell-1\right\}}}\sum\limits_{0\leq\mu,\nu\leq \kappa}\sum\limits_{\substack{\beta_1\in H_2(X;\ZZ)\\ \beta_2\in\Pi_{> 0}\\ \varpi(\beta_1)+\beta_2=\beta}}(-1)^{w_s(\beta_1)}&\GW_{\beta_1}\big(\Delta_a\otimes\Delta_b\otimes\Delta_\mu\otimes\bigotimes_{p\in S_1}\Delta_{i_p}\big)\\
				\notag g^{\mu\nu}&\OGWbar_{\beta_2,k}^{\parentheses{|S_2|+2}}\big(\Gamma_c\otimes\Gamma_\nu\otimes\bigotimes\limits_{q\in S_2}\Gamma_{i_q}\big)\\
				\notag-\sum\limits_{\substack{S_1\cupdot S_2 =\\ \left\{2,\dots,\ell-1\right\}}}\sum\limits_{\substack{k_1,k_2\in\ZZ_{\geq 0}\\ k_1+k_2=k}}\sum\limits_{\substack{\beta_1,\beta_2\in\Pi_{> 0}\\ \beta_1+\beta_2=\beta}}\binom{k}{k_1}&\OGWbar_{\beta_1,k_1}^{\parentheses{|S_1|+2}}\big(\Gamma_a\otimes\Gamma_b\otimes\bigotimes\limits_{p\in S_1}\Gamma_{i_p}\big)\\
				\notag&\OGWbar_{\beta_2,k_2+1}^{\parentheses{|S_1|+1}}\big(\Gamma_c\otimes\bigotimes\limits_{q\in S_2}\Gamma_{i_q}\big),
			\end{align}
			where the right-hand side of \eqref{summary of all OGW are computable lemma pt. 2} contains only terms of a lower order.
		\end{itemize}

		In each case, we get a relation of the sort
		\begin{equation*}
			\OGWbar_{\beta,k}^{\parentheses{\ell}}\left(\Gamma_{i_1}\otimes\dots\otimes\Gamma_{i_\ell}\right) = \text{a combination of terms of lower order}.
		\end{equation*}
		We also note that the energy gap axiom \eqref{Energy Gap} implies that such terms of lower order are of energy $\omega(\beta)$ or have energy lower than or equal to $\omega(\beta)-\hbar$.
		
		We conclude that the coefficient $\OGWbar_{\beta,k}^{\parentheses{\ell}}\left(\Gamma_{i_1}\otimes\dots\otimes\Gamma_{i_\ell}\right)$ is computable via coefficients of $\Omega$ of a lower order, with energy $\omega(\beta)$ or energy less than or equal to $\omega(\beta)-\hbar$, and closed Gromov-Witten invariants.
	\end{proof}	
	
	The third reduction step in our proof shows that all coefficients of $\Omega$ with $\ell\geq 1$ and $k\geq 1$ are computable via terms of lower order.
	\begin{Lemma}
		\label{All OGW are Computable lemma pt. 2, 2} Let $\beta\in\Pi_{\geq 0}$, $k,\ell\in\ZZ_{\geq 0}$, and $0\leq i_1,\dots,i_\ell\leq K$. Assume that $\ell\geq 1$ and $k\geq 1$, and that $\left|\Gamma_{i_1}\right|\geq\dots\geq\left|\Gamma_{i_\ell}\right|$. Then the coefficient $\OGWbar_{\beta,k}^{\parentheses{\ell}}(\Gamma_{i_1}\otimes\dots\otimes\Gamma_{i_\ell})$ is computable via coefficients of $\Omega$ of order lower than $(\omega(\beta),k,\ell,\left|\Gamma_{i_1}\right|,\dots,\left|\Gamma_{i_\ell}\right|)$, with energy $\omega(\beta)$ or energy less than or equal to $\omega(\beta)-\hbar$, and closed Gromov-Witten invariants.
	\end{Lemma}
	
	\begin{proof}
		By our assumption, $H^*(X;\RR)$ is generated by $H^2(X;\RR)$. We have several possible cases to consider:
		\begin{itemize}
			\item  {$\left|\Gamma_{i_{\ell}}\right| = 0$:} In this case, $\Gamma_{i_\ell} = 1$. Then $\OGWbar_{\beta,k}^{\parentheses{\ell}}\big(\Gamma_{i_1}\otimes\dots\otimes\Gamma_{i_{\ell}}\big)$ is directly computable via the fundamental class axiom \eqref{Open Fundamental Class}.
			\item  {$|\Gamma_{i_\ell}| = 2$}: In this case, by the divisor axiom \eqref{Open Divisor}, we have
			\begin{equation*}
				\OGWbar_{\beta,k}^{\parentheses{\ell}}(\Gamma_{i_1}\otimes\dots\otimes\Gamma_{i_\ell}) = \int_\beta\Gamma_{i_\ell}\cdot\OGWbar_{\beta,k}^{\parentheses{\ell-1}}(\Gamma_{i_1}\otimes\dots\otimes\Gamma_{i_{\ell-1}}),
			\end{equation*}
			and the coefficient is computable via another coefficient of a lower order.
			\item  {$|\Gamma_{i_\ell}|>2$}: In this case, we may write $\Gamma_{i_\ell} = \sum\limits_j\delta'_j\wedge\delta_j$, with $\delta_j\in H^2(X;\RR)$ and $|\delta'_j|\geq 2$ for all $j$. By linearity of the homomorphism $\OGWbar_{\beta,k}^{\parentheses{\ell}}$, we may assume without loss of generality that $\Gamma_{i_\ell} = \delta'\wedge\delta$, where $\delta\in H^2(X;\RR)$, $|\delta'|\geq 2$, and both $\delta,\delta'$ are basis elements in $H^*(X;\RR)$.
			
			Therefore, $\delta'=\Gamma_{a}$ and $\delta=\Gamma_b$ for some $0\leq a,b\leq K$. We now apply the second open WDVV relation -- let us equate coefficients of $T^\beta s^k \prod\limits_{j=1}^\ell t_{i_j}$ in \eqref{Open WDVV Boundary Constraints Polynomial After Expansion}. This gives
			
			\begin{align}
				\label{Open WDVV For Gamma_a=delta', Gamma_b=delta 1}- \sum\limits_{\substack{S_1\cupdot S_2=\\\left\{1,\dots,\ell-1\right\}}}\sum_{\substack{k_1,k_2\in\ZZ_{\geq 0}\\ k_1+k_2 = k-1}}\sum_{\substack{\beta_1,\beta_2\in\Pi_{\geq 0}\\ \beta_1+\beta_2=\beta}}\frac{1}{k_1!k_2!}&\OGWbar_{\beta_1,k_1+1}^{\parentheses{|S_1|+1}}\big(\Gamma_a\otimes\bigotimes\limits_{p\in S_1}\Gamma_{i_p}\big)\\
				\notag&\OGWbar_{\beta_2,k_2+1}^{\parentheses{|S_2|+1}}\big(\Gamma_b\otimes\bigotimes\limits_{q\in S_2}\Gamma_{i_q}\big) \\
				\label{Open WDVV For Gamma_a=delta', Gamma_b=delta 2} =\sum\limits_{\substack{S_1\cupdot S_2=\\ \left\{1,\dots,\ell-1\right\}}}\sum\limits_{0\leq \mu,\nu\leq \kappa}\sum\limits_{\substack{\beta_1\in H_2(X;\ZZ)\\ \beta_2\in \Pi_{\geq 0}\\ \varpi(\beta_1)+\beta_2=\beta}}\frac{(-1)^{w_s(\beta_1)}}{(k-1)!}&\GW_{\beta_1}\big(\Delta_a\otimes\Delta_b\otimes\Delta_\mu\otimes\bigotimes\limits_{p\in S_1}\Delta_{i_p}\big)\\
				\notag g^{\mu\nu}&\OGWbar_{\beta_2,k}^{\parentheses{|S_2|+1}}\big(\Gamma_\nu\otimes\bigotimes\limits_{q\in S_2}\Gamma_{i_q}\big)\\
				\label{Open WDVV For Gamma_a=delta', Gamma_b=delta 3} -\sum\limits_{\substack{S_1\cupdot S_2=\\\left\{1,\dots,\ell-1\right\}}}\sum_{\substack{k_1,k_2\in\ZZ_{\geq 0}\\ k_1+k_2=k-1}}\sum_{\substack{\beta_1,\beta_2\in\Pi_{\geq 0}\\ \beta_1+\beta_2=\beta}}\frac{1}{k_1!k_2!}&\OGWbar_{\beta_1,k_1}^{\parentheses{|S_1|+2}}\big(\Gamma_a\otimes\Gamma_b\otimes\bigotimes\limits_{p\in S_1}\Gamma_{i_p}\big)\\
				\notag&\OGWbar_{\beta_2,k_2+2}^{\parentheses{|S_2|}}\big(\bigotimes\limits_{q\in S_2}\Gamma_{i_q}\big).
			\end{align}
			
			As in the proof of Lemma \ref{All OGW are Computable lemma pt. 2}, we look for nonzero terms of order
			\[
			(\omega(\beta),k,\ell,\left|\Gamma_{i_1}\right|,\dots,\left|\Gamma_{i_\ell}\right|),
			\]
			and we do so by looking at each summand in the above equation separately. For \eqref{Open WDVV For Gamma_a=delta', Gamma_b=delta 1}, the higher order terms are of the form
			\begin{align}
				\label{Open WDVV For Gamma_a=delta', Gamma_b=delta 1, 0}&\OGWbar_{\beta_1,k_1+1}^{\parentheses{|S_1|+1}}\big(\Gamma_a\otimes\bigotimes\limits_{p\in S_1}\Gamma_{i_p}\big)\OGWbar_{\beta_2,k_2+1}^{\parentheses{|S_2|+1}}\big(\Gamma_b\otimes\bigotimes\limits_{q\in S_2}\Gamma_{i_q}\big),
			\end{align}
			for $\beta_1,\beta_2\in\Pi_{\geq 0}$, disjoint sets $S_1,S_2$ satisfying $S_1\cupdot S_2 = \left\{1,\dots,\ell-1\right\}$, and  $k_1,k_2\in\ZZ_{\geq 0}$ satisfying $k_1+k_2=k-1$. For one of the factors in \eqref{Open WDVV For Gamma_a=delta', Gamma_b=delta 1, 0} to have order $(\omega(\beta),k,\ell,\left|\Gamma_{i_1}\right|,\dots,\left|\Gamma_{i_\ell}\right|)$, we must have $\omega(\beta_1) = \omega(\beta)$ or $\omega(\beta_2) = \omega(\beta)$. Therefore, $\omega(\beta_2) = 0$ or $\omega(\beta_1)=0$. By the energy gap  axiom \eqref{Energy Gap}, we deduce that such nonzero contributions are possible only if $\beta_2=\beta_0$ or $\beta_1=\beta_0$, so \eqref{Open WDVV For Gamma_a=delta', Gamma_b=delta 1, 0} reduces to
			\begin{align}
				\label{Open WDVV For Gamma_a=delta', Gamma_b=delta 1, 1}&\OGWbar_{\beta,k_1+1}^{\parentheses{|S_1|+1}}\big(\Gamma_a\otimes\bigotimes\limits_{p\in S_1}\Gamma_{i_p}\big)\OGWbar_{\beta_0,k_2+1}^{\parentheses{|S_2|+1}}\big(\Gamma_b\otimes\bigotimes\limits_{q\in S_2}\Gamma_{i_q}\big)
			\end{align}
			or
			\begin{align}
				\label{Open WDVV For Gamma_a=delta', Gamma_b=delta 1, 2}&\OGWbar_{\beta_0,k_1+1}^{\parentheses{|S_1|+1}}\big(\Gamma_a\otimes\bigotimes\limits_{p\in S_1}\Gamma_{i_p}\big)\OGWbar_{\beta,k_2+1}^{\parentheses{|S_2|+1}}\big(\Gamma_b\otimes\bigotimes\limits_{q\in S_2}\Gamma_{i_q}\big).
			\end{align}
			
			Since $|\Gamma_a|,|\Gamma_b|\geq 2$, both \eqref{Open WDVV For Gamma_a=delta', Gamma_b=delta 1, 1} and \eqref{Open WDVV For Gamma_a=delta', Gamma_b=delta 1, 2} vanish by the zero axiom \eqref{Open Zero}. We thus conclude that there are no higher order factors in \eqref{Open WDVV For Gamma_a=delta', Gamma_b=delta 1}.

			In \eqref{Open WDVV For Gamma_a=delta', Gamma_b=delta 2}, the higher order terms are of the form
			\begin{align}
				\label{Open WDVV For Gamma_a=delta', Gamma_b=delta 2, 1}&\sum\limits_{0 \leq\mu,\nu\leq \kappa} \GW_{\tilde{\beta}}\big(\Delta_a\otimes\Delta_b\otimes\Delta_\mu\otimes\bigotimes\limits_{p\in S_1}\Delta_{i_p}\big)g^{\mu\nu}\OGWbar_{\beta_2,k}^{\parentheses{|S_2|+1}}\big(\Gamma_\nu\otimes\bigotimes\limits_{q\in S_2}\Gamma_{i_q}\big),
			\end{align}
			for some disjoint sets $S_1,S_2$ satisfying $S_1\cupdot S_2 = \left\{1,\dots,\ell-1\right\}$ and $\tilde{\beta}\in H_2(X;\ZZ)$, $\beta_2\in\Pi_{\geq 0}$ that satisfy $\varpi(\tilde{\beta})+\beta_2=\beta$. For the open factor $\OGWbar$ in \eqref{Open WDVV For Gamma_a=delta', Gamma_b=delta 2, 1} to have order higher or equal to $(\omega(\beta),k,\ell,\left|\Gamma_{i_1}\right|,\dots,\left|\Gamma_{i_\ell}\right|)$, we must have $\omega(\beta_2)=\omega(\beta)$, and therefore $\omega(\varpi(\tilde{\beta})) = 0$. By Lemma \ref{If varpi(beta)=0 then beta=0}, we deduce that such a contribution is possible only if $\tilde{\beta} = 0$.
			
			Since $\tilde{\beta}= 0$, the closed zero axiom \eqref{Closed Zero} implies that $\GW_{\tilde{\beta}}\big(\Delta_a\otimes\Delta_b\otimes\Delta_\mu\otimes\bigotimes\limits_{p\in S_1}\Delta_{i_p}\big)$ does not vanish only if $|S_1|=0$. In this case, $S_2 = \left\{1,\dots,\ell-1\right\}$, so \eqref{Open WDVV For Gamma_a=delta', Gamma_b=delta 2, 1} reduces, by Lemma \ref{g is block matrix when needed}, to
			\begin{align*}
				\sum\limits_{0\leq \mu,\nu\leq \kappa} \GW_{0}&\big(\Delta_a\otimes\Delta_b\otimes\Delta_\mu\big)g^{\mu\nu}\OGWbar_{\beta,k}^{\parentheses{\ell}}\big(\Gamma_\nu\otimes\Gamma_{i_1}\otimes\dots\otimes\Gamma_{i_{\ell-1}}\big) \\
				=&\sum\limits_{0\leq \mu,\nu\leq \kappa} \big(\int_X\Delta_a\wedge\Delta_b\wedge\Delta_\mu\big) g^{\mu\nu}\OGWbar_{\beta,k}^{\parentheses{\ell}}\big(\Gamma_\nu\otimes\Gamma_{i_1}\otimes\dots\otimes\Gamma_{i_{\ell-1}}\big) \\
				=&\OGWbar_{\beta,k}^{\parentheses{\ell}}\big(\big(\sum\limits_{0\leq \mu,\nu\leq \kappa}\int_X\Delta_a\wedge\Delta_b\wedge\Delta_\mu\big) g^{\mu\nu}\Gamma_\nu\otimes\Gamma_{i_1}\otimes\dots\otimes\Gamma_{i_{\ell-1}}\big) \\
				=&\OGWbar_{\beta,k}^{\parentheses{\ell}}\big(\Delta_a\wedge\Delta_b\otimes\Gamma_{i_1}\otimes\dots\otimes\Gamma_{i_{\ell-1}}\big)\\
				=&\OGWbar_{\beta,k}^{\parentheses{\ell}}\big(\Gamma_{i_1}\otimes\dots\otimes\Gamma_{i_{\ell}}\big).
			\end{align*}
			The last expression is the coefficient we are looking to compute.
			
			We therefore conclude that the only term where the open factor $\OGWbar$ has the highest order in \eqref{Open WDVV For Gamma_a=delta', Gamma_b=delta 2} is $\OGWbar_{\beta,k}^{\parentheses{\ell}}\big(\Gamma_{i_1}\otimes\dots\otimes\Gamma_{i_{\ell}}\big)$.
			
			In \eqref{Open WDVV For Gamma_a=delta', Gamma_b=delta 3}, the higher order terms are of the form			
			\begin{align}
				\label{Open WDVV For Gamma_a=delta', Gamma_b=delta 3, 0}&\OGWbar_{\beta_1,k_1}^{\parentheses{|S_1|+2}}\big(\Gamma_a\otimes\Gamma_b\otimes\bigotimes\limits_{p\in S_1}\Gamma_{i_p}\big)\OGWbar_{\beta_2,k_2+2}^{\parentheses{|S_2|}}\big(\bigotimes\limits_{q\in S_2}\Gamma_{i_q}\big),
			\end{align}
			for $\beta_1,\beta_2\in\Pi_{\geq 0}$ such that $\beta_1+\beta_2=\beta$, disjoint sets $S_1,S_2$ that satisfy $S_1\cupdot S_2 = \left\{2,\dots,\ell-1\right\}$, and $k_1,k_2\in\ZZ_{\geq 0}$ that satisfy $k_1+k_2=k-1$. If either of the factors in \eqref{Open WDVV For Gamma_a=delta', Gamma_b=delta 3, 0} has order higher or equal to $(\omega(\beta),k,\ell,\left|\Gamma_{i_1}\right|,\dots,\left|\Gamma_{i_\ell}\right|)$, then $\omega(\beta_1)=\omega(\beta)$ or $\omega(\beta_2) = \omega(\beta)$, and therefore $\omega(\beta_2)=0<\hbar$ or $\omega(\beta_1)=0<\hbar$. By the energy gap axiom \eqref{Energy Gap}, we deduce that the term does not vanish only if $\beta_2=\beta_0$ or $\beta_1 = \beta_0$. Therefore, the term \eqref{Open WDVV For Gamma_a=delta', Gamma_b=delta 3, 0} reduces to
			\begin{align}
				\label{Open WDVV For Gamma_a=delta', Gamma_b=delta 3, 1}&\OGWbar_{\beta,k_1}^{\parentheses{|S_1|+2}}\big(\Gamma_a\otimes\Gamma_b\otimes\bigotimes\limits_{p\in S_1}\Gamma_{i_p}\big)\OGWbar_{\beta_0,k_2+2}^{\parentheses{|S_2|}}\big(\bigotimes\limits_{q\in S_2}\Gamma_{i_q}\big)
			\end{align}
			or
			\begin{align}
				\label{Open WDVV For Gamma_a=delta', Gamma_b=delta 3, 2}&\OGWbar_{\beta_0,k_1}^{\parentheses{|S_1|+2}}\big(\Gamma_a\otimes\Gamma_b\otimes\bigotimes\limits_{p\in S_1}\Gamma_{i_p}\big)\OGWbar_{\beta,k_2+2}^{\parentheses{|S_2|}}\big(\bigotimes\limits_{q\in S_2}\Gamma_{i_q}\big).
			\end{align}
			In \eqref{Open WDVV For Gamma_a=delta', Gamma_b=delta 3, 1}, note that $k_2+2\geq 2$, and therefore the zero axiom \eqref{Open Zero} implies that this term vanishes.
			
			For \eqref{Open WDVV For Gamma_a=delta', Gamma_b=delta 3, 2}, note that $|S_2|+2\geq 2$, so the zero axiom \eqref{Open Zero} implies that $|S_2|=0$, $k_1 = 0$, and
			\begin{align*}
				\OGWbar_{\beta_0,k_1}^{\parentheses{|S_1|+2}}\big(\Gamma_a\otimes\Gamma_b\otimes\bigotimes\limits_{p\in S_1}\Gamma_{i_p}\big) &= P_\RR\big(\Gamma_a\wedge\Gamma_b \big) = 0,
			\end{align*}
			where the last equality comes from the fact that $P_\RR$ vanishes on $H^*(X;\RR)$.
			
			We conclude that all higher order terms in \eqref{Open WDVV For Gamma_a=delta', Gamma_b=delta 3, 2} vanish as well, and we therefore conclude that are no higher order terms in \eqref{Open WDVV For Gamma_a=delta', Gamma_b=delta 3}.
		\end{itemize}
		
		In summary, we get the relation
		\begin{align}\label{summary of all OGW are computable lemma pt. 2, 2}
			\OGWbar_{\beta,k}^{\parentheses{\ell}}(\Gamma_{i_1}\otimes\dots\otimes\Gamma_{i_\ell})\hspace{7em}& \\
			\notag= -\sum\limits_{\substack{S_1\cupdot S_2=\\\left\{1,\dots,\ell-1\right\}}}\sum_{\substack{k_1,k_2\in\ZZ_{\geq 0}\\ k_1+k_2 = k-1}}\sum_{\substack{\beta_1,\beta_2\in\Pi_{> 0}\\ \beta_1+\beta_2=\beta}}\binom{k-1}{k_1}&\OGWbar_{\beta_1,k_1+1}^{\parentheses{|S_1|+1}}\big(\Gamma_a\otimes\bigotimes\limits_{p\in S_1}\Gamma_{i_p}\big)\\
			\notag&\OGWbar_{\beta_2,k_2+1}^{\parentheses{|S_2|+1}}\big(\Gamma_b\otimes\bigotimes\limits_{q\in S_2}\Gamma_{i_q}\big) \\
			\notag+\sum\limits_{\substack{S_1\cupdot S_2=\\ \left\{1,\dots,\ell-1\right\}}}\sum\limits_{0\leq \mu,\nu\leq \kappa}\sum\limits_{\substack{\beta_1\in H_2(X;\ZZ)\\ \beta_2\in \Pi_{\geq 0}\\ \omega(\beta_1)>0\\ \varpi(\beta_1)+\beta_2=\beta}}(-1)^{w_s(\beta_1)}&\GW_{\beta_1}\big(\Delta_a\otimes\Delta_b\otimes\Delta_\mu\otimes\bigotimes\limits_{p\in S_1}\Delta_{i_p}\big)\\
			\notag g^{\mu\nu}&\OGWbar_{\beta_2,k}^{\parentheses{|S_2|+1}}\big(\Gamma_\nu\otimes\bigotimes\limits_{q\in S_2}\Gamma_{i_q}\big)\\
			\notag-\sum\limits_{\substack{S_1\cupdot S_2=\\\left\{1,\dots,\ell-1\right\}}}\sum_{\substack{k_1,k_2\in\ZZ_{\geq 0}\\ k_1+k_2=k-1}}\sum_{\substack{\beta_1,\beta_2\in\Pi_{> 0}\\ \beta_1+\beta_2=\beta}}\binom{k-1}{k_1}&\OGWbar_{\beta_1,k_1}^{\parentheses{|S_1|+2}}\big(\Gamma_a\otimes\Gamma_b\otimes\bigotimes\limits_{p\in S_1}\Gamma_{i_p}\big)\\
			\notag&\OGWbar_{\beta_2,k_2+2}^{\parentheses{|S_2|}}\big(\bigotimes\limits_{q\in S_2}\Gamma_{i_q}\big),
		\end{align}
		where the right-hand side contains only lower-order terms. As in Lemma \ref{All OGW are Computable lemma pt. 2}, we also note that the energy gap axiom \eqref{Energy Gap} implies that such terms of lower order are of energy $\omega(\beta)$ or have energy lower than or equal to $\omega(\beta)-\hbar$.
		
		We conclude that the coefficient $\OGWbar_{\beta,k}^{\parentheses{\ell}}\left(\Gamma_{i_1}\otimes\dots\otimes\Gamma_{i_\ell}\right)$ is computable via coefficients of $\Omega$ of a lower order, with energy $\omega(\beta)$ or energy lower than or equal to $\omega(\beta)-\hbar$, and closed Gromov-Witten invariants.
	\end{proof}
	
	Finally, we move to the last reduction step of our proof. We show that if the assumptions of Theorem \ref{All OGW are computable 2 two OGW two computable} are met, then all coefficients of $\Omega$ with $\ell=0$ and $k\geq 2$ are computable via terms of a lower order.
	
	\begin{Lemma}
		\label{All OGW are Computable lemma pt. 3} Assume that $S_L = \ker(\omega\oplus\mu)$. Assume also that there exist $\beta'\in \Pi_{\geq 0}$, $\ell'\in\ZZ_{\geq 1}$, and $0\leq i_1,\dots,i_{\ell'}\leq K$ for which $\omega(\beta')=\hbar$ and $\OGWbar_{\beta',0}^{\parentheses{\ell'}}\left(\Gamma_{i_1}\otimes\dots\otimes\Gamma_{i_\ell'}\right)\neq 0$. Assume further that the pair $(\beta',\ell')$ is minimal, in the sense that for any $\beta''\in\Pi_{>0}$ such that $\omega(\beta'')=\hbar$, any $\ell''\in\ZZ_{\geq 0}$ such that $\ell''<\ell'$, and any $0\leq j_1,\dots,j_{\ell''}\leq K$, the term $\OGWbar_{\beta'',0}^{\parentheses{\ell''}}\left(\Gamma_{j_1}\otimes\dots\otimes\Gamma_{j_{\ell''}}\right)$ vanishes. If $\ell'=1$, assume that $\left|\Gamma_{i_1}\right| = 2n$.\\
		Let $\beta\in\Pi_{\geq 0}$ and $k\in\ZZ_{\geq 0}$ such that $k\geq 2$. Then the coefficient $\OGW_{\beta,k}^{\parentheses{0}}$ is computable via coefficients of $\Omega$ of order lower than $(\omega(\beta),k,0,0)$, with energy either $\omega(\beta)$ less than or equal to $\omega(\beta)-\hbar$, and closed Gromov-Witten invariants.
	\end{Lemma}
	
	Before proving Lemma \ref{All OGW are Computable lemma pt. 3}, we prove an auxiliary result.
	\begin{Lemma}\label{All OGW are Computable lemma pt. 3, aux}
		Under the assumptions and notation of Lemma \ref{All OGW are Computable lemma pt. 3}, assume $\ell'\geq 2$ and let $S\subseteq\left\{1,\dots,\ell'\right\}$ be a subset of size $\ell'-1$. Assume also that $\omega(\beta)\geq \hbar$ and that the coefficient $\OGWbar_{\beta,k}$ satisfies the degree axiom \eqref{Open Degree}.\\
		Then the coefficient $\OGWbar_{\beta+\beta',k-1}^{\parentheses{\ell'-1}}\big(\bigotimes\limits_{p\in S}\Gamma_{i_p}\big)$ can be computed via coefficients of $\Omega$ of order lower than $(\omega(\beta),k,0,0)$, with energy either $\omega(\beta)$ or less than or equal to $\omega(\beta)-\hbar$, and closed Gromov-Witten invariants.
	\end{Lemma}
	
	\begin{proof}[Proof of Lemma \ref{All OGW are Computable lemma pt. 3, aux}]
		Let $j\in S$. First, we note that $|\Gamma_{i_j}|\geq 2$. If this were not the case then, since $H^2(X;\RR)$ generates $H^*(X;\RR)$, we would have $|\Gamma_{i_j}|=0$, and therefore $\Gamma_{i_j} = 1$. By the fundamental class axiom \eqref{Open Fundamental Class}, since $\ell'\geq 2$ and $k\geq 2$, the coefficient $\OGWbar_{\beta',0}^{\parentheses{\ell'}}\big(\Gamma_{i_1}\otimes\dots\otimes\Gamma_{i_\ell'}\big)$ would vanish. Next, we note that $|\Gamma_{i_{j}}|>2$. If this were not the case, we would have $|\Gamma_{i_{j}}|=2$, so by the divisor axiom \eqref{Open Divisor} we would get a contradiction to the minimality of $\ell'$.
		
		Thus, for all $j\in S$, we have $\left|\Gamma_{i_j}\right|>2$. Now, $\ell'\geq 2$ and $|S|=\ell'-1$, so $S$ is nonempty. Choose some $m\in S$. Since $|\Gamma_{i_{m}}|>2$, and since $H^*(X;\RR)$ is generated by $H^2(X;\RR)$, we may write $\Gamma_{i_m} = \sum\limits_j\delta'_j\wedge\delta_j$, with $\delta_j\in H^2(X;\RR)$ and $|\delta'_j|\geq 2$ for all $j$. By linearity of the homomorphism $\OGWbar_{\beta',0}^{\parentheses{\ell'-1}}$, we may assume without loss of generality that $\Gamma_{i_m} = \delta'\wedge\delta$, where $\delta\in H^2(X;\RR)$, $|\delta'|\geq 2$, and both $\delta,\delta'$ are basis elements in $H^*(X;\RR)$. Therefore, $\delta'=\Gamma_{a}$ and $\delta=\Gamma_b$ for some $0\leq a,b\leq K$.
		
		We now apply the second open WDVV relation with the above $\Gamma_a, \Gamma_b$. Equating the coefficients of $T^{\beta+\beta'}s^{k-2}\prod\limits_{j\in S\backslash\{m\}}t_{i_j}$ in \eqref{Open WDVV Boundary Constraints Polynomial After Expansion}, we get the relation
		\begin{align}
			\label{Open WDVV for aux lemma 1} -\sum\limits_{\substack{S_1\cupdot S_2=\\ S\backslash\{m\}}}\sum_{\substack{k_1,k_2\in\ZZ_{\geq 0}\\ k_1+k_2 = k-2}}\sum_{\substack{\beta_1,\beta_2\in\Pi_{\geq 0}\\ \beta_1+\beta_2=\beta+\beta'}}\frac{1}{k_1!k_2!}&\OGWbar_{\beta_1,k_1+1}^{\parentheses{|S_1|+1}}\big(\Gamma_a\otimes\bigotimes\limits_{p\in S_1}\Gamma_{i_p}\big)\\
			\notag&\OGWbar_{\beta_2,k_2+1}^{\parentheses{|S_2|+1}}\big(\Gamma_b\otimes\bigotimes\limits_{q\in S_2}\Gamma_{i_q}\big) \\
			\label{Open WDVV for aux lemma 2} =\sum\limits_{\substack{S_1\cupdot S_2=\\ S\backslash\left\{m\right\}}}\sum\limits_{0\leq \mu,\nu\leq \kappa}\sum\limits_{\substack{\beta_1\in H_2(X;\ZZ)\\ \beta_2\in \Pi_{\geq 0}\\ \varpi(\beta_1)+\beta_2=\beta+\beta'}}\frac{(-1)^{w_s(\beta_1)}}{(k-2)!}&\GW_{\beta_1}\big(\Delta_a\otimes\Delta_b\otimes\Delta_\mu\otimes\bigotimes\limits_{p\in S_1}\Delta_{i_p}\big)\\
			\notag g^{\mu\nu}&\OGWbar_{\beta_2,k-1}^{\parentheses{|S_2|+1}}\big(\Gamma_\nu\otimes\bigotimes\limits_{q\in S_2}\Gamma_{i_q}\big)\\
			\label{Open WDVV for aux lemma 3} -\sum\limits_{\substack{S_1\cupdot S_2=\\S\backslash\left\{m\right\}}}\sum_{\substack{k_1,k_2\in\ZZ_{\geq 0}\\ k_1+k_2=k-2}}\sum_{\substack{\beta_1,\beta_2\in\Pi_{\geq 0}\\ \beta_1+\beta_2=\beta+\beta'}}\frac{1}{k_1!k_2!}&\OGWbar_{\beta_1,k_1}^{\parentheses{|S_1|+2}}\big(\Gamma_a\otimes\Gamma_b\otimes\bigotimes\limits_{p\in S_1}\Gamma_{i_p}\big)\\
			\notag&\OGWbar_{\beta_2,k_2+2}^{\parentheses{|S_2|}}\big(\bigotimes\limits_{q\in S_2}\Gamma_{i_q}\big).
		\end{align}
		
		We will now look for nonzero terms with factors of order $(\omega(\beta),k,0,0)$ or higher in the above equation, and we will do so by considering each summand seperately. In \eqref{Open WDVV for aux lemma 1}, the terms with factors of higher order must be of the form
		
		\begin{equation}
			\label{Open WDVV for aux lemma 1, 1} \OGWbar_{\beta_1,k_1+1}^{\parentheses{|S_1|+1}}\big(\Gamma_a\otimes\bigotimes\limits_{p\in S_1}\Gamma_{i_p}\big)\OGWbar_{\beta_2,k_2+1}^{\parentheses{|S_2|+1}}\big(\Gamma_b\otimes\bigotimes\limits_{q\in S_2}\Gamma_{i_q}\big),
		\end{equation}
		for some disjoint sets $S_1,S_2$ such that $S_1\cupdot S_2 = S\backslash\{m\}$, some $\beta_1,\beta_2\in\Pi_{\geq 0}$ such that $\beta_1+\beta_2=\beta+\beta'$, and $k_1,k_2\in\ZZ_{\geq 0}$ such that $k_1+k_2=k-2$.
		
		We have several different cases to consider:
		\begin{enumerate}
			\item \label{Open WDVV for aux lemma 1, 1, case a} {$\beta_1 = \beta+\beta'$:} In this case, $\beta_2=\beta_0$. Note that $\left|\Gamma_b\right|\geq 2$, so $\Gamma_b\neq 1$, but on the other hand $k_1+1>0$. By the zero axiom \eqref{Open Zero}, the value $\OGWbar_{\beta_2,k_2+1}^{\parentheses{|S_2|+1}}\big(\Gamma_b\otimes\bigotimes\limits_{q\in S_2}\Gamma_{i_q}\big)$ vanishes, so in this case \eqref{Open WDVV for aux lemma 1, 1} reduces to zero.
			\item \label{Open WDVV for aux lemma 1, 1, case b} {$\beta_2 = \beta+\beta'$:} By an identical argument to case \ref{Open WDVV for aux lemma 1, 1, case a}, the term \eqref{Open WDVV for aux lemma 1, 1} reduces to zero in this case.
			\item \label{Open WDVV for aux lemma 1, 1, case c} {$\beta_2\neq \beta_0$ and $\omega(\beta_1) > \omega(\beta)$:} Note that in this case:
			\begin{equation*}
				\omega(\beta_2) = \omega(\beta+\beta'-\beta_1) = \omega(\beta)+\hbar - \omega(\beta_1) <\hbar.
			\end{equation*}
			By the energy gap axiom \eqref{Energy Gap}, the value $\OGWbar_{\beta_2,k_2+1}^{\parentheses{|S_2|+1}}\big(\Gamma_b\otimes\bigotimes\limits_{q\in S_2}\Gamma_{i_q}\big)$ vanishes, so in this case \eqref{Open WDVV for aux lemma 1, 1} reduces to zero.
			\item \label{Open WDVV for aux lemma 1, 1, case d} {$\beta_1\neq \beta_0$ and $\omega(\beta_2) > \omega(\beta)$:} By an identical argument to case \ref{Open WDVV for aux lemma 1, 1, case c},  \eqref{Open WDVV for aux lemma 1, 1} reduces to zero in this case.
			\item \label{Open WDVV for aux lemma 1, 1, case e} {$\omega(\beta_1) = \omega(\beta)$:} In this case, note that $k_1+1\leq k-1<k$, so the term $\OGWbar_{\beta_1,k_1+1}^{\parentheses{|S_1|+1}}\big(\Gamma_a\otimes\bigotimes\limits_{p\in S_1}\Gamma_{i_p}\big)$ has lower order than $(\omega(\beta),k,0,0)$. As for the term $\OGWbar_{\beta_2,k_2+1}^{\parentheses{|S_2|+1}}\big(\Gamma_b\otimes\bigotimes\limits_{q\in S_2}\Gamma_{i_q}\big)$, note that in this case $\omega(\beta_2) = \omega(\beta')=\hbar\leq\omega(\beta)$, and $k_2+1\leq k-1<k$, so this term also has order lower than $(\omega(\beta),k,0,0)$.
			\item \label{Open WDVV for aux lemma 1, 1, case f} {$\omega(\beta_2) = \omega(\beta)$:} By an identical argument to case \ref{Open WDVV for aux lemma 1, 1, case e}, term \eqref{Open WDVV for aux lemma 1, 1} is of lower order than $(\omega(\beta),k,0,0)$ in this case.
		\end{enumerate}
		We conclude there are no higher order terms in \eqref{Open WDVV for aux lemma 1}.
		
		In \eqref{Open WDVV for aux lemma 2}, the higher order terms are of the form
		\begin{equation}
			\label{Open WDVV for aux lemma 2, 1} \sum\limits_{0 \leq\mu,\nu\leq \kappa} \GW_{\beta_1}\big(\Delta_a\otimes\Delta_b\otimes\Delta_\mu\otimes\bigotimes\limits_{p\in S_1}\Delta_{i_p}\big)g^{\mu\nu}\OGWbar_{\beta_2,k-1}^{\parentheses{|S_2|+1}}\big(\Gamma_\nu\otimes\bigotimes\limits_{q\in S_2}\Gamma_{i_q}\big),
		\end{equation}
		for some disjoint sets $S_1,S_2$ that satisfy $S_1\cupdot S_2 = S\backslash\{m\}$, and some $\beta_1\in H_2(X;\ZZ)$ and $\beta_2\in\Pi_{\geq 0}$ that satisfy $\varpi(\beta_1)+\beta_2 = \beta+\beta'$.
		
		We again have several cases to consider:
		\begin{enumerate}
			\item  {$\beta_2 = \beta + \beta'$:} In this case, $\varpi(\beta_1) = \beta_0$. If $\beta_1\neq 0$, then the value $\GW_{\beta_1}\big(\Delta_a\otimes\Delta_b\otimes\Delta_\mu\otimes\bigotimes\limits_{p\in S_1}\Delta_{i_p}\big)$ vanishes due to Lemma \ref{If varpi(beta)=0 then beta=0}, and \eqref{Open WDVV for aux lemma 2, 1} reduces to zero. Otherwise, $\beta_1 = 0$. In this case, the zero axiom \eqref{Closed Zero} implies that $|S_1| = 0$, so $S_2 = S\backslash\{m\}$. Thus, \eqref{Open WDVV for aux lemma 2, 1} reduces, by Lemma \ref{g is block matrix when needed}, to:
			\begin{align*}
				\sum\limits_{0\leq \mu,\nu\leq \kappa} \GW_{0}&\big(\Delta_a\otimes\Delta_b\otimes\Delta_\mu\big)g^{\mu\nu}\OGWbar_{\beta+\beta',k-1}^{\parentheses{\ell'-1}}\big(\Gamma_\nu\otimes\bigotimes\limits_{p\in S\backslash\{m\}}\Gamma_{i_p}\big) \\
				=&\sum\limits_{0\leq \mu,\nu\leq \kappa} \big(\int_X\Delta_a\wedge\Delta_b\wedge\Delta_\mu\big) g^{\mu\nu}\OGWbar_{\beta+\beta',k-1}^{\parentheses{\ell'-1}}\big(\Gamma_\nu\otimes\bigotimes\limits_{p\in S\backslash\{m\}}\Gamma_{i_p}\big) \\
				=&\OGWbar_{\beta+\beta',k-1}^{\parentheses{\ell'-1}}\big(\big(\sum\limits_{0\leq \mu,\nu\leq \kappa} \int_X\Delta_a\wedge\Delta_b\wedge\Delta_\mu\big) g^{\mu\nu}\Gamma_\nu\otimes\bigotimes\limits_{p\in S\backslash\{m\}}\Gamma_{i_p}\big) \\
				=&\OGWbar_{\beta+\beta',k-1}^{\parentheses{\ell'-1}}\big(\Delta_a\wedge\Delta_b\otimes\bigotimes\limits_{p\in S\backslash\{m\}}\Gamma_{i_p}\big)\\
				=&\OGWbar_{\beta+\beta',k-1}^{\parentheses{\ell'-1}}\big(\Gamma_{i_m}\otimes\bigotimes\limits_{p\in S\backslash\{m\}}\Gamma_{i_p}\big)\\
				=&\OGWbar_{\beta+\beta',k-1}^{\parentheses{\ell'-1}}\big(\bigotimes\limits_{p\in S}\Gamma_{i_p}\big),
			\end{align*}
			And that is the coefficient we are looking to compute.
			
			\item  {$\varpi(\beta_1)\neq \beta_0$ and $\omega(\beta_2)>\omega(\beta)$:} In this case, note that
			
			\begin{equation*}
				\omega(\varpi(\beta_1)) = \omega(\beta+\beta' - \beta_2) = \omega(\beta)+\hbar - \omega(\beta_2) < \hbar.
			\end{equation*}
			
			By Lemma \ref{If varpi(beta)=0 then beta=0}, we deduce that $\GW_{\beta_1}\big(\Delta_a\otimes\Delta_b\otimes\Delta_\mu\otimes\bigotimes\limits_{p\in S_1}\Delta_{i_p}\big)$ vanishes, and therefore \eqref{Open WDVV for aux lemma 2, 1} reduces to zero.
			\item  {$\omega(\beta_2) = \omega(\beta)$:} In this case, the term $\OGWbar_{\beta_2,k-1}^{\parentheses{|S_2|+1}}\big(\Gamma_\nu\otimes\bigotimes\limits_{q\in S_2}\Gamma_{i_q}\big)$ has order lower than $(\omega(\beta),k,0,0)$.
		\end{enumerate}
		
		We conclude that the only higher order $\OGWbar$ term in \eqref{Open WDVV for aux lemma 2} is $\OGWbar_{\beta+\beta',k-1}^{\parentheses{\ell'-1}}\big(\bigotimes\limits_{p\in S}\Gamma_{i_p}\big)$.
		
		In \eqref{Open WDVV for aux lemma 3}, the higher order terms are of the form
		\begin{equation}\label{Open WDVV for aux lemma 3, 1}
			\OGWbar_{\beta_1,k_1}^{\parentheses{|S_1|+2}}\big(\Gamma_a\otimes\Gamma_b\otimes\bigotimes\limits_{p\in S_1}\Gamma_{i_p}\big)\OGWbar_{\beta_2,k_2+2}^{\parentheses{|S_2|}}\big(\bigotimes\limits_{q\in S_2}\Gamma_{i_q}\big),
		\end{equation}
		for some disjoint sets $S_1,S_2$ that satisfy $S_1\cupdot S_2 = S\backslash\{m\}$, some $\beta_1,\beta_2\in\Pi_{\geq 0}$ that satisfy $\beta_1+\beta_2 = \beta+\beta'$, and $k_1,k_2\in\ZZ_{\geq 0}$ that satisfy $k_1+k_2 = k-2$.
		
		Again, we have several cases to consider:
		\begin{enumerate}
			\item \label{Open WDVV for aux lemma 3, 1, case a} {$\beta_1 = \beta+\beta'$:} In this case, $\beta_2=\beta_0$. By the zero axiom \eqref{Open Zero} and the fact that $k_2+2\geq 2$, the value $\OGWbar_{\beta_2,k_2+2}^{\parentheses{|S_2|}}\big(\bigotimes\limits_{q\in S_2}\Gamma_{i_q}\big)$ vanishes, so in this case \eqref{Open WDVV for aux lemma 3, 1} reduces to zero.
			\item \label{Open WDVV for aux lemma 3, 1, case b} {$\beta_2 = \beta+\beta'$:} In this case $\beta_1 = \beta_0$. By the zero axiom \eqref{Open Zero}, this implies that $k_1 = 0$ and $|S_1| = 0$, so
			\begin{equation*}
				\OGWbar_{\beta_1,k_1}^{\parentheses{|S_1|+2}}\big(\Gamma_a\otimes\Gamma_b\otimes\bigotimes\limits_{p\in S_1}\Gamma_{i_p}\big) = P_\RR\big(\Gamma_a\wedge\Gamma_b\big) = 0,
			\end{equation*}
			since $P_\RR$ vanishes on $H^*(X;\RR)$. Therefore, \eqref{Open WDVV for aux lemma 3, 1} reduces to zero in this case.
			\item \label{Open WDVV for aux lemma 3, 1, case c} {$\beta_2\neq \beta_0$ and $\omega(\beta_1) > \omega(\beta)+\beta''$:} We note that
			\begin{equation*}
				\omega(\beta_2)  = \omega(\beta+\beta' - \beta_1) =  \omega(\beta)+\hbar - \omega(\beta_1) <\hbar.
			\end{equation*}
			By the energy gap axiom \eqref{Energy Gap}, the value $\OGWbar_{\beta_2,k_2+2}^{\parentheses{|S_2|}}\big(\bigotimes\limits_{q\in S_2}\Gamma_{i_q}\big)$ vanishes, so in this case \eqref{Open WDVV for aux lemma 3, 1} reduces to zero.
			\item \label{Open WDVV for aux lemma 3, 1, case d} {$\beta_1\neq \beta_0$ and $\omega(\beta_2) > \omega(\beta)$:} By an identical argument to case \ref{Open WDVV for aux lemma 3, 1, case c},  the term \eqref{Open WDVV for aux lemma 3, 1} reduces to zero in this case.
			\item \label{Open WDVV for aux lemma 3, 1, case e}  {$\omega(\beta_1) = \omega(\beta)$:} In this case, note that $k_1\leq k-2<k$, so $\OGWbar_{\beta_1,k_1}^{\parentheses{|S_1|+2}}\big(\Gamma_a\otimes\Gamma_b\otimes\bigotimes\limits_{p\in S_1}\Gamma_{i_p}\big)$ has lower order than $(\omega(\beta),k,0,0)$.
			
			On the other hand, note that $\omega(\beta_2) = \omega(\beta')$. If $\omega(\beta')<\omega(\beta)$, then the term $\OGWbar_{\beta_2,k_2+2}^{\parentheses{|S_2|}}\big(\bigotimes\limits_{q\in S_2}\Gamma_{i_q}\big)$ has lower order than $(\omega(\beta),k,0,0)$. If $\omega(\beta')=\omega(\beta)$, then for the term $\OGWbar_{\beta_2,k_2+2}^{\parentheses{|S_2|}}\big(\bigotimes\limits_{q\in S_2}\Gamma_{i_q}\big)$ to have order higher or equal to $(\omega(\beta),k,0,0)$, we must have that $k_2+2=k$, and therefore $k_1=0$. In this case $\omega(\beta_1)=\hbar$ and by the divisor axiom \eqref{Open Divisor},
			\[ \OGWbar_{\beta_1,k_1}^{\parentheses{|S_1|+2}}\big(\Gamma_a\otimes\Gamma_b\otimes\bigotimes\limits_{p\in S_1}\Gamma_{i_p}\big) = \int_{\beta_1}\Gamma_b\cdot\OGWbar_{\beta_1,0}^{\parentheses{|S_1|+1}}\big(\Gamma_a\otimes\bigotimes\limits_{p\in S_1}\Gamma_{i_p}\big). \]
			Since $|S_1|+1\leq\ell'-1<\ell'$, we deduce by minimality of $(\beta',\ell')$ that the term $\OGWbar_{\beta_1,k_1}^{\parentheses{|S_1|+2}}\big(\Gamma_a\otimes\Gamma_b\otimes\bigotimes\limits_{p\in S_1}\Gamma_{i_p}\big)$ vanishes.

			Thus, \eqref{Open WDVV for aux lemma 3, 1} has no non-vanishing factors of higher order in this case.
			\item  {$\omega(\beta_2) = \omega(\beta)$:} Then $\omega(\beta_1) = \omega(\beta') = \hbar$. If the term $\OGWbar_{\beta_2,k_2+2}^{\parentheses{|S_2|}}\big(\bigotimes\limits_{q\in S_2}\Gamma_{i_q}\big)$ has order higher or equal to $(\omega(\beta),k,0,0)$, then $k_2+2 = k$. In this case $k_1 = 0$, and by the divisor axiom \eqref{Open Divisor},
			\[ \OGWbar_{\beta_1,k_1}^{\parentheses{|S_1|+2}}\big(\Gamma_a\otimes\Gamma_b\otimes\bigotimes\limits_{p\in S_1}\Gamma_{i_p}\big) = \int_{\beta_1}\Gamma_b\cdot\OGWbar_{\beta_1,0}^{\parentheses{|S_1|+1}}\big(\Gamma_a\otimes\bigotimes\limits_{p\in S_1}\Gamma_{i_p}\big). \]
			
			Since $|S_1|+1\leq\ell'-1<\ell'$, we deduce by minimality of $(\beta',\ell')$ that the term $\OGWbar_{\beta_1,k_1}^{\parentheses{|S_1|+2}}\big(\Gamma_a\otimes\Gamma_b\otimes\bigotimes\limits_{p\in S_1}\Gamma_{i_p}\big)$ vanishes.
			
			As for the order of the term $\OGWbar_{\beta_1,k_1}^{\parentheses{|S_1|+2}}\big(\Gamma_a\otimes\Gamma_b\otimes\bigotimes\limits_{p\in S_1}\Gamma_{i_p}\big)$, note that $\omega(\beta_1)=\hbar\leq\omega(\beta)$ and $k_1\leq k-2<k$, so this term has in any case an order lower than $(\omega(\beta),k,0,0)$.
			
			Thus, \eqref{Open WDVV for aux lemma 3, 1} has no non-vanishing factors of higher order in this case.
			
		\end{enumerate}
		We conclude that \eqref{Open WDVV for aux lemma 3} has no nonzero terms where one of the factors has order higher or equal to $(\omega(\beta),k,0,0)$.
		
		We get the relation
		\begin{align}
			\OGWbar_{\beta+\beta',k-1}^{\parentheses{\ell'-1}}\big(\bigotimes_{p\in S}\Gamma_{i_p}\big)\hspace{8em}& \\
			\notag= -\sum\limits_{\substack{S_1\cupdot S_2=\\ S\backslash\{m\}}}\sum_{\substack{k_1,k_2\in\ZZ_{\geq 0}\\ k_1+k_2 = k-2}}\sum_{\substack{\beta_1,\beta_2\in\Pi_{\geq 0}\\ \beta_1+\beta_2=\beta+\beta'}}\binom{k-2}{k_1}&\OGWbar_{\beta_1,k_1+1}^{\parentheses{|S_1|+1}}\big(\Gamma_a\otimes\bigotimes\limits_{p\in S_1}\Gamma_{i_p}\big)\\
			\notag&\OGWbar_{\beta_2,k_2+1}^{\parentheses{|S_2|+1}}\big(\Gamma_b\otimes\bigotimes\limits_{q\in S_2}\Gamma_{i_q}\big) \\
			\notag-\sum\limits_{\substack{S_1\cupdot S_2=\\ S\backslash\left\{m\right\}}}\sum\limits_{0\leq \mu,\nu\leq \kappa}\sum\limits_{\substack{\beta_1\in H_2(X;\ZZ)\\ \beta_2\in \Pi_{\geq 0}\\ \omega(\beta_1)>0\\ \varpi(\beta_1)+\beta_2=\beta+\beta'}}(-1)^{w_s(\beta_1)}&\GW_{\beta_1}\big(\Delta_a\otimes\Delta_b\otimes\Delta_\mu\otimes\bigotimes\limits_{p\in S_1}\Delta_{i_p}\big)\\
			\notag g^{\mu\nu}&\OGWbar_{\beta_2,k-1}^{\parentheses{|S_2|+1}}\big(\Gamma_\nu\otimes\bigotimes\limits_{q\in S_2}\Gamma_{i_q}\big)\\
			\notag+\sum\limits_{\substack{S_1\cupdot S_2=\\S\backslash\left\{m\right\}}}\sum_{\substack{k_1,k_2\in\ZZ_{\geq 0}\\ k_1+k_2=k-2}}\sum_{\substack{\beta_1,\beta_2\in\Pi_{\geq 0}\\ \beta_1+\beta_2=\beta+\beta'}}\binom{k-2}{k_1}&\OGWbar_{\beta_1,k_1}^{\parentheses{|S_1|+2}}\big(\Gamma_a\otimes\Gamma_b\otimes\bigotimes\limits_{p\in S_1}\Gamma_{i_p}\big)\\
			\notag&\OGWbar_{\beta_2,k_2+2}^{\parentheses{|S_2|}}\big(\bigotimes\limits_{q\in S_2}\Gamma_{i_q}\big),
		\end{align}
		where the right-hand side contains only terms of order lower than $(\omega(\beta),k,0,0)$. Again, we note that the energy gap axiom \eqref{Energy Gap} implies that such terms of lower order are of degree $\beta$ or have energy lower than or equal to $\omega(\beta)-\hbar$.
		
		We conclude that the coefficient $\OGWbar_{\beta+\beta',k-1}^{\parentheses{\ell'-1}}\big(\bigotimes_{p\in S}\Gamma_{i_p}\big)$ is computable via coefficients of $\Omega$ of order lower than $(\omega(\beta),k,0,0)$, with energy $\omega(\beta)$ or less than or equal to $\omega(\beta)-\hbar$, and closed Gromov-Witten invariants.
	\end{proof}

	\begin{proof}[Proof of Lemma \ref{All OGW are Computable lemma pt. 3}]
		First, note that if $\OGWbar_{\beta,k}^{\parentheses{0}}$ does not satisfy the degree axiom \eqref{Open Degree}, then $\OGWbar_{\beta,k}^{\parentheses{0}}$ vanishes and thus is directly computable. Also, if $\beta = \beta_0$ or $\beta\neq\beta_0$ and $\omega(\beta)<\hbar$, then $\OGW_{\beta,k}^{\parentheses{0}}$ is directly computable via the zero axiom \eqref{Open Zero} and the energy gap axiom \eqref{Energy Gap}. We may assume without loss of generality, then, that $\OGW_{\beta,k}^{\parentheses{0}}$ satisfies the degree axiom and that $\omega(\beta)\geq \hbar=\omega(\beta')$.
		
		Now, let $1\leq j\leq \ell'$. First, we note that $|\Gamma_{i_j}|\geq 2$. If this were not the case then, since $H^2(X;\RR)$ generates $H^*(X;\RR)$, we would have $|\Gamma_{i_j}|=0$, and therefore $\Gamma_{i_j} = 1$. By the fundamental class axiom \eqref{Open Fundamental Class}, since $\beta'\neq\beta_0$, the coefficient $\OGWbar_{\beta',0}^{\parentheses{\ell'}}\big(\Gamma_{i_1}\otimes\dots\otimes\Gamma_{i_\ell'}\big)$ would vanish. Next, we note that $|\Gamma_{i_{j}}|>2$. If this were not the case, we would have $|\Gamma_{i_{j}}|=2$, so by the divisor axiom \eqref{Open Divisor} we would get a contradiction to the minimality of $\ell'$. Thus, for all $1\leq j\leq \ell'$, we have $\left|\Gamma_{i_j}\right|>2$.
		
		In the following proof, we use the second WDVV relation \eqref{Open WDVV Boundary Constraint} to consider the relations between coefficients of $\Omega$ with degrees that add up to $\beta+\beta'$. We have two different cases to consider: $\ell'\geq 2$ and $\ell'=1$.
		
		Assume first that $\ell'\geq 2$, and take $a=i_1$ and $b=i_2$. We apply the second open WDVV relation -- let us equate the coefficients $T^{\beta+\beta'}s^{k-2}\prod\limits_{j=3}^{\ell'}t_{i_j}$ in \eqref{Open WDVV Boundary Constraints Polynomial After Expansion}. This gives
		\begin{align}
			\label{Open WDVV For Gamma_a, Gamma_b Minimal 1} -\sum_{\substack{S_1\cupdot S_2\\ =\left\{3,\dots,\ell'\right\}}}\sum_{\substack{k_1,k_2\in\ZZ_{\geq 0}\\ k_1+k_2 = k-2}}\sum_{\substack{\beta_1,\beta_2\in\Pi_{\geq 0}\\ \beta_1+\beta_2=\beta+\beta'}}\frac{1}{k_1!k_2!}&\OGWbar_{\beta_1,k_1+1}^{\parentheses{|S_1|+1}}\big(\Gamma_a\otimes\bigotimes\limits_{p\in S_1}\Gamma_{i_p}\big)\\
			&\notag\OGWbar_{\beta_2,k_2+1}^{\parentheses{|S_2|+1}}\big(\Gamma_b\otimes\bigotimes\limits_{q\in S_2}\Gamma_{i_q}\big)\\
			\label{Open WDVV For Gamma_a, Gamma_b Minimal 2} =\sum_{\substack{S_1\cupdot S_2\\ =\left\{3,\dots,\ell'\right\}}}\sum\limits_{0\leq \mu,\nu\leq \kappa} \sum\limits_{\substack{\beta_1\in H_2(X;\ZZ)\\ \beta_2\in\Pi_{\geq 0}\\ \varpi(\beta_1)+\beta_2=\beta+\beta'}}\frac{(-1)^{w_s(\beta_1)}}{(k-2)!}&\GW_{\beta_1}\big(\Delta_a\otimes\Delta_b\otimes\Delta_\mu\bigotimes\limits_{p\in S_1}\Delta_{i_p}\big)\\
			&\notag g^{\mu\nu}\OGWbar_{\beta_2,k-1}^{\parentheses{|S_2|+1}}\big(\Gamma_\nu\otimes\bigotimes\limits_{q\in S_2}\Gamma_{i_q}\big)\\
			\label{Open WDVV For Gamma_a, Gamma_b Minimal 3} -\sum_{\substack{S_1\cupdot S_2\\ =\left\{3,\dots,\ell'\right\}}}\sum_{\substack{k_2,k_2\in\ZZ_{\geq 0}\\ k_1+k_2=k-2}}\sum_{\substack{\beta_1,\beta_2\in\Pi_{\geq 0}\\ \beta_1+\beta_2=\beta+\beta'}}\frac{1}{k_1!k_2!}&\OGWbar_{\beta_1,k_1}^{\parentheses{|S_1|+2}}\big(\Gamma_a\otimes\Gamma_b\otimes\bigotimes\limits_{p\in S_1}\Gamma_{i_p}\big)\\
			&\notag\OGWbar_{\beta_2,k_2+2}^{\parentheses{|S_2|}}\big(\bigotimes\limits_{q\in S_2}\Gamma_{i_q}\big).
		\end{align}
		
		We now look for nonzero terms of order $(\omega(\beta), k, 0, 0)$ or higher, and again we do so by considering each summand separately. In \eqref{Open WDVV For Gamma_a, Gamma_b Minimal 1}, all terms are of the form
		\begin{equation}
			\label{Open WDVV For Gamma_a, Gamma_b Minimal 1, 1}\OGWbar_{\beta_1,k_1+1}^{\parentheses{|S_1|+1}}\big(\Gamma_a\otimes\bigotimes\limits_{p\in S_1}\Gamma_{i_p}\big)\OGWbar_{\beta_2,k_2+1}^{\parentheses{|S_2|+1}}\big(\Gamma_b\otimes\bigotimes\limits_{q\in S_2}\Gamma_{i_q}\big),
		\end{equation}
		for some $\beta_1,\beta_2\in\Pi_{\geq 0}$ that satisfy $\beta_1+\beta_2=\beta+\beta'$, some $k_1,k_2\in\ZZ_{\geq 0}$ that satisfy $k_1+k_2=k-2$, and some disjoint sets $S_1,S_2$ that satisfy $S_1\cupdot S_2=\left\{3,\dots,\ell'\right\}$. If one of the factors in \eqref{Open WDVV For Gamma_a, Gamma_b Minimal 1, 1} has an order $(\omega(\beta),k,0,0)$ or higher, one of the following must be true:
		\begin{enumerate}
			\item\label{Open WDVV For Gamma_a, Gamma_b Minimal 1, 1, case a} {$\omega(\beta_1)>\omega(\beta)$:} Then
			\[  \omega(\beta_2) = \omega(\beta)+\omega(\beta')-\omega(\beta_1)  <\omega(\beta')=\hbar. \]
			By the energy gap axiom \eqref{Energy Gap}, if the value $\OGWbar_{\beta_2,k_2+2}\big(\Gamma_b\otimes\bigotimes\limits_{q\in S_2}\Gamma_{i_q}\big)$ does not vanish, then $\beta_2 = \beta_0$. By the zero axiom, and since $k_2+1\geq 1$ but $\Gamma_b\neq 1$, we deduce that $\OGWbar_{\beta_2,k_2+1}^{\parentheses{|S_2|+1}}\big(\Gamma_b\otimes\bigotimes\limits_{q\in S_2}\Gamma_{i_q}\big)$ still vanishes. We deduce that in this case \eqref{Open WDVV For Gamma_a, Gamma_b Minimal 1, 1} reduces to zero.
			\item\label{Open WDVV For Gamma_a, Gamma_b Minimal 1, 1, case b} {$\omega(\beta_1)=\omega(\beta)$:} In this case $k_1\leq k-2$ and therefore $k_1+1\leq k-2+1=k-1<k$, so the term $\OGWbar_{\beta_1,k_1+1}^{\parentheses{|S_1|+1}}\big(\Gamma_a\otimes\bigotimes\limits_{p\in S_1}\Gamma_{i_p}\big)$ has order lower than $(\omega(\beta),k,0,0)$. As for $\OGWbar_{\beta_2,k_2+1}^{\parentheses{|S_2|+1}}\big(\Gamma_b\otimes\bigotimes\limits_{q\in S_2}\Gamma_{i_q}\big)$, note that in this case $\omega(\beta_2)=\omega(\beta')=\hbar\leq\omega(\beta)$, and $k_2+1\leq k-1<k$, so this term also has lower order than $(\omega(\beta),k,0,0)$. We conclude that in this case \eqref{Open WDVV For Gamma_a, Gamma_b Minimal 1, 1} has no higher order factors.
			
			\item\label{Open WDVV For Gamma_a, Gamma_b Minimal 1, 1, case c} {$\omega(\beta_2)>\omega(\beta)$:} By an identical argument to case \ref{Open WDVV For Gamma_a, Gamma_b Minimal 1, 1, case a}, the term \eqref{Open WDVV For Gamma_a, Gamma_b Minimal 1, 1} reduces to zero in this case.
			\item\label{Open WDVV For Gamma_a, Gamma_b Minimal 1, 1, case d} {$\omega(\beta_2)=\omega(\beta)$:} By an identical argument to case \ref{Open WDVV For Gamma_a, Gamma_b Minimal 1, 1, case b}, the term \eqref{Open WDVV For Gamma_a, Gamma_b Minimal 1, 1} has order lower than $(\omega(\beta),k,0,0)$ in this case.
		\end{enumerate}
		We see that \eqref{Open WDVV For Gamma_a, Gamma_b Minimal 1, 1} has no non-vanishing terms where one of the factors has order greater than or equal to $(\omega(\beta),k,0,0)$.
		
		In \eqref{Open WDVV For Gamma_a, Gamma_b Minimal 2}, the higher order terms are of the form
		\begin{equation}
			\label{Open WDVV For Gamma_a, Gamma_b Minimal 2, 1}\sum\limits_{0\leq \mu,\nu\leq K } \GW_{\beta_1}\big(\Delta_a\otimes\Delta_b\otimes\Delta_\mu\otimes\bigotimes\limits_{p\in S_1}\Delta_{i_p}\big)g^{\mu\nu}\OGWbar_{\beta_2,k-1}^{\parentheses{|S_2|+1}}\big(\Gamma_\nu\otimes\bigotimes\limits_{q\in S_2}\Gamma_{i_q}\big),
		\end{equation}
		for some $\beta_1\in H_2(X;\ZZ)$ and $\beta_2\in\Pi_{\geq 0}$ that satisfy $\varpi(\beta_1)+\beta_2=\beta+\beta'$, and some disjoint sets $S_1,S_2$ that satisfy $S_1\cupdot S_2=\left\{3,\dots,\ell'\right\}$. For the $\OGWbar$ factor in \eqref{Open WDVV For Gamma_a, Gamma_b Minimal 2, 1} to have order $(\omega(\beta),k,0,0)$ or higher, one of the following must hold:
		\begin{enumerate}
			\item  {$\beta_1 = 0$:} By the zero axiom \eqref{Closed Zero}, term \eqref{Open WDVV For Gamma_a, Gamma_b Minimal 2, 1} does not vanish in this case only if $|S_1|=0$. This implies that $S_2 = \left\{3,\dots,\ell'\right\}$, so \eqref{Open WDVV For Gamma_a, Gamma_b Minimal 2, 1} reduces, by Lemma \ref{g is block matrix when needed}, to
			\begin{align*}
				\sum\limits_{0\leq \mu,\nu\leq \kappa} \GW_{0}&\big(\Delta_a\otimes\Delta_b\otimes\Delta_\mu\big)g^{\mu\nu}\OGWbar_{\beta_2,k-1}^{\parentheses{\ell'-1}}\big(\Gamma_\nu\otimes\Gamma_{i_3}\otimes\dots\otimes\Gamma_{i_{\ell'}}\big) \\
				=&\sum\limits_{0\leq \mu,\nu\leq \kappa} \big(\int_X\Delta_a\wedge\Delta_b\wedge\Delta_\mu\big) g^{\mu\nu}\OGWbar_{\beta+\beta',k-1}^{\parentheses{\ell-1}}\big(\Gamma_\nu\otimes\Gamma_{i_3}\otimes\dots\otimes\Gamma_{i_{\ell'}}\big) \\
				=&\OGWbar_{\beta+\beta',k-1}^{\parentheses{\ell'-1}}\big(\big(\sum\limits_{0\leq \mu,\nu\leq \kappa } \int_X\Delta_a\wedge\Delta_b\wedge\Delta_\mu\big) g^{\mu\nu}\Gamma_\nu\otimes\Gamma_{i_3}\otimes\dots\otimes\Gamma_{i_{\ell'}}\big) \\
				=&\OGWbar_{\beta+\beta',k-1}^{\parentheses{\ell'-1}}\big(\Delta_a\wedge\Delta_b\otimes\Gamma_{i_3}\otimes\dots\otimes\Gamma_{i_{\ell'}}\big)\\
				=&\OGWbar_{\beta+\beta',k-1}^{\parentheses{\ell'-1}}\big(\Gamma_{i_1}\wedge\Gamma_{i_2}\otimes\Gamma_{i_3}\otimes\dots\otimes\Gamma_{i_{\ell'}}\big).
			\end{align*}
			By linearity of the homomorphism $\OGW_{\beta+\beta',k-1}^{\parentheses{\ell'-1}}$, we may assume without loss of generality that $\Gamma_{i_1}\wedge\Gamma_{i_2}$ is a basis element of $H^*(X;\RR)$. By Lemma \ref{All OGW are Computable lemma pt. 3, aux}, the term $\OGWbar_{\beta+\beta',k-1}^{\parentheses{\ell'-1}}\big(\Gamma_{i_1}\wedge\Gamma_{i_2}\otimes\Gamma_{i_3}\otimes\dots\otimes\Gamma_{i_{\ell'}}\big)$ is computable by coefficients of order lower than $(\omega(\beta),k,0,0)$,  with energy $\omega(\beta)$ or energy lower than or equal to $\omega(\beta)-\hbar$, and closed Gromov-Witten invariants. Therefore, \eqref{Open WDVV For Gamma_a, Gamma_b Minimal 2, 1} is computable, in this case, by coefficients of order lower than $(\omega(\beta),k,0,0)$, with energy $\omega(\beta)$ or energy lower than or equal to $\omega(\beta)-\hbar$, and closed Gromov-Witten invariants.
			\item  {$\beta_1\neq 0$ and $\omega(\beta_2)>\omega(\beta)$:} In this case,
			\[
			\omega(\varpi(\beta_1)) = \omega(\beta)+\omega(\beta') - \omega(\beta_2)<\omega(\beta')=\hbar.
			\]
			By Lemma \ref{If varpi(beta)=0 then beta=0}, we deduce that the invariant $\GW_{\beta_1}\big(\Delta_a\otimes\Delta_b\otimes\Delta_\mu\otimes\bigotimes\limits_{p\in S_1}\Delta_{i_p}\big)$ vanishes. Thus, in this case, the term \eqref{Open WDVV For Gamma_a, Gamma_b Minimal 2, 1} reduces to zero.
			\item  {$\omega(\beta_2)=\omega(\beta)$:} In this case, the coefficient  $\OGWbar_{\beta_2,k-1}^{\parentheses{|S_2|+1}}\big(\Gamma_\nu\otimes\bigotimes\limits_{q\in S_2}\Gamma_{i_q}\big)$ has lower order than $(\omega(\beta),k,0,0)$.
		\end{enumerate}
		We conclude that all terms in \eqref{Open WDVV For Gamma_a, Gamma_b Minimal 2} are computable via coefficients of $\Omega$ of order lower than $(\omega(\beta),k,0,0)$.
		
		In \eqref{Open WDVV For Gamma_a, Gamma_b Minimal 3}, the higher order terms are of the form
		\begin{equation}
			\label{Open WDVV For Gamma_a, Gamma_b Minimal 3, 1}\OGWbar_{\beta_1,k_1}^{\parentheses{|S_1|+2}}\big(\Gamma_a\otimes\Gamma_b\otimes\bigotimes\limits_{p\in S_1}\Gamma_{i_p}\big)\OGWbar_{\beta_2,k_2+2}^{\parentheses{|S_2|}}\big(\bigotimes\limits_{q\in S_2}\Gamma_{i_q}\big),
		\end{equation}
		for some $\beta_1,\beta_2\in\Pi_{\geq 0}$ that satisfy $\beta_1+\beta_2=\beta+\beta'$, some $k_1,k_2\in\ZZ_{\geq 0}$ that satisfy $k_1+k_2=k-2$, and some disjoint sets $S_1,S_2$ that satisfy $S_1\cupdot S_2=\left\{3,\dots,\ell'\right\}$. For one of the factors in \eqref{Open WDVV For Gamma_a, Gamma_b Minimal 3, 1} to have order $(\omega(\beta),k,0,0)$ or higher, one of the following must be true:
		\begin{enumerate}
			\item\label{Open WDVV For Gamma_a, Gamma_b Minimal 3, case a} {$\omega(\beta_1)>\omega(\beta)$:} In this case, we deduce $\omega(\beta_2)<\omega(\beta')=\hbar$. If $\beta_2 = \beta_0$, note that $k_2+2\geq 2$, so the zero axiom \eqref{Open Zero} implies that $\OGWbar_{\beta_2,k_2+2}^{\parentheses{|S_2|}}\big(\bigotimes\limits_{q\in S_2}\Gamma_{i_q}\big)$ vanishes. If $\beta_2\neq\beta_0$, the energy gap axiom \eqref{Energy Gap} implies that $\OGWbar_{\beta_2,k_2+2}^{\parentheses{|S_2|}}\big(\bigotimes\limits_{q\in S_2}\Gamma_{i_q}\big)$ vanishes. In any case, \eqref{Open WDVV For Gamma_a, Gamma_b Minimal 3, 1} reduces to zero.
			\item\label{Open WDVV For Gamma_a, Gamma_b Minimal 3, case b} {$\omega(\beta_1)=\omega(\beta)$:} In this case  $k_1\leq k-2<k$, so $\OGWbar_{\beta_1,k_1}^{\parentheses{|S_1|+2}}\big(\Gamma_a\otimes\Gamma_b\otimes\bigotimes\limits_{p\in S_1}\Gamma_{i_p}\big)$ has lower order than $(\omega(\beta),k,0,0)$. As for $\OGWbar_{\beta_2,k_2+2}^{\parentheses{|S_2|}}\big(\bigotimes\limits_{q\in S_2}\Gamma_{i_q}\big)$, note that $\omega(\beta_2) = \omega(\beta')$ in this case.
			
			If $\omega(\beta)>\omega(\beta') = \omega(\beta_2)$, then the order of $\OGWbar_{\beta_2,k_2+2}^{\parentheses{|S_2|}}\big(\bigotimes\limits_{q\in S_2}\Gamma_{i_q}\big)$ is lower order than $(\omega(\beta),k,0,0)$.
			
			Otherwise, $\omega(\beta) = \omega(\beta') = \omega(\beta_2)$. In this case, for $\OGWbar_{\beta_2,k_2+2}^{\parentheses{|S_2|}}\big(\bigotimes\limits_{q\in S_2}\Gamma_{i_q}\big)$ to have order $(\omega(\beta),k,0,0)$ or higher, we must have $k_2 +2 =k$. This implies, in turn, that $k_1=0$.
			By minimality of the pair $(\beta',\ell')$, we deduce that $|S_1|+2 = \ell'$, so $S_1 = \left\{3,\dots,\ell'\right\}$ and $S_2 = \emptyset$. Let us separate into two sub-cases:
			
			\begin{enumerate}
				\item \label{nonzero multiple of OGW_beta,k 1} {$\beta_1 = \beta'$:} Then \eqref{Open WDVV For Gamma_a, Gamma_b Minimal 3, 1} reduces to
				\[ \OGWbar_{\beta', 0}^{\parentheses{\ell'}}\big(\Gamma_{i_1}\otimes\Gamma_{i_2}\otimes\dots\otimes\Gamma_{i_{\ell'}}\big)\OGWbar_{\beta,k}^{\parentheses{0}}, \]
				which is a nonzero multiple of the term we are looking to compute.
				\item  {$\beta_1\neq\beta'$:} Then (since $\omega(\beta_1) = \omega(\beta')$ and since $\Pi = H_2(X,L;\ZZ)/(\ker\omega\oplus\ker\mu)$) we have $\mu(\beta_1)\neq \mu(\beta')$. By the degree axiom \eqref{Open Degree}, the fact that the value $\OGWbar_{\beta', 0}^{\parentheses{\ell'}}\big(\Gamma_{i_1}\otimes\Gamma_{i_2}\otimes\dots\otimes\Gamma_{i_\ell'}\big)$ does not vanish implies that
				\[ n-3 + \mu(\beta') + 2\ell' = \sum\limits_{j=1}^{\ell'}\left|\Gamma_{i_j}\right|, \]
				and therefore
				\[ n-3 + \mu(\beta_1) + 2\ell' \neq \sum\limits_{j=1}^{\ell'}\left|\Gamma_{i_j}\right|. \]
				
				We deduce by the degree axiom, again, that
				\[  \OGWbar_{\beta_1, 0}^{\parentheses{\ell'}}\big(\Gamma_{i_1}\otimes\Gamma_{i_2}\otimes\dots\otimes\Gamma_{i_\ell'}\big) = 0, \]
				and therefore \eqref{Open WDVV For Gamma_a, Gamma_b Minimal 3, 1} vanishes in this case.
			\end{enumerate}
			
			We conclude that there are no non-vanishing higher order factors in \eqref{Open WDVV For Gamma_a, Gamma_b Minimal 3, 1} in this case, except for a nonzero multiple of $\OGW_{\beta,k}^{\parentheses{0}}$.
			\item\label{Open WDVV For Gamma_a, Gamma_b Minimal 3, case c} {$\omega(\beta_2)>\omega(\beta)$:} In this case, we deduce $\omega(\beta_1)<\omega(\beta')=\hbar$. If $\beta_1 = \beta_0$, the zero axiom \eqref{Open Zero} and the fact that $|S_1|+2\geq 2$ implies that $k_1 = 0$ and that
			\[ \OGWbar_{\beta_1,k_1}^{\parentheses{|S_1|+2}}\big(\Gamma_a\otimes\Gamma_b\bigotimes\limits_{p\in S_1}\Gamma_{i_p}\big) = P_\RR(\Gamma_a\wedge\Gamma_b) = 0, \]
			since $P_\RR$ vanishes on $H^*(X,\RR)$. If $\beta_1\neq\beta_0$, the energy gap axiom \eqref{Energy Gap} implies that $\OGWbar_{\beta_1,k_1}^{\parentheses{|S_1|+2}}\big(\Gamma_a\otimes\Gamma_b\bigotimes\limits_{p\in S_1}\Gamma_{i_p}\big)$ vanishes. In any case, \eqref{Open WDVV For Gamma_a, Gamma_b Minimal 3, 1} reduces to zero.
			
			\item\label{Open WDVV For Gamma_a, Gamma_b Minimal 3, case d} {$\omega(\beta_2)=\omega(\beta)$:} Note that $\omega(\beta_1) = \omega(\beta')$ in this case. Let us look into the order of $\OGWbar_{\beta_1,k_1}^{\parentheses{|S_1|+2}}\big(\Gamma_a\otimes\Gamma_b\otimes\bigotimes\limits_{p\in S_1}\Gamma_{i_p}\big)$ first.
			
			If $\omega(\beta)>\omega(\beta')$, then the coefficient  $\OGWbar_{\beta_1,k_1}^{\parentheses{|S_1|+2}}\big(\Gamma_a\otimes\Gamma_b\otimes\bigotimes\limits_{p\in S_1}\Gamma_{i_p}\big)$ has lower order than $(\omega(\beta),k,0,0)$.
			
			If $\omega(\beta) = \omega(\beta')$, we note that $k_1\leq k-2<k$, and therefore the coefficient $\OGWbar_{\beta_1,k_1}^{\parentheses{|S_1|+2}}\big(\Gamma_a\otimes\Gamma_b\otimes\bigotimes\limits_{p\in S_1}\Gamma_{i_p}\big)$ has lower order than $(\omega(\beta),k,0,0)$.
			
			As for $\OGWbar_{\beta_2,k_2+2}^{\parentheses{|S_2|}}\big(\bigotimes\limits_{q\in S_2}\Gamma_{i_q}\big)$, for this term to have order $(\omega(\beta),k,0,0)$ or higher, we must have $k_2+2=k$, and thus $k_1 = 0$.
			
			Let us now divide into three different sub-cases:
			\begin{enumerate}
				\item  {$|S_2|>0$:} In this case, $|S_1|+2<\ell'-2+2 = \ell'$. By minimality of the pair $(\beta',\ell')$, we deduce that $\OGWbar_{\beta_1,0}^{\parentheses{|S_1|+2}}\big(\Gamma_a\otimes\Gamma_b\otimes\bigotimes\limits_{p\in S_1}\Gamma_{i_p}\big)$ vanishes, and therefore the term \eqref{Open WDVV For Gamma_a, Gamma_b Minimal 3, 1} reduces to zero.
				\item  {$|S_2|=0$ and $\beta_2 = \beta$:} Then $S_1=\left\{3,\dots,\ell'\right\}$ and \eqref{Open WDVV For Gamma_a, Gamma_b Minimal 3, 1} reduces to
				\begin{equation*}
					\OGWbar_{\beta',0}^{\parentheses{\ell'}}\left(\Gamma_{i_1}\otimes\Gamma_{i_2}\otimes\Gamma_{i_3}\otimes\dots\otimes\Gamma_{i_{\ell'}}\right)\OGWbar_{\beta,k}^{\parentheses{0}},
				\end{equation*}
				which by our assumption is a nonzero multiple of $\OGWbar_{\beta,k}^{\parentheses{0}}$. Note that this is the case already considered in \eqref{nonzero multiple of OGW_beta,k 1}.
				\item  {$|S_2|=0$ and $\beta_2 \neq \beta$:} Then $\beta_1\neq \beta'$ but $\omega(\beta_1) = \omega(\beta')$. Since
				\[  \Pi = H_2(X,L;\ZZ)/(\ker\omega\oplus\ker\mu), \]
				we deduce that $\mu(\beta_1)\neq \mu(\beta')$. By the degree axiom \eqref{Open Degree}, the fact that the value $\OGWbar_{\beta', 0}^{\parentheses{\ell'}}\big(\Gamma_{i_1}\otimes\Gamma_{i_2}\otimes\dots\otimes\Gamma_{i_\ell'}\big)$ does not vanish implies that
				\[ n-3 + \mu(\beta') + 2\ell' = \sum\limits_{j=1}^{\ell'}\left|\Gamma_{i_j}\right|, \]
				and therefore
				\[ n-3 + \mu(\beta_1) + 2\ell' \neq \sum\limits_{j=1}^{\ell'}\left|\Gamma_{i_j}\right|. \]
				
				We deduce by the degree axiom, again, that
				\[  \OGWbar_{\beta_1, 0}^{\parentheses{\ell'}}\big(\Gamma_{i_1}\otimes\Gamma_{i_2}\otimes\dots\otimes\Gamma_{i_{\ell'}}\big) = 0. \]
			\end{enumerate}
			
			We conclude that in this case the term \eqref{Open WDVV For Gamma_a, Gamma_b Minimal 3, 1} has no terms with factors of order higher than $(\omega(\beta),k,0,0)$, and the only non-vanishing term with a factor of order $(\omega(\beta),k,0,0)$ is a nonzero multiple of $\OGWbar_{\beta,k}^{\parentheses{0}}$.
		\end{enumerate}
		
		We conclude that \eqref{Open WDVV For Gamma_a, Gamma_b Minimal 3} has no terms with factors of order higher than $(\omega(\beta),k,0,0)$, and the only term of order $(\omega(\beta),k,0,0)$ is a nonzero multiple of $\OGWbar_{\beta,k}^{\parentheses{0}}$.
		
		We therefore get the relation
		\begin{align}
			A\cdot\OGWbar_{\beta,k}^{\parentheses{0}} = -\sum_{\substack{S_1\cupdot S_2\\ =\left\{3,\dots,\ell'\right\}}}\sum_{\substack{k_1,k_2\in\ZZ_{\geq 0}\\ k_1+k_2 = k-2}}\sum_{\substack{\beta_1,\beta_2\in\Pi_{\geq 0}\\ \beta_1+\beta_2=\beta+\beta'}}\binom{k-2}{k_1}&\OGWbar_{\beta_1,k_1+1}^{\parentheses{|S_1|+1}}\big(\Gamma_a\otimes\bigotimes\limits_{p\in S_1}\Gamma_{i_p}\big)\\
			&\notag\OGWbar_{\beta_2,k_2+1}^{\parentheses{|S_2|+1}}\big(\Gamma_b\otimes\bigotimes\limits_{q\in S_2}\Gamma_{i_q}\big)\\
			\notag-\sum_{\substack{S_1\cupdot S_2\\ =\left\{3,\dots,\ell'\right\}}}\sum\limits_{0\leq \mu,\nu\leq \kappa} \sum\limits_{\substack{\beta_1\in H_2(X;\ZZ)\\ \beta_2\in\Pi_{\geq 0}\\ \varpi(\beta_1)+\beta_2=\beta+\beta'}}(-1)^{w_s(\beta_1)}&\GW_{\beta_1}\big(\Delta_a\otimes\Delta_b\otimes\Delta_\mu\bigotimes\limits_{p\in S_1}\Delta_{i_p}\big)\\
			\notag g^{\mu\nu}&\OGWbar_{\beta_2,k-1}^{\parentheses{|S_2|+1}}\big(\Gamma_\nu\otimes\bigotimes\limits_{q\in S_2}\Gamma_{i_q}\big)\\
			\notag+\sum_{\substack{S_1\cupdot S_2\\ =\left\{3,\dots,\ell'\right\}}}\sum_{\substack{k_2,k_2\in\ZZ_{\geq 0}\\ k_1+k_2=k-2}}\sum_{\substack{\beta_1,\beta_2\in\Pi_{\geq 0}\\ (\beta_2,k_2)\neq(\beta,k-2)\\ \beta_1+\beta_2=\beta+\beta'}}\binom{k-2}{k_1}&\OGWbar_{\beta_1,k_1}^{\parentheses{|S_1|+2}}\big(\Gamma_a\otimes\Gamma_b\otimes\bigotimes\limits_{p\in S_1}\Gamma_{i_p}\big)\\
			\notag&\OGWbar_{\beta_2,k_2+2}^{\parentheses{|S_2|}}\big(\bigotimes\limits_{q\in S_2}\Gamma_{i_q}\big),
		\end{align}
		where $	A = \OGWbar_{\beta',0}^{\parentheses{\ell'}}\big(\Gamma_{i_1}\otimes\dots\otimes\Gamma_{i_{\ell'}}\big)\neq 0$ and the right-hand side contains only terms of lower order. Again, note that the energy gap axiom \eqref{Energy Gap} implies that such terms of lower order have energy $\omega(\beta)$ or have energy lower than or equal to $\omega(\beta)-\hbar$.
		
		We conclude that the coefficient $\OGWbar_{\beta,k}^{\parentheses{0}}$ is computable via coefficients of $\Omega$ of a lower order, with energy $\omega(\beta)$ or energy lower than or equal to $\omega(\beta)-\hbar$, and closed Gromov-Witten invariants.
		
		Assume now that $\ell'=1$. Note that the Poincare dual to $\beta'$ is a divisor that satisfies $\int_{\beta'}\text{PD}(\beta')\neq 0$. Therefore, there exists some $1\leq a\leq K$ such that $\Gamma_a$ is a divisor satisfying $\int_{\beta'}\Gamma_a\neq 0$.
		
		Take $b=i_1$. Note that by the divisor axiom \eqref{Open Divisor},
		\begin{equation*}
			\OGWbar_{\beta',0}^{\parentheses{2}}\big(\Gamma_a\otimes\Gamma_b\big) = \big(\int_{\beta'}
			\Gamma_a\big)\cdot \OGWbar_{\beta',0}^{\parentheses{1}}\big(\Gamma_b\big)\neq 0.
		\end{equation*}
		
		We now apply the second open WDVV relation -- let us equate the coefficients of $T^{\beta+\beta'}s^{k-2}$ in \eqref{Open WDVV Boundary Constraints Polynomial After Expansion}. This gives
		\begin{align}
			\label{Open WDVV For Gamma_a, Gamma_b Minimal 1b} -\sum_{\substack{\beta_1,\beta_2\in\Pi_{\geq 0}\\ \beta_1+\beta_2=\beta+\beta'}}\sum_{\substack{k_1,k_2\in\ZZ_{\geq 0}\\ k_1+k_2 = k-2}}\frac{1}{k_1!k_2!}
			&\OGWbar_{\beta_1,k_1+1}^{\parentheses{1}}\big(\Gamma_a\big)\OGWbar_{\beta_2,k_2+1}^{\parentheses{1}}\big(\Gamma_b\big) \\
			\label{Open WDVV For Gamma_a, Gamma_b Minimal 2b} =\sum\limits_{0\leq \mu,\nu\leq \kappa}\sum\limits_{\substack{\beta_1\in H_2(X;\ZZ)\\ \beta_2\in\Pi_{\geq 0}\\ \varpi(\beta_1)+\beta_2=\beta+\beta'}}\frac{(-1)^{w_s(\beta_1)}}{(k-2)!}&\GW_{\beta_1}\big(\Delta_a\otimes\Delta_b\otimes\Delta_\mu\big)g^{\mu\nu}\OGWbar_{\beta_2,k-1}^{\parentheses{1}}\big(\Gamma_\nu\big)\\
			\label{Open WDVV For Gamma_a, Gamma_b Minimal 3b} -\sum_{\substack{\beta_1,\beta_2\in\Pi_{\geq 0}\\ \beta_1+\beta_2=\beta+\beta'}}\sum_{\substack{k_1,k_2\in\ZZ_{\geq 0}\\ k_1+k_2=k-2}}\frac{1}{k_1!k_2!}&\OGWbar_{\beta_1,k_1}^{\parentheses{2}}\big(\Gamma_a\otimes\Gamma_b\big)\OGWbar_{\beta_2,k_2+2}^{\parentheses{0}}.
		\end{align}
		
		Let us look for terms of order $(\omega(\beta), k, 0, 0)$ or higher. In \eqref{Open WDVV For Gamma_a, Gamma_b Minimal 1b}, the higher order terms must be of the form
		\begin{equation*}
			\label{Open WDVV For Gamma_a, Gamma_b Minimal 1b, 1}\OGWbar_{\beta_1,k_1+1}^{\parentheses{1}}\big(\Gamma_a\big)\OGWbar_{\beta_2,k_2+1}^{\parentheses{1}}\big(\Gamma_b\big)
		\end{equation*}
		for some $\beta_1,\beta_2\in\Pi_{\geq 0}$ that satisfy $\beta_1+\beta_2=\beta+\beta'$ and some $k_1,k_2\geq 0$ that satisfy $k_1+k_2=k-2$. For one of the factors in \eqref{Open WDVV For Gamma_a, Gamma_b Minimal 1b, 1} to have order $(\omega(\beta),k,0,0)$ or higher, one of the following must be true:
		\begin{enumerate}
			\item\label{Open WDVV For Gamma_a, Gamma_b Minimal 1b, 1, case a} {$\omega(\beta_1)>\omega(\beta)$:} In this case, $\omega(\beta_2)<\omega(\beta)=\hbar$. If $\beta_2 = \beta_0$, then since $\Gamma_b\neq 1$ and $k_2+1\geq 1$, we get by the zero axiom \eqref{Open Zero}:
			\begin{equation*}
				\OGWbar_{\beta_2,k_2+1}^{\parentheses{1}}\big(\Gamma_b\big) = 0.
			\end{equation*}
			If $\beta_2\neq \beta_0$, the coefficient $\OGWbar_{\beta_2,k_2+1}^{\parentheses{1}}\big(\Gamma_b\big)$ vanishes by the energy gap axiom \eqref{Energy Gap}. In any case, $\OGWbar_{\beta_2,k_2+1}^{\parentheses{1}}\big(\Gamma_b\big)$ vanishes, so \eqref{Open WDVV For Gamma_a, Gamma_b Minimal 1b, 1} reduces to zero in this case.
			
			\item\label{Open WDVV For Gamma_a, Gamma_b Minimal 1b, 1, case b} {$\omega(\beta_1)=\omega(\beta)$:} In this case $k_1\leq k-2$ and therefore $k_1+1\leq k-1<k$, so $\OGWbar_{\beta_1,k_1+1}^{\parentheses{1}}\big(\Gamma_a\big)$ has lower order than $(\omega(\beta),k,0,0)$. As for $\OGWbar_{\beta_2,k_2+1}^{\parentheses{1}}\big(\Gamma_b\big)$, note that in this case $\omega(\beta_2) = \omega(\beta')=\hbar\leq\omega(\beta)$.
			
			If $\omega(\beta)>\omega(\beta_2)$, then $\OGWbar_{\beta_2,k_2+1}^{\parentheses{1}}\big(\Gamma_b\big)$ has lower order than $(\omega(\beta),k,0,0)$. If $\omega(\beta)=\omega(\beta_2)$, note that $k_2+1\leq k-1<k$, so again $\OGWbar_{\beta_2,k_2+1}^{\parentheses{1}}\big(\Gamma_b\big)$ has lower order than $(\omega(\beta),k,0,0)$. In any case, we get that \eqref{Open WDVV For Gamma_a, Gamma_b Minimal 1b, 1} has lower order than $(\omega(\beta),k,0,0)$.
			\item\label{Open WDVV For Gamma_a, Gamma_b Minimal 1b, 1, case c} {$\omega(\beta_2)>\omega(\beta)$:} In this case, $\omega(\beta_1)<\omega(\beta)=\hbar$. If $\beta_1 = \beta_0$, then since $\Gamma_a\neq 1$ and $k_1+1\geq 1$, we get by the zero axiom \eqref{Open Zero}:
			\begin{equation*}
				\OGWbar_{\beta_1,k_1+1}^{\parentheses{1}}\big(\Gamma_a\big) = 0.
			\end{equation*}
			If $\beta_1\neq \beta_0$, the coefficient $\OGWbar_{\beta_1,k_1+1}^{\parentheses{1}}\big(\Gamma_a\big)$ vanishes by the energy gap axiom \eqref{Energy Gap}. In any case, $\OGWbar_{\beta_1,k_1+1}^{\parentheses{1}}\big(\Gamma_a\big)$ vanishes, so \eqref{Open WDVV For Gamma_a, Gamma_b Minimal 1b, 1} reduces to zero in this case.
			\item\label{Open WDVV For Gamma_a, Gamma_b Minimal 1b, 1, case d} {$\omega(\beta_2)=\omega(\beta)$:} By an identical argument to the case \ref{Open WDVV For Gamma_a, Gamma_b Minimal 1b, 1, case b}, the term \eqref{Open WDVV For Gamma_a, Gamma_b Minimal 1b, 1} has lower order than $(\omega(\beta),k,0,0)$ in this case.
		\end{enumerate}
		We conclude that \eqref{Open WDVV For Gamma_a, Gamma_b Minimal 1b} contains no terms with non-vanishing factors of order higher than or equal to $(\omega(\beta),k,0,0)$.
		
		In \eqref{Open WDVV For Gamma_a, Gamma_b Minimal 2b}, the higher order terms must be of the form
		\begin{equation}
			\label{Open WDVV For Gamma_a, Gamma_b Minimal 2b, 1}\sum\limits_{0\leq \mu,\nu\leq \kappa} \GW_{\beta_1}\left(\Delta_a\otimes\Delta_b\otimes\Delta_\mu\right)g^{\mu\nu}\OGWbar_{\beta_2,k-1}^{\parentheses{1}}\big(\Gamma_\nu\big),
		\end{equation}
		for some $\beta_1\in H_2(X;\RR)$ and $\beta_2\in\Pi_{\geq 0}$ that satisfy $\varpi(\beta_1)+\beta_2=\beta+\beta'$. For the $\OGWbar$ factor in \eqref{Open WDVV For Gamma_a, Gamma_b Minimal 2b, 1} to have order $(\omega(\beta),k,0,0)$ or higher, one of the following must hold:
		\begin{enumerate}
			\item  {$\beta_1 = 0$:} By the zero axiom \eqref{Closed Zero}, the term \eqref{Open WDVV For Gamma_a, Gamma_b Minimal 2b, 1} reduces in this case, by Lemma \ref{g is block matrix when needed}, to
			\begin{align*}
				\sum\limits_{0\leq \mu,\nu\leq \kappa} \GW_{0}&\big(\Delta_a\otimes\Delta_b\otimes\Delta_\mu\big)g^{\mu\nu}\OGWbar_{\beta_2,k-1}^{\parentheses{1}}\big(\Gamma_\nu\big) \\
				=&\sum\limits_{0\leq \mu,\nu\leq \kappa} \big(\int_X\Delta_a\wedge\Delta_b\wedge\Delta_\mu\big) g^{\mu\nu}\OGWbar_{\beta+\beta',k-1}^{\parentheses{1}}\big(\Gamma_\nu\big) \\
				=&\OGWbar_{\beta+\beta',k-1}^{\parentheses{1}}\big(\big(\sum\limits_{0\leq \mu,\nu\leq \kappa} \int_X\Delta_a\wedge\Delta_b\wedge\Delta_\mu\big) g^{\mu\nu}\Gamma_\nu\big) \\
				=&\OGWbar_{\beta+\beta',k-1}^{\parentheses{1}}\big(\Gamma_a\wedge\Gamma_b\big).
			\end{align*}
			Recall that we assume that $\left|\Gamma_{i_1}\right| = 2n$. Since $\left|\Gamma_a\right|=2>0$, we deduce that
			\[ \Gamma_{a}\wedge\Gamma_{b} = \Gamma_{a}\wedge\Gamma_{i_1} = 0. \]
			By linearity of the homomorphism $\OGW_{\beta+\beta',k-1}^{\parentheses{1}}$, it follows that the term
			\[  \OGWbar_{\beta+\beta',k-1}^{\parentheses{1}}\big(\Gamma_{a}\wedge\Gamma_{b}\big)\]
			vanishes. Thus, in this case, the term \eqref{Open WDVV For Gamma_a, Gamma_b Minimal 2b, 1} reduces to zero.
			
			\item  {$\beta_1\neq 0$ and $\omega(\beta_2)>\omega(\beta)$:} In this case,
			\[
			\omega(\varpi(\beta_1)) = \omega(\beta)+\omega(\beta') - \omega(\beta_2)<\omega(\beta')=\hbar.
			\]
			By Lemma \ref{If varpi(beta)=0 then beta=0}, we deduce that the invariant $\GW_{\beta_1}\big(\Delta_a\otimes\Delta_b\otimes\Delta_\mu\big)$ vanishes. Thus, in this case, the term \eqref{Open WDVV For Gamma_a, Gamma_b Minimal 2b, 1} reduces to zero.
			\item  {$\omega(\beta_2)=\omega(\beta)$:} In this case, $\OGWbar_{\beta_2,k-1}^{\parentheses{1}}\big(\Gamma_\nu\big)$ has lower order than $(\omega(\beta),k,0,0)$.
		\end{enumerate}

		We conclude that \eqref{Open WDVV For Gamma_a, Gamma_b Minimal 2b} has no non-vanishing terms where the $\OGWbar$ factor is of order higher or equal to $(\omega(\beta),k,0,0)$.
		
		In \eqref{Open WDVV For Gamma_a, Gamma_b Minimal 3b}, the higher order terms are of the form
		\begin{equation}
			\label{Open WDVV For Gamma_a, Gamma_b Minimal 3b, 1}\OGWbar_{\beta_1,k_1}^{\parentheses{2}}\big(\Gamma_a\otimes\Gamma_b\big)\OGWbar_{\beta_2,k_2+2}^{\parentheses{0}},
		\end{equation}
		for some $\beta_1,\beta_2\in\Pi_{\geq 0}$ that satisfy $\beta_1+\beta_2=\beta+\beta'$ and $k_1,k_2\geq 0$ that satisfy $k_1+k_2=k-2$. For one of the factors in \eqref{Open WDVV For Gamma_a, Gamma_b Minimal 3b, 1} to have order $(\omega(\beta),k,0,0)$ or higher, one of the following must be true:
		
		\begin{enumerate}
			\item\label{Open WDVV For Gamma_a, Gamma_b Minimal 3b, case a} {$\omega(\beta_1)>\omega(\beta)$:} In this case, we deduce that $\omega(\beta_2)<\omega(\beta)=\hbar$. If $\beta_2 = \beta_0$, note that $k_2+2\geq 2$, so the zero axiom \eqref{Open Zero} implies that $\OGWbar_{\beta_2,k_2+2}^{\parentheses{0}}$ vanishes. If $\beta_2\neq\beta_0$, the energy gap axiom \eqref{Energy Gap} implies that $\OGWbar_{\beta_2,k_2+2}^{\parentheses{0}}$ vanishes. In any case, \eqref{Open WDVV For Gamma_a, Gamma_b Minimal 3b, 1} reduces to zero.
			\item\label{Open WDVV For Gamma_a, Gamma_b Minimal 3b, case b} {$\omega(\beta_1)=\omega(\beta)$:} In this case  $k_1\leq k-2<k$, so $\OGWbar_{\beta_1,k_1}^{\parentheses{2}}\big(\Gamma_a\otimes\Gamma_b\big)$ has lower order than $(\omega(\beta),k,0,0)$. As for $\OGWbar_{\beta_2,k_2+2}^{\parentheses{0}}$, note that $\omega(\beta_2) = \omega(\beta')$ in this case.
			
			If $\omega(\beta)>\omega(\beta') = \omega(\beta_2)$, then $\OGWbar_{\beta_2,k_2+2}^{\parentheses{0}}$ has lower order than $(\omega(\beta),k,0,0)$.
			
			Otherwise, $\omega(\beta) = \omega(\beta') = \omega(\beta_2)$. In this case, for $\OGWbar_{\beta_2,k_2+2}^{\parentheses{0}}$ to have order $(\omega(\beta),k,0,0)$ or higher, we must have $k_2 +2 =k$, and therefore $k_1 = 0$. Let us separate into two sub-cases:
			
			\begin{enumerate}
				\item \label{nonzero multiple of OGW_beta,k , ell' = 1, 1} {$\beta_1 = \beta'$:} Then \eqref{Open WDVV For Gamma_a, Gamma_b Minimal 3b, 1} reduces to
				\[ \OGWbar_{\beta', 0}^{\parentheses{2}}\big(\Gamma_{a}\otimes\Gamma_{b}\big)\OGWbar_{\beta,k}^{\parentheses{0}}, \]
				which is a nonzero multiple of the term we are looking to compute.
				\item  {$\beta_1\neq\beta'$:} Then (since $\omega(\beta_1) = \omega(\beta')$ and since $\Pi = H_2(X,L;\ZZ)/(\ker\omega\oplus\ker\mu)$) we have $\mu(\beta_1)\neq \mu(\beta')$. By the degree axiom \eqref{Open Degree}, the fact that the value $\OGWbar_{\beta', 0}^{\parentheses{2}}\big(\Gamma_{a}\otimes\Gamma_{b}\big)$ does not vanish implies that
				\[ n-3 + \mu(\beta') + 4 = \left|\Gamma_{a}\right| +\left|\Gamma_b\right|, \]
				and therefore
				\[ n-3 + \mu(\beta_1) + 4 \neq \left|\Gamma_{a}\right| +\left|\Gamma_b\right|. \]
				
				We deduce by the degree axiom, again, that
				\[  \OGWbar_{\beta_1, 0}^{\parentheses{2}}\big(\Gamma_{a}\otimes\Gamma_{b}) = 0, \]
				and therefore \eqref{Open WDVV For Gamma_a, Gamma_b Minimal 3b, 1} vanishes in this case.
			\end{enumerate}
			
			We conclude that there are no non-vanishing terms with factors of higher order in \eqref{Open WDVV For Gamma_a, Gamma_b Minimal 3b, 1} in this case, except for a nonzero multiple of $\OGW_{\beta,k}^{\parentheses{0}}$.
			\item\label{Open WDVV For Gamma_a, Gamma_b Minimal 3b, case c} {$\omega(\beta_2)>\omega(\beta)$:} In this case, we deduce $\omega(\beta_1)<\omega(\beta)=\hbar$. If $\beta_1 = \beta_0$, the zero axiom \eqref{Open Zero} implies that $k_1 = 0$ and that
			\[ \OGWbar_{\beta_1,k_1}^{\parentheses{2}}\big(\Gamma_a\otimes\Gamma_b\big) = P_\RR(\Gamma_a\wedge\Gamma_b) = 0, \]
			since $P_\RR$ vanishes on $H^*(X,\RR)$. If $\beta_1\neq\beta_0$, the energy gap axiom \eqref{Energy Gap} implies that $\OGWbar_{\beta_1,k_1}^{\parentheses{2}}\big(\Gamma_a\otimes\Gamma_b\big)$ vanishes. In any case, \eqref{Open WDVV For Gamma_a, Gamma_b Minimal 3b, 1} reduces to zero.
			
			\item\label{Open WDVV For Gamma_a, Gamma_b Minimal 3b, case d} {$\omega(\beta_2)=\omega(\beta)$:} Note that $\omega(\beta_1) = \omega(\beta')$ in this case. Let us look into the order of $\OGWbar_{\beta_1,k_1}^{\parentheses{2}}\big(\Gamma_a\otimes\Gamma_b\big)$ first.
			
			If $\omega(\beta)>\omega(\beta')$, then $\OGWbar_{\beta_1,k_1}^{\parentheses{2}}\big(\Gamma_a\otimes\Gamma_b\big)$ has lower order than $(\omega(\beta),k,0,0)$.
			
			If $\omega(\beta) = \omega(\beta')$, then note that $k_1\leq k-2<k$, and therefore $\OGWbar_{\beta_1,k_1}^{\parentheses{2}}\big(\Gamma_a\otimes\Gamma_b\big)$ has lower order than $(\omega(\beta),k,0,0)$.
			
			As for $\OGWbar_{\beta_2,k_2+2}^{\parentheses{0}}$, for this term to have order $(\omega(\beta),k,0,0)$ or higher, we must have $k_2+2=k$, and thus $k_1 = 0$.
			
			Let us consider two different sub-cases:
			\begin{enumerate}
				\item  {$\beta_2 = \beta$:} Then \eqref{Open WDVV For Gamma_a, Gamma_b Minimal 3b, 1} reduces to
				\begin{equation*}
					\OGWbar_{\beta',0}^{\parentheses{2}}\big(\Gamma_{a}\otimes\Gamma_{b}\big)\OGWbar_{\beta,k}^{\parentheses{0}},
				\end{equation*}
				which is a nonzero multiple of $\OGWbar_{\beta,k}^{\parentheses{0}}$ by our assumption. Note that it is the same case that we considered in case \ref{nonzero multiple of OGW_beta,k , ell' = 1, 1}.
				\item  {$\beta_2 \neq \beta$:} Then $\beta_1\neq \beta'$ but $\omega(\beta_1) = \beta'$. Since $\Pi_{\geq 0} = H_2(X,L;\ZZ)/(\ker\omega\oplus\ker\mu)$, we deduce that $\mu(\beta_1)\neq \mu(\beta')$. By the degree axiom \eqref{Open Degree}, the fact that the value $\OGWbar_{\beta', 0}^{\parentheses{2}}\big(\Gamma_{a}\otimes\Gamma_{b}\big)$ does not vanish implies that
				\[ n-3 + \mu(\beta') + 4 = \left|\Gamma_{a}\right| + \left|\Gamma_{b}\right|, \]
				and therefore,
				\[ n-3 + \mu(\beta_1) + 4 \neq \left|\Gamma_{a}\right| + \left|\Gamma_{b}\right|. \]
				
				We deduce by the degree axiom, again, that
				\[  \OGWbar_{\beta_1, 0}^{\parentheses{2}}\big(\Gamma_{a}\otimes\Gamma_{b}\big) = 0. \]
			\end{enumerate}			
			
			We conclude that in this case the term \eqref{Open WDVV For Gamma_a, Gamma_b Minimal 3, 1} has no factors of order higher than $(\omega(\beta),k,0,0)$, and the only non-vanishing term of order $(\omega(\beta),k,0,0)$ is the one already considered in case \ref{nonzero multiple of OGW_beta,k , ell' = 1, 1}.
		\end{enumerate}

		In total, we see that \eqref{Open WDVV For Gamma_a, Gamma_b Minimal 3b, 1} has no terms with factors of order higher than $(\omega(\beta),k,0,0)$, and the only term with order $(\omega(\beta),k,0,0)$, is a nonzero multiple of $\OGWbar_{\beta,k}^{\parentheses{0}}$.
		
		We finally get the relation
		\begin{align}
			A\cdot\OGWbar_{\beta,k}^{\parentheses{0}} = -\sum_{\substack{\beta_1,\beta_2\in\Pi_{\geq 0}\\ \beta_1+\beta_2=\beta+\beta'}}\sum_{\substack{k_1,k_2\in\ZZ_{\geq 0}\\ k_1+k_2 = k-2}}\binom{k-2}{k_1}&\OGWbar_{\beta_1,k_1+1}^{\parentheses{1}}\big(\Gamma_a\big)\OGWbar_{\beta_2,k_2+1}^{\parentheses{1}}\big(\Gamma_b\big) \\
			\notag-\sum\limits_{0\leq \mu,\nu\leq \kappa}\sum\limits_{\substack{\beta_1\in H_2(X;\ZZ)\\ \beta_2\in\Pi_{\geq 0}\\ \varpi(\beta_1)+\beta_2=\beta+\beta'}}(-1)^{w_s(\beta_1)}&\GW_{\beta_1}\big(\Delta_a\otimes\Delta_b\otimes\Delta_\mu\big)g^{\mu\nu}\OGWbar_{\beta_2,k-1}^{\parentheses{1}}\big(\Gamma_\nu\big)\\
			\notag+\sum_{\substack{\beta_1,\beta_2\in\Pi_{\geq 0}\\ (\beta_2,k_2)\neq(\beta,k-2)\\ \beta_1+\beta_2=\beta+\beta'}}\sum_{\substack{k_1,k_2\in\ZZ_{\geq 0}\\ k_1+k_2=k-2}}\binom{k-2}{k_1}&\OGWbar_{\beta_1,k_1}^{\parentheses{2}}\big(\Gamma_a\otimes\Gamma_b\big)\OGWbar_{\beta_2,k_2+2}^{\parentheses{0}},
		\end{align}
		where $A = \OGWbar_{\beta',0}^{\parentheses{2}}\big(\Gamma_a\otimes\Gamma_b\big)\neq 0$, and the right-hand side only contains terms of lower order. Again, we note that the energy gap axiom \eqref{Energy Gap} implies that such terms of lower order have energy $\omega(\beta)$ or have energy lower than or equal to $\omega(\beta)-\hbar$.
		
		We conclude that the coefficient $\OGWbar_{\beta,k}^{\parentheses{0}}$ is computable via coefficients of $\Omega$ of a lower order, with energy $\omega(\beta)$ or energy lower than or equal to $\omega(\beta)-\hbar$, and closed Gromov-Witten invariants.
		
	\end{proof}

	Using the four reduction steps of Lemmas~\ref{All OGW are Computable lemma pt. 1}, \ref{All OGW are Computable lemma pt. 2}, \ref{All OGW are Computable lemma pt. 2, 2}, and \ref{All OGW are Computable lemma pt. 3}, we can now prove Theorems \ref{All OGW are computable 1 Fellowship of the OGW} and \ref{All OGW are computable 2 two OGW two computable}.
	\begin{proof}[Proof of Theorem \ref{All OGW are computable 1 Fellowship of the OGW}]
		Let $\beta\in\Pi_{\geq 0}$, $k,\ell\in\ZZ_{\geq 0}$, and $\gamma_1,\dots,\gamma_{{\ell}}\in H^*(X,L;\RR)$, and let us look at the coefficient $\OGWbar_{\beta,k}^{\parentheses{\ell}}\big(\gamma_1\otimes\dots\otimes\gamma_\ell\big)$. By Lemma \ref{All OGW are Computable lemma pt. 1}, we may assume without loss of generality that there exist $0\leq i_1,\dots,i_\ell\leq K$ such that $\gamma_1=\Gamma_{i_1},\dots,\gamma_\ell=\Gamma_{i_{\ell}}$, and that $\left|\Gamma_{i_1}\right|\geq\dots\geq\left|\Gamma_{i_\ell}\right|$.
		
		If $\ell\geq 2$, we deduce by Lemma \ref{All OGW are Computable lemma pt. 2} that $\OGWbar_{\beta,k}^{\parentheses{\ell}}\big(\gamma_1\otimes\dots\otimes\gamma_\ell\big)$ is computable via terms of lower order, with energy $\omega(\beta)$ or energy lower than or equal to $\omega(\beta)-\hbar$, and closed Gromov-Witten invariants.
		
		If $\ell= 1$ and $k\geq 1$, we deduce by Lemma \ref{All OGW are Computable lemma pt. 2, 2} that $\OGWbar_{\beta,k}^{\parentheses{\ell}}\big(\gamma_1\otimes\dots\otimes\gamma_\ell\big)$ is computable via terms of lower order, with energy $\omega(\beta)$ or energy lower than or equal to $\omega(\beta)-\hbar$, and closed Gromov-Witten invariants.
		
		In any other case, either $\ell=0$ and $k\in\ZZ_{\geq 0}$ or $\ell=1$ and $k=0$. The degree axiom \eqref{Open Degree} implies that we are left with precisely the set of initial open Gromov-Witten invariants that was stated in the theorem. As we have noted, at each step of the recursion we go down by at least $\hbar$ in the energy of the order, or by at least $1$ in the minimal degree of the inner constraints, the number of interior constrains, or number of boundary constraints of the order. Therefore, we reach the initial set of values in a finite number of steps. By \cite{KontsevitchManin1994}, any closed Gromov-Witten invariant can be computed in a finite number of steps from the set of initial values
		\begin{align*}
			\bigg\{\GW_{\beta}\left(\gamma_1\otimes\gamma_2\otimes\gamma_3\right)\, |\, &\beta\in H_2(X;\ZZ),\, \gamma_1,\gamma_2,\gamma_3\in H^*(X;\RR),\\
			&|\gamma_3|=2, 2n+2c_1(\beta)=|\gamma_1|+|\gamma_2|+|\gamma_3|\bigg\}.
		\end{align*}
		
		We see, therefore, that indeed any coefficient of $\Omega$ can be computed after a finite number of steps from the following sets of initial values:
		\begin{equation*}
			\left\{\OGWbar_{\beta,k}^{\parentheses{0}}\, \middle|\, \beta\in\Pi_{\geq 0},\, k\in\ZZ_{\geq 0},\, n-3+\mu(\beta)+k=kn\right\},
		\end{equation*}
		\begin{equation*}
			\left\{\OGWbar_{\beta,0}^{\parentheses{1}}\big(\Gamma_i\big)\, \middle|\, \beta\in\Pi_{\geq 0},\, 1\leq i\leq K,\, n-3+\mu(\beta)+2=|\Gamma_i|\right\},
		\end{equation*}
		\begin{align*}
			\bigg\{\GW_{\beta}\left(\gamma_1\otimes\gamma_2\otimes\gamma_3\right)\, |\, &\beta\in H_2(X;\ZZ),\, \gamma_1,\gamma_2,\gamma_3\in H^*(X;\RR),\\
			&\;|\gamma_3|=2, 2n+2c_1(\beta)=|\gamma_1|+|\gamma_2|+|\gamma_3|\bigg\}.
		\end{align*}
		Thus Theorem \ref{All OGW are computable 1 Fellowship of the OGW} is proven.
		
	\end{proof}
	
	\begin{proof}[Proof of Remark \ref{ell'=0 remark}]
		Assume there exists some $\beta'\in\Pi_{>0}$, with $\omega(\beta')=\hbar$, such that
		\[ \OGWbar_{\beta',0}^{\parentheses{0}} \neq 0. \]
		
		By the degree axiom \eqref{Open Degree}, we deduce that
		\[ n-3+\mu(\beta') = 0, \]
		and thus $\mu(\beta') = 3-n$, as stated.
	\end{proof}
	
	\begin{proof}[Proof of Theorem \ref{All OGW are computable 2 two OGW two computable}]
		Let $\beta\in\Pi_{\geq 0}$, $k,\ell\in\ZZ_{\geq 0}$, and $\gamma_1,\dots,\gamma_{{\ell}}\in H^*(X,L;\RR)$, and consider the coefficient $\OGWbar_{\beta,k}^{\parentheses{\ell}}\big(\gamma_1\otimes\dots\otimes\gamma_\ell\big)$. By Lemma \ref{All OGW are Computable lemma pt. 1}, we may assume without loss of generality that there exist some $0\leq i_1,\dots,i_\ell\leq K$ such that $\gamma_1=\Gamma_{i_1},\dots,\gamma_\ell=\Gamma_{i_{\ell}}$, and that $\left|\Gamma_{i_1}\right|\geq\dots\geq\left|\Gamma_{i_\ell}\right|$.
		
		If $\ell\geq 2$, we deduce by Lemma \ref{All OGW are Computable lemma pt. 2} that $\OGWbar_{\beta,k}^{\parentheses{\ell}}\big(\gamma_1\otimes\dots\otimes\gamma_\ell\big)$ is computable via terms of lower order, with energy $\omega(\beta)$ or energy lower than or equal to $\omega(\beta)-\hbar$, and closed Gromov-Witten invariants.
		
		If $\ell= 1$ and $k\geq 1$, we deduce by Lemma \ref{All OGW are Computable lemma pt. 2, 2} that $\OGWbar_{\beta,k}^{\parentheses{\ell}}\big(\gamma_1\otimes\dots\otimes\gamma_\ell\big)$ is computable via terms of lower order, with energy $\omega(\beta)$ or energy lower than or equal to $\omega(\beta)-\hbar$, and closed Gromov-Witten invariants.
		
		If $\ell=0$ and $k\geq 2$, by the assumption of the theorem there exists some $\beta'\in\Pi_{>0}$, some $\ell'\in\ZZ_{\geq 0}$, and some $0\leq i_1,...,i_{\ell'}\leq K$ for which $\OGWbar_{\beta',0}\left(\Gamma_{i_1}\otimes\dots\otimes\Gamma_{i_{\ell'}}\right)\neq 0$. By Lemma \ref{All OGW are Computable lemma pt. 3}, the coefficient $\OGWbar_{\beta,k}\left(\gamma_1\otimes\dots\otimes\gamma_\ell\right)$ is computable via terms of lower order, with energy $\omega(\beta)$ or energy lower than or equal to $\omega(\beta)-\hbar$, and closed Gromov-Witten invariants.
		
		In any other case, either $\ell=0$ and $k\in\{0,1\}$ or $\ell=1$ and $k=0$. The degree axiom \eqref{Open Degree} implies that we are left with precisely the set of initial open Gromov-Witten invariants that was stated in the theorem. As we have noted, at each step of the recursion we go down by at least $\hbar$ in the energy of the order, or by at least $1$ in the minimal degree of the inner constraints, the number of interior constrains, or number of boundary constraints of the order. Therefore, we reach the initial set of values in a finite number of steps. By \cite{KontsevitchManin1994}, any closed Gromov-Witten invariant may be computed in a finite number of steps from the set of initial values
		\begin{align*}
			\bigg\{\GW_{\beta}\left(\gamma_1\otimes\gamma_2\otimes\gamma_3\right)\, |\,& \beta\in H_2(X;\ZZ),\, \gamma_1,\gamma_2,\gamma_3\in H^*(X;\RR),\\
			&\;|\gamma_3|=2, 2n+2c_1(\beta)=|\gamma_1|+|\gamma_2|+|\gamma_3|\bigg\}.
		\end{align*}
		We see, therefore, that indeed any coefficient of $\Omega$ can be computed after a finite number of steps from the following sets of initial values:
		\begin{equation*}
			\left\{\OGWbar_{\beta,k}^{\parentheses{0}}\, \middle|\, \beta\in\Pi_{\geq 0},\, k\in\{0,1\},\, n-3+\mu(\beta)+k=kn\right\},
		\end{equation*}
		\begin{equation*}
			\left\{\OGWbar_{\beta,0}^{\parentheses{1}}\big(\Gamma_i\big)\, \middle|\, \beta\in\Pi_{\geq 0},\, 1\leq i\leq K,\, n-3+\mu(\beta)+2=|\Gamma_i|\right\},
		\end{equation*}
		\begin{align*}
			\bigg\{\GW_{\beta}\left(\gamma_1\otimes\gamma_2\otimes\gamma_3\right)\, |\, &\beta\in H_2(X;\ZZ),\, \gamma_1,\gamma_2,\gamma_3\in H^*(X;\RR),\\
			&\;|\gamma_3|=2, 2n+2c_1(\beta)=|\gamma_1|+|\gamma_2|+|\gamma_3|\bigg\}.
		\end{align*}
		
		Thus Theorem \ref{All OGW are computable 2 two OGW two computable} is proven.
	\end{proof}

	\section{Products of projective spaces}\label{products of projective spaces}

	Let $n_1,n_2$ be odd natural numbers with $n_1,n_2>1$. Take
	\[
	(X,L) := (\CC P^{n_1}\times\CC P^{n_2}, \RR P^{n_1}\times\RR P^{n_2}),
\quad n:=n_1+n_2.
	\]
	Equip $X$ with the symplectic form $\omega = \omega_1\oplus\omega_2$, where $\omega_i = \omega_{FS,n_i}$ is the Fubini-Study form on $\CC P^{n_i}$ for $i\in\{1,2\}$. Equip $X$ with the complex structure $J = J_{n_1}\oplus J_{n_2}$, where $J_{n_i}$ is the standard complex structure on $\CC P^{n_i}$ for $i\in\{1,2\}$. Then $L\subset X$ is a Lagrangian submanifold. Equip $L$ with the product relative spin structure.
	Identify the second relative homology $H_2(\CC P^{n_i},\RR P^{n_i};\ZZ)$ and the second absolute homology $H_2(\CC P^{n_i};\ZZ)$ with $\ZZ$ so that curves of positive energy are identified with $\ZZ_{>0}$. So, $\Pi = H_2(X,L;\ZZ)$ and $H_2(X;\ZZ)$ are identified with $\ZZ\oplus\ZZ$ accordingly.
	The map $\varpi:H_2(X;\ZZ)\rightarrow\Pi$ is given, under this identification, via multiplication by $2$ in each coordinate.
	The first Chern class is given by
	\[
	c_1(\tilde{a},\tilde{b}) = (n_1+1)\tilde{a}+(n_2+1)\tilde{b},
	\quad (\tilde{a},\tilde{b})\in H_2(X;\ZZ),
	\]
	and the Maslov index is given by
	\[
	\mu(a,b) = \left(n_1+1\right)a + \left(n_2 + 1\right)b,
	\quad (a,b)\in H_2(X,L;\ZZ).
	\]
	For $i\in\{1,2\}$, the cohomology of $\CC P^{n_i}$ is generated by the class of $\omega_i$. By the K\"{u}nneth formula, the cohomology of $X$ is generated, as a ring, by $H^2(X;\RR) = \left\{\pi_1^*[\omega_1], \pi_2^*[\omega_2]\right\}$, where $\pi_i:X\rightarrow \CC P^{n_i}$ is the projection to $\CC P^{n_i}$, for $i\in \{1,2\}$. Then
	\[
	\left\{\Delta_{i}\right\}_{i=1}^{(n_1+1)(n_2+1)} = \left\{\left(\pi_1^*[\omega_1]\right)^i\wedge\left(\pi_2^*[\omega_2]\right)^j\right\}_{0\leq i\leq n_1, 0\leq j\leq n_2}
	\]
	is a basis for $H^*(X;\RR)$, and
	\[ \left\{\Gamma_{i}\right\}_{i=1}^{(n_1+1)(n_2+1)}\cup\{\Gamma_\diamond\} = \left\{\left(\pi_1^*[\omega_1]\right)^i\wedge\left(\pi_2^*[\omega_2]\right)^j\right\}_{0\leq i\leq n_1, 0\leq j\leq n_2}\cup\{\Gamma_\diamond\} \]
	is a basis for $H^*(X,L;\RR)$. Let $K=(n_1+1)(n_2+1)$ and $N = K+1$, as in Section \ref{main theorems}.

	As mentioned in Section~\ref{subsection:Geometric Realization}, any definition of open Gromov-Witten invariants on $(X,L)$ is expected to arise from a Frobenius superpotential
	\begin{equation}\label{Omega in section 6}
		\Omega = \sum\limits_{\substack{\beta\in\Pi_{\geq 0}\\ k,\ell\in\ZZ_{\geq 0}}}\sum\limits_{\substack{\invecZN{r}\\ \sumsupto{\ell}{\vec{r}}}}
		T^\beta\frac{s^k}{k!}\prod\limits_{i=0}^{N}\frac{t_{N-i}^{r_{N-i}}}{(r_{N-i})!}
		\OGWbar^{(\ell)}_{\beta,k}\big(\bigotimes\limits_{i=0}^{N}\Gamma_{N-i}^{\otimes r_{N-i}}\big).
	\end{equation}
	Since $n$ is even, it follows from the degree axiom~\eqref{Open Degree} that $\OGW_{\beta,k}(\gamma_1\otimes\dots\otimes\gamma_\ell)=0$ unless $k$ is odd. In the following, we use Theorem~\ref{All OGW are computable 1 Fellowship of the OGW} to show a stronger vanishing phenomenon.
	
	From now on, assume there exists a Frobenius superpotential $\Omega$ and write it in the form of~\eqref{Omega in section 6}.

	\begin{Proposition}\label{OGWs Vanish for n1 n2 are 5mod6}
		If $n_1,n_2\equiv 5\pmod 6$, then $\OGW_{\beta,k}^{(\ell)}(\gamma_1\otimes\dots\otimes\gamma_\ell)=0$ for all $k \in \ZZ_{\ge 0}$, $\gamma_j \in H^*(X,L;\RR)$, and $\beta \ne \beta_0$.
	\end{Proposition}

	\begin{Proposition}\label{OGWs Vanish for n1 n2 are 1mod2 and pointlike}
		Assume $\Omega$ is defined by means of a point-like bounding pair, as in~\cite{SolomonTukachinsky2019relative}, and assume $\Omega$ satisfies the deformation invariance axiom~\ref{Deformation invariance}.
		If $n_1,n_2 \equiv 1\pmod 2$ and $n_1,n_2>1$, then $\OGW_{\beta,k}^{(\ell)}(\gamma_1\otimes\dots\otimes\gamma_\ell)=0$ for all $k \in \ZZ_{\ge 0}$, $\gamma_j \in H^*(X,L;\RR)$, and $\beta \ne \beta_0$.
	\end{Proposition}

	To prove both propositions,
	we follow the computational steps in the proof of Theorem \ref{All OGW are computable 1 Fellowship of the OGW} and show inductively that all coefficients of $\Omega$ of the form $\OGWbar_{\beta, k}^{(\ell)}\left(\gamma_{1}\otimes\dots\otimes\gamma_{\ell}\right)$ with $\beta\neq\beta_0$ vanish. This will be done in three steps -- first, for the basis of the induction, we show that all open initial values prescribed by Theorem \ref{All OGW are computable 1 Fellowship of the OGW} vanish. Next, we show that for $\ell\geq 2$, if all invariants of nonzero degree and lower order vanish, then $\OGWbar_{\beta, k}^{(\ell)}\left(\gamma_{1}\otimes\dots\otimes\gamma_{\ell}\right)$ vanishes. Lastly, we show that for $\ell=1$ and $k\geq 1$, if all invariants of nonzero degree and lower order vanish, then $\OGWbar_{\beta, k}^{(\ell)}\left(\gamma_{1}\otimes\dots\otimes\gamma_{\ell}\right)$ vanishes.
	
	We begin with an auxiliary observation.
	\begin{Lemma}\label{Deformation invariance implies degrees positive}
		Assume $\Omega$ satisfies the deformation invariance axiom~\ref{Deformation invariance}. If $\OGWbar_{\beta,k}^{\parentheses{\ell}}\big(\gamma_1\otimes\dots\otimes\gamma_\ell\big)\ne 0$ then $\beta\in\ZZ_{\geq 0}\oplus\ZZ_{\geq 0}$.
	\end{Lemma}
	
	\begin{proof}
		Let $\beta=(a,b)\in H_2(X,L;\ZZ)$, and assume without loss of generality that $b<0$. For any $t>0$, let $\omega^{t} := \omega_{1}\oplus t\cdot\omega_{2}$. Note that $L$ remains Lagrangian with respect to $\omega^t$. By the deformation invariance axiom~\ref{Deformation invariance}, the coefficient $\OGWbar_{\beta,k}^{\parentheses{\ell}}\big(\gamma_1\otimes\dots\otimes\gamma_\ell\big)$ remains constant in $t$. But there exists $t>>0$ such that $\omega^t(a,b)<0$, so by the energy gap axiom \eqref{Energy Gap} we get that $\OGWbar_{\beta,k}^{\parentheses{\ell}}\big(\gamma_1\otimes\dots\otimes\gamma_\ell\big)=0$.
	\end{proof}

	We move to the basis of induction.

	\begin{Lemma}\label{Initial values vanish for n1 n2 5mod6}
		In the setup of Proposition~\ref{OGWs Vanish for n1 n2 are 5mod6},
		let $\OGW_{\beta,k}^{\parentheses{0}}$ or $\OGW_{\beta,0}^{\parentheses{1}}\left(\Gamma_{i}\right)$ be an initial value in one of the sets prescribed by Theorem~\ref{All OGW are computable 1 Fellowship of the OGW}. Then  this value vanishes.
	\end{Lemma}
	
	\begin{proof}
		As noted above, the dimension being even implies that $\OGW_{\beta,0}^{\parentheses{1}}\left(\Gamma_{i}\right)=0$ by the degree axiom.
		
		We wish to show that all invariants in the set
		\begin{equation}
			\label{OGW_beta,k I}\big\{\OGWbar_{(a,b),k}^{\parentheses{0}}\, \big|\, a,b\in\ZZ_{\geq 0},\,  n-3+\mu(a,b)+k=kn\big\}
		\end{equation}
		vanish.
		
		Let $\OGWbar_{(a,b),k}^{\parentheses{0}}$ be an invariant in the  set \eqref{OGW_beta,k I}, and assume by contradiction that it does not vanish. Then it satisfies the degree axiom \eqref{Open Degree}:
		\[ n-3+\mu(a,b)+k=kn. \]
		Substituting $n = n_1 + n_2$ and $\mu(a,b) = \left(n_1+1\right)a + \left(n_2+1\right)b$, we get
		\begin{equation}\label{Degree for OGW_{(a,b),k} I}
			n_1 + n_2 - 3 + \left(n_1+1\right)a + \left(n_2+1\right)b=k\left(n_1+n_2-1\right).
		\end{equation}
		Since $n_1,n_2\equiv 5\pmod6$, the right-hand side of~\eqref{Degree for OGW_{(a,b),k} I} is divisible by $3$, while the left-hand side is not.
		We deduce that all invariants in the  set \eqref{OGW_beta,k I} vanish.
	\end{proof}

	\begin{Lemma}\label{Initial values Vanish OMG}
		In the setup of Proposition~\ref{OGWs Vanish for n1 n2 are 1mod2 and pointlike},	
		let $\OGW_{\beta,k}^{\parentheses{0}}$ or $\OGW_{\beta,0}^{\parentheses{1}}\left(\Gamma_{i}\right)$ be an initial value in one of the sets prescribed by Theorem~\ref{All OGW are computable 1 Fellowship of the OGW}. Then  this value vanishes.
	\end{Lemma}

	\begin{proof}
		As noted above, the dimension being even implies that $\OGW_{\beta,0}^{\parentheses{1}}\left(\Gamma_{i}\right)=0$ by the degree axiom.
		
		We wish to show that all invariants in the set
		\begin{equation}
			\label{OGW_beta,k}\big\{\OGWbar_{(a,b),k}^{\parentheses{0}}\, \big|\, a,b\in\ZZ_{\geq 0},\,  n-3+\mu(a,b)+k=kn\big\}
		\end{equation}
		vanish.
		
		Let $\OGWbar_{(a,b),k}^{\parentheses{0}}$ be an invariant in the  set \eqref{OGW_beta,k}, and assume by contradiction that it does not vanish. By the zero axiom \eqref{Open Zero}, $(a,b)\neq (0,0)$. Note that $n = n_1 + n_2$ is even, and therefore by Remark 1.9 in \cite{SolomonTukachinsky2019relative}, the invariant $\OGWbar_{(a,b),k}^{\parentheses{0}}$ vanishes for $k\geq 2$. So, $k\leq 1$. Since this invariant does not vanish, it satisfies the degree axiom \eqref{Open Degree}:
		\[ n-3+\mu(a,b)+k=kn. \]
		Substituting $n = n_1 + n_2$ and $\mu(a,b) = \left(n_1+1\right)a + \left(n_2+1\right)b$, we get
		\begin{equation}\label{Degree for OGW_{(a,b),k}}
			n_1 + n_2 - 3 + \left(n_1+1\right)a + \left(n_2+1\right)b=k\left(n_1+n_2-1\right).
		\end{equation}
		Since $(a,b)\neq(0,0)$, Lemma~\ref{Deformation invariance implies degrees positive} implies $a>0$ or $b>0$. Assume without loss of generality that $a>0$, so $\left(n_1+1\right)a + \left(n_2+1\right)b\geq \left(n_1+1\right)$. Therefore
		\begin{equation}\label{a>0}
			n_1 + n_2 - 3 + \left(n_1+1\right)a + \left(n_2+1\right)b\geq n_1 + n_2 - 3 + \left(n_1+1\right) = 2n_1 + n_2 -2.
		\end{equation}
		On the other hand, $k\leq 1$, so
		\begin{equation}\label{k leq 1}
			k\left(n_1+n_2-1\right)\leq n_1+n_2-1.
		\end{equation}
		Combining \eqref{a>0} and \eqref{k leq 1}, we see that
		\[ 2n_1 + n_2 -2\leq n_1+n_2-1, \]
		so $n_1\leq 1$, contradicting our assumption. We deduce that all invariants in the  set \eqref{OGW_beta,k} vanish.
	\end{proof}

	We move on to prove the lemmas that make up the inductive step of the proof of Propositions~\ref{OGWs Vanish for n1 n2 are 5mod6} and~\ref{OGWs Vanish for n1 n2 are 1mod2 and pointlike}. We first prove the lemma relevant to invariants with two or more interior constraints.
	
	\begin{Lemma}
		\label{Everything is zero OMG Lemma pt. 2}
		In the setup of Proposition~\ref{OGWs Vanish for n1 n2 are 5mod6} or Proposition~\ref{OGWs Vanish for n1 n2 are 1mod2 and pointlike},
		let $\beta\in\Pi_{\geq 0}$,  $k,\ell\in\ZZ_{\geq 0}$, and $0\leq i_1,\dots,i_\ell\leq K$, and assume that $\beta\neq\beta_0$, $\ell\geq 2$ and $\left|\Gamma_{i_1}\right|\geq\dots\geq\left|\Gamma_{i_\ell}\right|$. Assume that for any $\beta'\in\Pi_{\geq 0}$,  $k',\ell'\in\ZZ_{\geq 0}$, and $0\leq i'_1,\dots,i'_\ell\leq K$, where $\beta\neq\beta_0$ and
		\[
		\left(\omega(\beta'),k',\ell',\left|\Gamma_{i'_1}\right|,\dots,\left|\Gamma_{i'_\ell}\right|\right)<\left(\omega(\beta),k,\ell,\left|\Gamma_{i_1}\right|,\dots,\left|\Gamma_{i_\ell}\right|\right),
		\]
		the coefficient $\OGWbar_{\beta',k'}^{\parentheses{\ell'}}(\Gamma_{i'_1}\otimes\dots\otimes\Gamma_{i'_\ell})$ vanishes. Then the coefficient $\OGWbar_{\beta,k}^{\parentheses{\ell}}(\Gamma_{i_1}\otimes\dots\otimes\Gamma_{i_\ell})$ also vanishes.
	\end{Lemma}
	
	\begin{proof}
		We go over the proof of Lemma~\ref{All OGW are Computable lemma pt. 2} and show that $\OGWbar_{\beta,k}^{\parentheses{\ell}}(\Gamma_{i_1}\otimes\dots\otimes\Gamma_{i_\ell})$ vanishes under our assumptions. There are three possible cases:
		\begin{itemize}
			\item  {$|\Gamma_{i_\ell}| = 0$:} In this case, $\Gamma_{i_\ell}=1$. Since we assume that $\beta\neq\beta_0$, we deduce that $\OGWbar_{\beta,k}^{\parentheses{\ell}}(\Gamma_{i_1}\otimes\dots\otimes\Gamma_{i_\ell})$ vanishes by the fundamental class axiom \eqref{Open Fundamental Class}.
			\item  {$|\Gamma_{i_\ell}| = 2$}: In this case, by the divisor axiom \eqref{Open Divisor}, we have
			\begin{equation*}
				\OGWbar_{\beta,k}^{\parentheses{\ell}}(\Gamma_{i_1}\otimes\dots\otimes\Gamma_{i_\ell}) = \int_\beta\Gamma_{i_\ell}\cdot\OGWbar_{\beta,k}^{\parentheses{\ell-1}}(\Gamma_{i_1}\otimes\dots\otimes\Gamma_{i_{\ell-1}}).
			\end{equation*}
			The coefficient $\OGWbar_{\beta,k}^{\parentheses{\ell-1}}(\Gamma_{i_1}\otimes\dots\otimes\Gamma_{i_{\ell-1}})$ is of lower order and of nonzero degree, and therefore vanishes. We deduce that the coefficient $\OGWbar_{\beta,k}^{\parentheses{\ell}}(\Gamma_{i_1}\otimes\dots\otimes\Gamma_{i_\ell})$ also vanishes.
			\item  {$|\Gamma_{i_\ell}| > 2$}: In this case, following the proof of Lemma \ref{All OGW are Computable lemma pt. 2}, we get the relation \eqref{summary of all OGW are computable lemma pt. 2}, where the right-hand side does not contain any terms of order $\left(\omega(\beta),k,\ell,\left|\Gamma_{i_1}\right|,\dots,\left|\Gamma_{i_\ell}\right|\right)$ or higher, nor any terms with degree $\beta_0$. We deduce by the assumptions of the lemma that $\OGWbar_{\beta,k}^{\parentheses{\ell}}(\Gamma_{i_1}\otimes\dots\otimes\Gamma_{i_\ell})$ vanishes.

		\end{itemize}
		The result follows.
	\end{proof}
	
	Before proving the lemma that makes up the induction step for invariants with one interior constraint, we need to prove a preliminary fact.
	
	\begin{Lemma}\label{g^{munu} where nu = 1 has mu = fundamental class}
		Let $1\leq\nu\leq K$ such that $\Gamma_\nu = 1$. Then $g^{\mu\nu} = 0$ for all $1\leq\mu\leq K$ such that $\left|\Gamma_\mu\right|\neq 2n$.
	\end{Lemma}
	
	\begin{proof}
		Let $1\leq\mu\leq K$ such that $\left|\Gamma_\mu\right|\neq 2n$. let $G_{\mu\nu}$ denote the $\mu,\nu$-minor matrix of $\left(g_{ij}\right)$, given by deleting the $\mu$-th row and $\nu$-th column of $\left(g_{ij}\right)$. By properties of the adjugate matrix to $\left(g_{ij}\right)$, we get
		\begin{equation*}
			g^{\mu\nu} = \frac{1}{\det(g)}\cdot(-1)^{\mu+\nu}\cdot\det\left(G_{\nu\mu}\right).
		\end{equation*}
		Now, $X$ is orientable, so there exists a unique $1\leq\mu'\leq K$ such that $\left|\Gamma_{\mu'}\right| = 2n$. By Poincar\'{e} duality,
		\[ g_{\mu'\nu'} = \int_X\Delta_{\mu'}\wedge\Delta_{\nu'} = 0 \]
		for all $1\leq\nu'\leq K$ such that $\nu'\neq\nu$. We deduce that the $\mu'$-th column of the minor matrix $G_{\nu\mu}$ is a column of zeros, and therefore $\det\left(G_{\nu\mu}\right)$ vanishes. Therefore $g^{\mu\nu}=0$.
	\end{proof}
	
	In the following proof, we utilize a fact about (closed) Gromov-Witten classes in products of symplectic manifolds. Let us first recall the definition of Gromov-Witten classes.
	\begin{Definition}
		Let $(X,\omega)$ be a symplectic manifold with an $\omega$-tame almost complex structure $J$. Let $\beta\in H_2(X;\ZZ)$, $\ell\geq 0$.
		Denote by $\overline{\mathcal{M}}_{0,\ell}(\beta)$
		the Gromov compactification of the moduli space of $J$-holomorphic curves of genus $0$ and degree $\beta$ with $\ell$ marked points. Denote by
		\[
		ev_i:\overline{\mathcal{M}}_{0,\ell}(\beta)\rightarrow X
		\]
		the evaluation map at the $i$'th marked point, for $1\leq i\leq \ell$.
		Denote by $\overline{\mathcal{M}}_{0,\ell}$ the space of stable curves of genus $0$ with $\ell$ marked points. Let
		\[
		\pi:\overline{\mathcal{M}}_{0,\ell}(\beta)\rightarrow\overline{\mathcal{M}}_{0,\ell}
		\]
		be the forgetful map.
		The \textit{Gromov-Witten class} $I_{\beta,\ell}^X$ is a mapping
		\[
		I_{\beta,\ell}^X:H^*(X;\RR)^{\otimes \ell}\rightarrow H^*(\overline{\mathcal{M}}_{0,\ell};\RR),
		\]
		defined by
		\[ I_{\beta,\ell}^X(\gamma_1\otimes\dots\otimes\gamma_\ell) =\pi_*(ev_1^*\gamma_1\wedge\dots\wedge ev_n^*\gamma_\ell)\in H^*( \overline{\mathcal{M}}_{0,\ell};\RR) \]
		for any $\gamma_1,\dots,\gamma_\ell\in H^*(X;\RR)$. These classes are related to the closed Gromov-Witten invariants via
		\[ \GW_{\beta}\left(\gamma_1\otimes\dots\otimes\gamma_\ell\right) = \int_{\overline{\mathcal{M}}_{0,\ell}}I_{\beta,\ell}^X\left(\gamma_1\otimes\dots\otimes\gamma_\ell\right), \]
		whenever the integral over $\overline{\mathcal{M}}_{0,\ell}$ is well-defined (for example, whenever $\overline{\mathcal{M}}_{0,\ell}$ is a smooth orbifold, as is the case with projective spaces~\cite{RobbinSalamon}).
	\end{Definition}
	For a fuller discussion of Gromov-Witten classes, see \cite{CoxKatz}. We now cite a theorem, proven in greater generality in \cite[Theorem 9]{behrend1997product}, that gives a formula for Gromov-Witten classes of products of manifolds.
	\begin{Theorem}\label{behrend products}
		For any $\beta\in H_2(X;\ZZ)$ and $\gamma_{1,i},\dots,\gamma_{\ell,i}\in H^*(\CC P^{n_i};\RR)$ for $i\in\{1,2\}$, we have
		\begin{multline*}
			I_{\beta,\ell}^{\CC P^{n_1}\times \CC P^{n_2}}((\pi_1^*\gamma_{1,1}\wedge \pi_2^*\gamma_{1,2})\otimes\dots\otimes(\pi_1^*\gamma_{\ell,1}\wedge \pi_2^*\gamma_{\ell,2}))\\
			= (-1)^\sigma I_{\pi_{1*}\beta,\ell}^{\CC P^{n_1}}(\gamma_{1,1}\otimes\dots\otimes\gamma_{\ell,1})\wedge I_{\pi_{2*}\beta,\ell}^{\CC P^{n_2}}(\gamma_{1,2}\otimes\dots\otimes\gamma_{\ell,2}),
		\end{multline*}
		where $\pi_1:\CC P^{n_1}\times \CC P^{n_2}\rightarrow \CC P^{n_1}$ and $\pi_2:\CC P^{n_1}\times \CC P^{n_2}\rightarrow \CC P^{n_2}$ are the projections and
		$\sigma = \sum\limits_{i>j}|\gamma_{i,1}|\cdot|\gamma_{j,2}|.$
		In particular, the corresponding Gromov-Witten invariant is
		\begin{multline*}
			\GW_{\beta}((\pi_1^*\gamma_{1,1}\wedge \pi_2^*\gamma_{1,2})\otimes\dots\otimes(\pi_1^*\gamma_{\ell,1}\wedge \pi_2^*\gamma_{\ell,2}))\\
			= (-1)^\sigma \int_{\overline{\mathcal{M}}_{0,\ell}}
			I_{\pi_{1*}\beta,\ell}^{\CC P^{n_1}}(\gamma_{1,1}\otimes\dots\otimes\gamma_{\ell,1})\wedge I_{\pi_{2*}\beta,\ell}^{\CC P ^{n_2}}(\gamma_{1,2}\otimes\dots\otimes\gamma_{\ell,2}).
		\end{multline*}
	\end{Theorem}
	
	We are now ready to prove the next inductive step.
	\begin{Lemma}
		\label{Everything is zero OMG Lemma pt. 2, 2}
		In the setup of Proposition~\ref{OGWs Vanish for n1 n2 are 5mod6} or Proposition~\ref{OGWs Vanish for n1 n2 are 1mod2 and pointlike},
		let $\beta\in\Pi_{\geq 0}$,  $k\in\ZZ_{\geq 0}$, and $0\leq i\leq K$, and assume that $\beta\neq\beta_0$ and $k\geq 1$. Assume that for any $\beta'\in\Pi_{\geq 0}$,  $k',\ell'\in\ZZ_{\geq 0}$, and $0\leq i'_1,\dots,i'_\ell\leq K$, where $\beta\neq\beta_0$ and $\left(\omega(\beta'),k',\ell',\left|\Gamma_{i'_1}\right|,\dots,\left|\Gamma_{i'_\ell}\right|\right)<\left(\omega(\beta),k,1,\left|\Gamma_{i}\right|\right)$, the coefficient $\OGWbar_{\beta',k'}^{\parentheses{\ell'}}(\Gamma_{i'_1}\otimes\dots\otimes\Gamma_{i'_\ell})$ vanishes. Then the coefficient $\OGWbar_{\beta,k}^{\parentheses{1}}(\Gamma_{i})$ also vanishes.
	\end{Lemma}
	
	\begin{proof}
		We go over the proof of Lemma~\ref{All OGW are Computable lemma pt. 2, 2}, and show that $\OGWbar_{\beta,k}^{\parentheses{1}}(\Gamma_{i})$ vanishes under our assumptions. There are several possible cases:
		\begin{itemize}
			\item  {$|\Gamma_{i}| = 0$:} In this case, $\Gamma_{i}=1$. Since we assume that $\beta\neq\beta_0$, we deduce that $\OGWbar_{\beta,k}^{\parentheses{1}}(\Gamma_{i})$ vanishes due to the fundamental class axiom \eqref{Open Fundamental Class}.
			\item  {$|\Gamma_{i}| = 2$}: In this case, by the divisor axiom \eqref{Open Divisor}, we have
			\begin{equation*}
				\OGWbar_{\beta,k}^{\parentheses{1}}(\Gamma_{i}) = \int_\beta\Gamma_{i}\cdot\OGWbar_{\beta,k}^{\parentheses{0}}.
			\end{equation*}
			The coefficient $\OGWbar_{\beta,k}^{\parentheses{0}}$ is of lower order and of nonzero degree, and therefore vanishes. We deduce, then, that the coefficient $\OGWbar_{\beta,k}^{\parentheses{1}}(\Gamma_{i})$ also vanishes.
			\item  {$|\Gamma_{i_\ell}| > 2$}: In this case, following our proof of Lemma \ref{All OGW are Computable lemma pt. 2, 2}, we get the relation \eqref{summary of all OGW are computable lemma pt. 2, 2}, where the right-hand side does not contain any terms of order $\left(\omega(\beta),k,1,\left|\Gamma_{i}\right|\right)$ or higher, and the only terms that involve degree $\beta_0$ are of the form
			
			\begin{align}
				\label{Open WDVV For Gamma_a=delta', Gamma_b=delta 2, 1 OMG Vanishing}&\sum\limits_{0 \leq\mu,\nu\leq \kappa} \GW_{\tilde{\beta}}\big(\Delta_a\otimes\Delta_b\otimes\Delta_\mu\big)g^{\mu\nu}\OGWbar_{\beta_0,k}^{\parentheses{1}}\big(\Gamma_\nu\big),
			\end{align}		
			for some $\tilde{\beta}\in H_2(X;\ZZ)$ that satisfies $\varpi(\tilde{\beta})=\beta$. In particular, $\tilde{\beta}\neq(0,0)$. If $\OGWbar_{\beta_0,k}^{\parentheses{1}}\big(\Gamma_\nu\big)$ is nonzero, the zero axiom \eqref{Open Zero} implies that $k=1$ and that $\Gamma_\nu = 1$. By Lemma \ref{g^{munu} where nu = 1 has mu = fundamental class}, we get $\left|\Gamma_\mu\right| = 2n$. If $\GW_{\tilde{\beta}}\big(\Delta_a\otimes\Delta_b\otimes\Delta_\mu\big)$ does not vanish, the degree axiom \eqref{Open Degree} implies that
			\begin{equation*}
				2n-6 + 2c_1(\tilde{\beta}) + 6 = \left|\Delta_a\right| + \left|\Delta_b\right| + \left|\Delta_\mu\right|.
			\end{equation*}
			Since $\left|\Delta_a\right| + \left|\Delta_b\right| = \left|\Delta_{i}\right|$ and $\left|\Delta_\mu\right| = 2n$, we get
			\begin{equation*}
				2n + 2c_1(\tilde{\beta}) = 2n + \left|\Delta_{i}\right|.
			\end{equation*}
			So,
			\begin{equation}\label{closed degree in OMG vanishing}
				2c_1(\tilde{\beta}) = \left|\Delta_{i}\right|.
			\end{equation}
			Write $\tilde{\beta} = (\tilde{a},\tilde{b})$ for some $\tilde{a},\tilde{b}\in\ZZ_{\geq 0}$. The first Chern class is
			\[ c_1(\tilde{a},\tilde{b}) = (n_1+1)\tilde{a}+(n_2+1)\tilde{b}. \]
			It follows that
			\[
			2(n_1+1)\tilde{a}+2(n_2+1)\tilde{b} = \left|\Delta_{i}\right|\leq 2n = 2(n_1+n_2).
			\]
			Therefore, $\tilde{a}$ and $\tilde{b}$ cannot simultaneously be positive.
			
			Assume without loss of generality that $\tilde{b}=0$ and $\tilde{a}>0$. By Theorem~\ref{behrend products}, we know that
			\begin{align*}
				\GW_{(\tilde{a},0)}\left(\Delta_a\otimes\Delta_b\otimes\Delta_\mu\right) = \int_{\overline{\mathcal{M}}_{0, 3}} I_{\tilde{a}}^{\CC P^{n_1}}\left(\alpha_1\otimes\alpha_2\otimes \omega_1^{n_1}\right)\wedge I_{0}^{\CC P^{n_2}}\left(\gamma_1\otimes\gamma_2\otimes \omega_2^{n_2}\right),
			\end{align*}
			for $\alpha_1,\alpha_2\in H^*(\CC P^{n_1};\RR)$ and $\gamma_1,\gamma_2\in H^*(\CC P^{n_2};\RR)$ such that $\pi_1^*\alpha_1\wedge\pi_2^*\gamma_1 = \Delta_a$ and $\pi_1^*\alpha_2\wedge\pi_2^*\gamma_2 = \Delta_b$.
			
			The class $I_{0}^{\CC P^{n_2}}\left(\gamma_1\otimes\gamma_2\otimes \omega_2^{n_2}\right)$ is a top form in $\overline{\mathcal{M}}_{0, 3}=\{pt\}$, and it is therefore a real number. The closed zero axiom \eqref{Closed Zero} tells us that
			\[
			I_{0}^{\CC P^{n_2}}\left(\gamma_1\otimes\gamma_2\otimes \omega_2^{n_2}\right)
			= \GW_{0}\left(\gamma_1\otimes\gamma_2\otimes \omega_2^{n_2}\right)
			= \int_{\CC P^{n_2}}\gamma_1\wedge\gamma_2\wedge \omega_2^{n_2}.
			\]
			This term is nonzero, so $\gamma_1 = \gamma_2 = 1$. Therefore, $\pi_1^*\alpha_1 = \Delta_a$ and $\pi_1^*\alpha_2 = \Delta_b$. In particular,
			\[  \Delta_i = \Delta_a\wedge\Delta_b = \pi_1^*\left(\alpha_1\wedge\alpha_2\right)\in \pi_1^*\left(H^*(\CC P^{n_1};\RR)\right), \]
			and therefore $\left|\Delta_i\right|\leq\dim_\RR\CC P^{n_1} = 2n_1$. Equation \eqref{closed degree in OMG vanishing} implies
			\[ 2(n_1+1)\tilde{a} = \left|\Delta_i\right| \leq 2n_1, \]
			and therefore $\tilde{a} = 0$, contradicting our assumption. We deduce that the term \eqref{Open WDVV For Gamma_a=delta', Gamma_b=delta 2, 1 OMG Vanishing} must vanish.
			
			In total, we see that there are no nonvanishing terms in the right-hand side of \eqref{summary of all OGW are computable lemma pt. 2, 2} with a coefficient of degree $\beta_0$. Therefore, $\OGWbar_{\beta,k}^{\parentheses{1}}(\Gamma_{i})$ vanishes by our assumption.
		\end{itemize}
	\end{proof}
	
	Combining Lemmas \ref{Initial values vanish for n1 n2 5mod6}, \ref{Initial values Vanish OMG}, \ref{Everything is zero OMG Lemma pt. 2}, and \ref{Everything is zero OMG Lemma pt. 2, 2}, we now proceed to prove Propositions~\ref{OGWs Vanish for n1 n2 are 5mod6} and~\ref{OGWs Vanish for n1 n2 are 1mod2 and pointlike}.
	
	\begin{proof}[Proof of Propositions~\ref{OGWs Vanish for n1 n2 are 5mod6} and~\ref{OGWs Vanish for n1 n2 are 1mod2 and pointlike}]
		Let $\OGWbar_{\beta, k}^{(\ell)}\left(\gamma_{1}\otimes\dots\otimes\gamma_{\ell}\right)$ be as in the statement of Proposition~\ref{OGWs Vanish for n1 n2 are 5mod6} or~\ref{OGWs Vanish for n1 n2 are 1mod2 and pointlike}. By linearity, we may assume without loss of generality that there exist $1\leq i_1,\dots,i_\ell\leq N$ such that $\gamma_j = \Gamma_{i_j}$ for all $1\leq j\leq \ell$. By repeated application of the wall crossing formula \eqref{Wall Crossing}, we may assume without loss of generality that $1\leq i_1,\dots, i_\ell\leq K$. By the symmetry axiom \eqref{Open Symmetry}, we may assume without loss of generality that $\left|\Gamma_{i_1}\right|\geq\dots\geq\left|\Gamma_{i_\ell}\right|$.
		
		We now proceed via induction on the number of computational steps needed to compute $\OGWbar_{\beta, k}^{(\ell)}\left(\Gamma_{i_1}\otimes\dots\otimes\Gamma_{i_\ell}\right)$, via the algorithm prescribed in the proof of Theorem \ref{All OGW are computable 1 Fellowship of the OGW}:
		\begin{itemize}
			\item {Basis of induction:} If $\OGWbar_{\beta, k}^{(\ell)}\left(\Gamma_{i_1}\otimes\dots\otimes\Gamma_{i_\ell}\right)$ requires $0$ steps to compute according to Theorem~\ref{All OGW are computable 1 Fellowship of the OGW}, then it is an initial value, as prescribed by Theorem \ref{All OGW are computable 1 Fellowship of the OGW}. In the setup of Proposition~\ref{OGWs Vanish for n1 n2 are 5mod6} (resp. Proposition~\ref{OGWs Vanish for n1 n2 are 1mod2 and pointlike}), the invariant $\OGWbar_{\beta, k}^{(\ell)}\left(\Gamma_{i_1}\otimes\dots\otimes\Gamma_{i_\ell}\right)$ vanishes by Lemma~\ref{Initial values vanish for n1 n2 5mod6} (resp. Lemma~\ref{Initial values Vanish OMG}).
			\item {Induction step:} Assume that all open Gromov-Witten invariants that require a smaller number of steps to compute than that of $\OGWbar_{\beta, k}^{(\ell)}\left(\Gamma_{i_1}\otimes\dots\otimes\Gamma_{i_\ell}\right)$ vanish, except perhaps the invariants of degree $\beta_0$. Then Lemmas \ref{Everything is zero OMG Lemma pt. 2} and \ref{Everything is zero OMG Lemma pt. 2, 2} imply that the invariant $\OGWbar_{\beta, k}^{(\ell)}\left(\Gamma_{i_1}\otimes\dots\otimes\Gamma_{i_\ell}\right)$ vanishes.
		\end{itemize}
		This completes the proof.
	\end{proof}

	\bibliographystyle{../../amsabbrvcnobysame.bst}
	\bibliography{../../mybibfile.bib}

	
\end{document}